%% file: main.tex
\documentclass[11pt]{amsart}
\usepackage[utf8]{inputenc}
\usepackage{amssymb,amsmath,amsthm,mathrsfs}
\usepackage{xypic}
\usepackage{comment}
\usepackage{hyperref}
\usepackage{enumerate}
\usepackage{color,graphicx,lpic}
\definecolor{darkgreen}{RGB}{47,139,79}
\definecolor{darkblue}{RGB}{36,24,130}
\usepackage{hyperref}
\hypersetup{
    colorlinks=true, 
    linktoc=all,     
    linktocpage,
    citecolor=darkgreen,
    filecolor=black,
    linkcolor=darkblue,
    urlcolor=black
}

\let\oldtocsection=\tocsection
\let\oldtocsubsection=\tocsubsection
\renewcommand{\tocsection}[2]{\hspace{0em}\oldtocsection{#1}{#2}}
\renewcommand{\tocsubsection}[2]{\hspace{1em}\oldtocsubsection{#1}{#2}}
\DeclareRobustCommand{\SkipTocEntry}[5]{}
\newcommand{\nocontentsline}[3]{}
\let\origcontentsline\addcontentsline
\newcommand\stoptoc{\let\addcontentsline\nocontentsline}
\newcommand\resumetoc{\let\addcontentsline\origcontentsline}

\title{Homological stability  for symplectic groups via algebraic arc complexes}
\author{Ismael Sierra}
\author{Nathalie Wahl}
\date{\today}

\newcommand{\Rad} {\operatorname{Rad}}

\newcommand{\im} {\operatorname{Im}}

\newcommand{\Aut} {\operatorname{Aut}}
\newcommand{\Diff} {\operatorname{Diff}}

\newcommand{\Aa}{\mathcal{A}}

\newcommand{\Dr}{csr}

\newcommand{\B}{\mathcal{B}}

\newcommand{\D}{\mathcal{D}}
\newcommand{\Hd}{\operatorname{F}}

\newcommand{\bbN} {\ensuremath{\mathbb{N}}}

\newcommand{\F}{\mathcal{F}}
\newcommand{\Fd}{\mathcal{F}_\del}

\newcommand{\FdX}{\mathcal{F}_{\del,X}}
\newcommand{\M}{\mathcal{M}}
\newcommand{\tM}{\mathbf{r}}
\newcommand{\del}{\partial}

\usepackage{tikz, tikz-cd}

\newcommand{\hash}{\mathbin{\text{\normalfont \texttt{\#}}}}

\DeclareMathAlphabet\mathbfcal{OMS}{cmsy}{b}{n}

\DeclareMathOperator\id{id}

\DeclareMathOperator\rk{rk}

\DeclareMathOperator\link{Link}

\newcommand{\Ga}{\Gamma}
\DeclareMathOperator\HH{\mathcal{H}}

\newcommand{\GL}{\operatorname{GL}}

\newcommand{\Sp}{\operatorname{Sp}}

\newcommand{\Uu}{\mathbb{U}}
\newcommand{\X}{\mathcal{U}}
\newcommand{\Z}{\mathbb{Z}}
\newcommand{\St}{\operatorname{Star}}
\newcommand{\cS}{\mathcal{S}}

\newcommand{\DD}{\mathbb{D}}

\theoremstyle{definition}
\newtheorem{thm}{Theorem}[section]
\newtheorem{prop}[thm]{Proposition}
\newtheorem*{prop*}{Proposition}
\newtheorem{lemma}[thm]{Lemma}
\newtheorem{lem}[thm]{Lemma}
\newtheorem{cor}[thm]{Corollary}

\newtheorem*{claim*}{Claim}
\newtheorem{rem}[thm]{Remark}
\newtheorem*{rem*}{Remark}
\newtheorem{defn}[thm]{Definition}
\newtheorem{Def}[thm]{Definition}
\newtheorem{ex}[thm]{Example}
\numberwithin{equation}{section}

\newtheorem{Th}{Theorem}

\newcommand{\al}{\alpha}
\newcommand{\la}{\lambda}
\newcommand{\s}{\sigma}

\newcommand{\x}{\times}
\newcommand{\minus}{\backslash}

\newcommand{\inc}{\hookrightarrow}

\newcommand{\rar}{\longrightarrow}

\newcommand{\note}[1]
{{\bf [N: #1]}}
\newcommand{\Inote}[1]
{{\bf [I: #1]}}

\begin{document}

\begin{abstract}
    We use algebraic arc complexes to prove a homological stability result for symplectic groups with  slope $\frac{2}{3}$ for rings with finite unitary stable rank. Symplectic groups are here interpreted as the automorphism groups of {\em formed spaces with boundary}, which are algebraic analogues of surfaces with boundary, that we also study in the present paper. Our stabilization map is a rank one stabilization in the category of formed spaces with boundary, going through both odd and even symplectic groups. 
\end{abstract}

\maketitle

\section{Introduction}

\stoptoc

Let $R$ be a commutative ring. The symplectic group $\Sp_{2n}(R)$ is the automorphism group of the hyperbolic space $\HH^{\oplus n}$, where $\HH=(R^2,\la_{\HH})$ is $R^2$ equipped with the non-degenerate alternating form $\la_{\HH}=\begin{pmatrix}
    \ 0 & 1 \\ -1 & 0
\end{pmatrix}$. 
A pair $(M,\la)$, with $M$ an $R$-module and $\la$ an alternating form, is called a formed space, and one has traditionally studied the stability properties of symplectic groups as coming from stabilizing by direct sum with copies of $\HH$ in the category of formed spaces. 

In the present paper, we improve the best known stability range for the homology of symplectic groups in the case of rings $R$ with finite  unitary stable rank $usr(R)$ (see Definition \ref{def:sr}), by working instead in the new category $\Fd$ of {\em formed spaces with boundary} $(M,\la,\del)$. Here $\del:M\to R$ is the additional data of a linear map, which we think of as a boundary. We then  replace the rank two stabilization map $\oplus \HH$ by a rank one stabilization $\# X$ with the object $X=(R,0,\id)$ of $\Fd$, where $\#$ is a monoidal structure on $\Fd$, with the property that both $X^{\# 2g}$ and $X^{\# 2g+1}$ are closely related to $\HH^g$. More specifically, there is an identification $\Sp_{2n}(R)\cong \Aut(X^{\# 2n+1})$,   
while the intermediate \textit{odd symplectic group} $\Sp_{2n-1}(R):=\Aut(X^{\# 2n})$ identifies with a parabolic subgroup in $\Sp_{2n}(R)$. These particular odd symplectic groups have appeared elsewhere before e.g.; in \cite{AST_1985__S131__117_0} or in \cite{SarSch,Sch22} in the context of homological stability, but maybe defined in a more ad-hoc manner.

\medskip

Our main result is the following:

\begin{Th}\label{thm:A}
Let $R$ be a commutative ring with $usr(R)<\infty$ and set $c=0$ if $R$ is a PID, and $c=2 usr(R)+2$ otherwise. Then the map 
    \[H_i(\Sp_n(R);\Z)\to H_{i}(\Sp_{n+1}(R);\Z)\]
    is an epimorphism for $i\le \frac{n-c}{3}$, and a monomorphism for $i\le \frac{n-c-3}{3}$ ($n$ odd) and for all $i$ ($n$ is even).  
    In particular, restricting to even symplectic groups we get that \[H_i(\Sp_{2g}(R);\Z)\to H_{i}(\Sp_{2g+2}(R);\Z)\]
    is an epimorphism for $i \le \frac{2g-c}{3}$ and an isomorphism for $i\le \frac{2g-c-2}{3}$. 
\end{Th}

The above slope $2/3$ stability result improves the earlier results of Charney \cite[Cor 4.5]{Cha87} and Mirzaii-van der Kallen \cite[Thm 8.2]{MirzaiivdK}, who gave a stability result for even symplectic group with a slope $1/2$ for respectively Dedekind domains and rings with finite unitary stable rank. The range given here is however not optimal for finite fields, where stability is known to hold with the better slope $1$ instead by Sprehn and the second author \cite{SprWah} (for fields other than $\mathbb{F}_2$), and for local rings with infinite residue field, where a slope $2$ result was recently proved by Schlichting \cite{Sch22}. 
Rationally in the case $R=\Z$, the stability slope is known to be exactly $1$ by \cite{tshishiku2019borels}.

\subsection*{Finite degree coefficients}
As is often the case for homological stability results, stability also holds for finite degree coefficient systems. A sequence of compatible $\Sp_n(R)$--representations $M_n$ is here called a finite degree coefficient system if the maps $M_n\to M_{n+1}$ are injective with trivial iterated cokernels, and a certain braid condition holds, see Definitions~\ref{def:coeff} and~\ref{def:findeg}.

\begin{Th}\label{thm:B} Let $R$ be a commutative ring with $usr(R)<\infty$ and set $c=0$ if $R$ is a PID, and $c=2usr(R)+2$ otherwise. Let $M_n$ be a coefficient system of degree $r$. Then the map 
    $$H_i(\Sp_n(R);M_n)\to H_{i}(\Sp_{n+1}(R);M_{n+1})$$ 
     is an epimorphism for $i\le \frac{n-c-3r+1}{3}$ and an isomorphism for $i\le \frac{n-c-3r-2}{3}$. 

\end{Th}
Stability results for even symplectic groups with finite degree coefficient systems appear in \cite[Thm 5.15-16]{RWW} and \cite[Thm 3.25]{Fri17}, with a stability slope $1/2$ (corresponding to $1/4$ in terms of the rank $n$). One should though note that, while examples tend to fit in both our framework and the framework of \cite{RWW,Fri17}, the definition of finite degree coefficients in those papers actually differs from ours. See Remark~\ref{rem:twisted} for more details.  

\subsection*{Abelian coefficients}
Theorem~\ref{thm:stabab} states that stability also holds with {\em abelian coefficients} $M$, that is those systems of coefficients coming from an action of $H_1(\Sp_n(R);\Z)\cong H_1(\Sp_\infty(R);\Z)$ on a fixed module $M$, 
with essentially the same bounds as in Theorem~\ref{thm:A}, improving again on earlier results of  \cite{RWW,Fri17}. However this is only relevant when $H_1(\Sp_n(R);\Z)$ is non-zero. In particular it is not relevant for $R=\Z$ by \cite[Lem A1]{BCRR}.

\subsection*{Geometric interpretation}

When $R=\Z$,  
the symplectic group $\Sp_{2g}(\Z)$ is also the automorphism group of the middle homology of a surface of genus $g$ that preserves the intersection form. In this case, the sequence of odd and even symplectic groups naturally fits in the following diagram, where $B_n$ denotes the braid group on $n$ elements and $\Ga_{g,b}=\pi_0\Diff(S_{g,b})$  the mapping class of a surface of genus $g$ with $b$ boundary components: 
\begin{align*}
\xymatrix{B_1 \ar[r]\ar[d]&B_2 \ar[r]\ar[d]& B_3 \ar[r]\ar[d] &\dots \ar[r] & B_{2g+1} \ar[r]\ar[d]& B_{2g+2} \ar[r]\ar[d]& \dots \\
\Ga_{0,1} \ar[r]\ar[d]& \Ga_{0,2} \ar[r]\ar[d]&\Ga_{1,1} \ar[r]\ar[d]&\dots \ar[r]& \Ga_{g,1} \ar[r]\ar[d]& \Ga_{g,2} \ar[r]\ar[d]& \dots \\
  \Sp_{0}(\Z)\ar[r] &\Sp_{1}(\Z)\ar[r]  &\Sp_{2}(\Z)\ar[r] & \dots \ar[r]& \Sp_{2g}(\Z)\ar[r] &\Sp_{2g+1}(\Z)\ar[r] & \dots}
\end{align*}
Here the map $B_{2g+i}\to \Ga_{g,i}$ identifies the braid group with the hyperelliptic mapping class group, the subgroup of $\Ga_{g,i}$ that preserves the hyperelliptic involution, while the map $\Ga_{g,i}\to \Sp_{2g+i-1}(\Z)$ comes from the action of diffeomorphisms on the middle homology of the surface, see Section~\ref{sec:surfaces}.

The top two sequences of groups were used in \cite{HVW} to give a short proof of the best known isomorphism range for the homology of the mapping class group of surfaces, where the ranges obtained precisely match those of Theorem~\ref{thm:A}. The proof in \cite{HVW} uses stabilization with a disc in a certain category of {\em bidecorated surfaces}, and the present paper can be seen as an algebraic version of this proof. As we will see in Section~\ref{sec:surfaces}, there is a monoidal functor from the category of bidecorated surfaces to that of formed spaces with boundary, and our stabilizing object $X=(R,0,\id)$ is the image of the bidecorated disc under that functor. 

Note that the subgroup of $\Sp_n(\Z)$ that is the image of the braid group under the above maps has since be studied in \cite{MPPRW}, where they show that it also stabilizes.

\subsection*{Orthogonal and unitary groups}
   It is natural to ask whether Theorems~\ref{thm:A} and~\ref{thm:B} also hold for orthogonal or unitary groups, that are often treated simultaneously to symplectic groups. Indeed, these groups also identify with the automorphism group of a version of $\HH^{\oplus n}$, in the context of formed spaces with a different flavour of {\em form parameters}, with alternating forms replaced by symmetric or Hermitian forms, see e.g.~\cite{SW0}. 
Our proof here however does not directly adapt to these other groups. 
In fact, a naive adaptation of the framework presented here to symmetric or Hermitian forms, stabilizing with the analogue of $X$ in such contexts, relates instead to a different family of groups, unrelated to stabilizing with hyperbolic summands. (See Remark \ref{other form parameters} for more details.)

\subsection*{Algebraic arc complexes and the proof of Theorems~\ref{thm:A} and~\ref{thm:B}}

To prove Theorems~\ref{thm:A} and~\ref{thm:B}, we use the general stabiltity machine of \cite{RWW,krannich}: we show that the stabilization maps considered come from the action of an $E_2$-module (the braided monoidal subcategory of $\Fd$ generated by $X$) on an $E_1$-algebra (the monoidal category $\Fd$). In fact the results hold more generally for the stabilization maps 
$$\Aut_{\Fd}(A\# X^{\# n})\rar \Aut_{\Fd}(A\# X^{\# n+1})$$
for $A=(M,\la,\del)$ any formed space with boundary. (See Theorems~\ref{thm:stabconst} and \ref{thm:stabtwist}.)

Stability follows if one can show that a certain {\em destabilization complex} is highly connected. Guided by the case of mapping class groups of surfaces, we identify this destabilization complex with an algebraic version of the disordered arc complex of \cite{HVW}. An {\em arc} in a formed space with boundary $(M,\la,\del)$ is defined here as an element $a\in M$ such that $\del(a)=1$. Building on \cite[Sec 6]{S22a}, we then define what it means for an arc to be {\em non-separating} and {\em disordered} in this algebraic context, see Definitions~\ref{def:arc} and \ref{def:disordered}. 

As already noted in \cite{S22a}, the connectivity of the non-separating arc complex can be deduced from that of a certain poset of unimodular vectors, see Proposition~\ref{prop:connectivity B} and Theorem~\ref{connectivity of U} that generalize earlier results of \cite{MirzaiivdK,Fri17,S22a}. We then deduce the connectivity of the disordered arc complex (Theorem~\ref{thm: connectivity D}) mimicking the geometric argument of \cite{HVW}. The proof uses the interplay between the {\em arc genus}, that is the number of $X$--summands that can be split off, and the {\em hyperbolic genus}, that is the number of $\HH$-summands one can split off. It requires most importantly having a good grasp on how the  arc genus behaves when ``cutting along a simplex of arcs".  The key properties of the ring used are the finite unitary stable rank, that allows cancellation of $\HH$--summands, and the fact that, for rings with finite Bass stable rank and $M$ with large enough hyperbolic genus, any $l:M\to R$ unimodular restricts to a unimodular map on some hyperbolic summand, see Proposition~\ref{prop: finiteness of usr*}. The stability slope $2/3$ comes from the fact that the arc genus behaves just like the geometric arc genus when cutting arcs, under our assumption on the rings, see Corollary~\ref{cor R of a kernel}. It was surprising to us that this geometrically inspired argument works for rather general rings; there are for example no restriction on the characteristic. All the way through, our results have slight improvements for PIDs because of their special proporties recorded in Lemma~\ref{lem: pid}.  

The identification between the disordered arc complex and the destabilization complex requires a weak cancellation property for formed spaces with boundary. We give in the paper two cancellation results, with very different proofs: one valid for general rings with finite unitary stable rank, see Theorem \ref{thm: general cancellation}, and a second result that works for PIDs and is deduced from a classification of formed spaces with boundary in that particular case, see Theorems~\ref{thm: classification} and \ref{thm: cancellation}.

\subsection*{Organization of the paper} In Section~\ref{sec: formedd} we define the category of formed spaces with boundaries, and study its properties. We define and relate the arc and hyperbolic genera. Section~\ref{sec: complexes} studies algebraic arc complexes and proves the high connectivity of the disordered arc complex. Section~\ref{sec:cancell} gives the equivalence (up a skeleton) between the disordered arc complex and the complex of destabilization. The results of Sections~\ref{sec: complexes} and~\ref{sec:cancell} are then combined in Section~\ref{sec:stability}, to prove the stability theorems.

\subsection*{Acknowledgements}
The authors would like to thank Jeremy Miller, Peter Patzt, Dan Petersen and most particularly Oscar Randal-Williams for discussions at various stages of this paper, as well as the anonimous referee for helpful comments and suggestions. 
The second author was supported by the Danish National Research Foundation through the Copenhagen Centre for Geometry and Topology (DNRF151), and the first author thanks the Mathematics Department of the University of Toronto, and in particular Alexander Kupers, for continuous support while this article was written. 

\resumetoc

\tableofcontents

\section{Formed spaces with boundary} \label{sec: formedd}

\input{sectionformsnew}

\section{Arc complexes} \label{sec: complexes}

\input{arccomplexes.tex}



\appendix

\section{Connectivity of the complexes of unimodular sequences}\label{appA}

\input{unimodular}

\section{Cancellation and classification of formed spaces with boundary For PIDs}\label{appB}

\input{cancellation}




\bibliographystyle{plain}
\bibliography{biblio}

\end{document}

%% file: sectionformsnew.tex
In this section, we  define a braided monoidal category $\Fd$ of {\em formed spaces with boundary}, that is an algebraic version of the category of bidecorated  surfaces of \cite{HVW}. 
A {\em bidecorated  surface} is an (oriented) surface $S$ equipped with two marked (oriented) intervals in its boundary. It has a monoidal structure induced by gluing intervals in pairs. The category of {\em formed spaces with boundary} defined here will be the algebraic shadow of the category of bidecorated  surfaces, with objects behaving like the first homology of the surface relative to the two intervals, with an appropriately defined intersection form, and a ``boundary map" corresponding to the boundary map in the long exact sequence of homology groups. We will also construct a monoidal structure on $\Fd$ that corresponds to the one of bidecorated  surfaces.

We start by recalling what formed spaces are, in the context relevant to symplectic groups, that is where the forms are alternating forms. 

\medskip

Let $R$ be a commutative ring. 

\begin{Def}
    A {\em formed space} is a pair $(M,\la)$ with $M$ a finitely generated free $R$--module and $\lambda:M\otimes M\to R$ an alternating form on $M$, i.e.~$\la$ is a $R$-bilinear map such that $\la(v,v)=0$ for all $v\in M$. We denote by $\F$ the category of formed spaces with morphisms the module maps that preserve the forms. 
    
    The category $\F$ is symmetric monoidal, with monoidal structure $\oplus$ induced by taking the direct sum of modules and orthogonal direct sum of forms. 
\end{Def}

The {\em hyperbolic form} is the formed space $\HH=(R^2,\la_{\HH})$ for $\la_{\HH}$ given in the standard basis by the matrix
$$\la_{\HH}=\begin{pmatrix}
    0&1\\ -1&0
\end{pmatrix}.$$

The $2g$-dimensional hyperbolic formed space $\HH^{\oplus g}$ is isomorphic to the formed space $(H_1(S_g),\la)$ with $S_g$ a closed surface of genus $g$ and $\la$ the intersection pairing on $S_g$, and the symplectic group $\Sp_{2g}(\Z)$ identifies with its automorphism group: 
$$\Sp_{2g}(R)\cong \Aut_{\F}(\HH^{\oplus g}), $$
see e.g.~Example 6.1(i) in \cite{MirzaiivdK}. 

\smallskip

We will work with the following enhanced version of the category $\F$: 

\begin{Def} The category $\Fd$ of {\em formed spaces with boundary} has objects triples 
$(M,\lambda,\del)$ where $(M,\la)$ is a formed space and $\del:M\to R$ a linear map. 
Morphisms in $\Fd$ are structure-preserving module maps. 
\end{Def}
The subcategory of formed spaces $(M,\lambda,\del)$ with $\del=0$ is isomorphic to the category $\F$ defined above. 
The case of particular importance to us is actually the opposite case, namely the triples $(M,\lambda,\del)$ where $\del$ is surjective instead.

  We define a  monoidal structure on the category $\Fd$ of such objects by setting
  $$(M_1,\lambda_1,\del_1)\#(M_2,\lambda_2,\del_2)=(M_1\oplus M_2,\lambda_1\#\la_2,\del_1+\del_2)$$
  where $$\la_1\#\la_2=\begin{pmatrix}\la_1& \del_1^T\del_2\\ -\del_2^T\del_1& \la_2\end{pmatrix}.$$
  The  monoidal unit is $0:=(0,0,0)$. 
  On morphisms the monoidal structure is given by direct sum. 
  One checks that this monoidal structure is strictly associative.

\begin{rem} While the above definition of the monoidal structure was guided by the geometric gluing of surfaces, it is also a natural choice if one thinks of it purely algebraically. Indeed, 
   an object in $\Fd$ is a triple $(M,\la,\del)$ with $\la \in \Lambda^2 M^\vee$ and $\del \in M^\vee$. 
   There are natural isomorphisms $(M_1 \oplus M_2)^\vee\cong M_1^\vee \oplus M_2^\vee$ and $\Lambda^2 (M_1 \oplus M_2)^\vee \cong \Lambda^2 M_1^\vee \oplus \Lambda^2 M_2^\vee \oplus M_1^\vee \otimes M_2^\vee$, making the choices 
    $\del_1 \# \del_2=\del_1\oplus\del_2$ and  $\lambda_1 \# \lambda_2=\lambda_1 \,\oplus\, \la_2 \,\oplus\, \del_1 \otimes \del_2$ with respect to the above canonical decomposition, ``most obvious" choices. These correspond to the choice given in the above definition. 
\end{rem}

Note that the functor $\F\to \Fd$ taking $(M,\la)$ to $(M,\la,0)$ is a monoidal functor. On the other hand, the forgetful functor $\Fd\to \F$ is not monoidal. We will see below that the  monoidal structure of $\Fd$ is chosen so that there is a monoidal functor from the category $\M_2$ of bidecorated  surfaces mentioned above, to $\Fd$.

\smallskip

Our main example of  a formed space with boundary (and also the minimal such object with $\del$ surjective)  is $$X=(R,0,\id).$$
It is an example than can be associated to the disc in the category of bidecorated  surfaces,  as we will explain in the following section.

\subsection{Formed spaces with boundary from bidecorated  surfaces}\label{sec:surfaces}

Recall from \cite{HVW} the category $\M_2$ of bidecorated  surfaces. It has objects triples $(S,I_0,I_1)$, where $S$ is an oriented surface and $I_0,I_1$ are two disjoint compatibly oriented intervals marked in its boundary, and morphisms are mapping classes fixing the marked intervals.  

To a bidecorated  surface $S=(S,I_0,I_1)$, we associate a new surface 
$S^+=S\cup H$ obtained from $S$ by attaching a handle $H=I\x I$ to $S$ along the identification of $\del I \times I$ with $I_0\sqcup I_1\inc \partial S$. (See Figure~\ref{fig:Splus}.) 
Since $(S,I_0,I_1)$ is oriented datum then $S^+$ is naturally oriented, so there is a natural intersection pairing on $H_1(S^+;R)$ induced, under the Poincaré duality isomorphism $H_1(S^+;R) \cong H^1(S^+,\del S^+;R)$, by the (relative) cup product $H^1(S^+,\del S^+;R) \otimes H^1(S^+,\del S^+;R) \xrightarrow{\cup} H^2(S^+,\del S^+;R) \cong H_0(S^+;R)=R$. 
\begin{figure}
    \centering
    \includegraphics[width=0.6\textwidth]{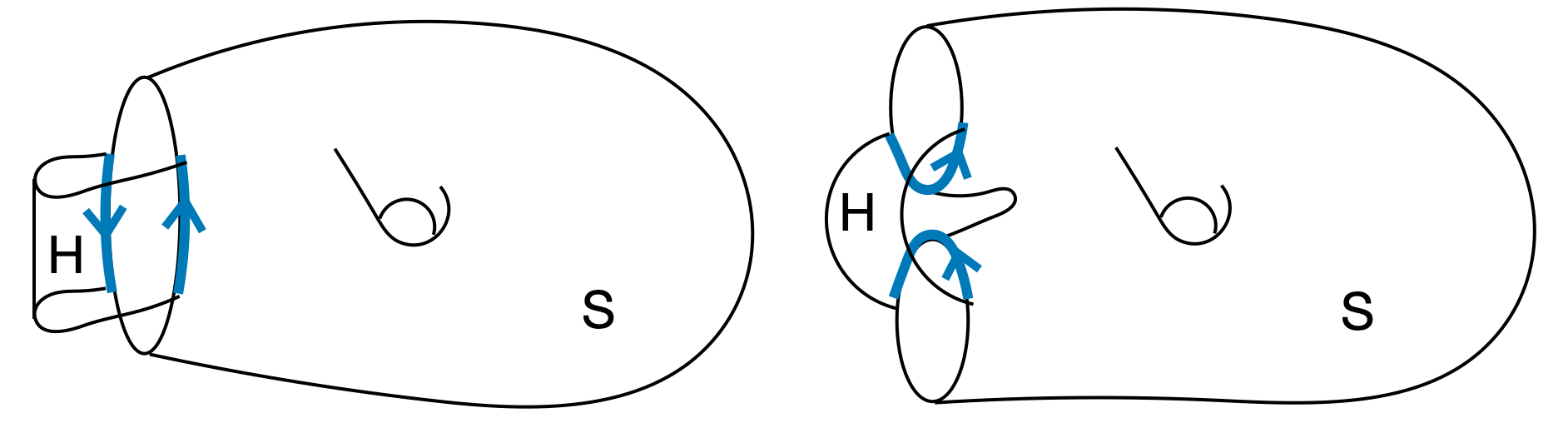}
    \caption{Two bidecorated  surfaces $(S,I_0,I_1)$ and their associated surface $S^+=S\cup H$}
    \label{fig:Splus}
\end{figure}
The long exact sequence in homology $H_*(-)=H_*(-;R)$ for the pair $(S^+,H)$ 
$$0=H_1(H)\rar H_1(S^+) \rar H_1(S^+,H)\cong H_1(S,I_0\sqcup I_1) \xrightarrow{0} H_0(H)\xrightarrow{\cong} H_0(S^+)$$
shows that there is an isomorphism $ H_1(S^+) \cong H_1(S,I_0\sqcup I_1)$. 
The inverse map $$M=H_1(S,I_0\sqcup I_1) \rar H_1(S^+)$$ is given explicitly by  representing classes in $M$ by collections of arcs and circles, and closing any arc component in $H_1(S,I_0\sqcup I_1)$ using the core $I\x \{\frac{1}{2}\}$ of the handle. 

Now we define a functor 
$$\Hd: \M_2 \rar \Fd$$
by 
$$\Hd(S,I_0,I_1):=(H_1(S,I_0\sqcup I_1),\la,\del)$$ where $\del:H_1(S,I_0\sqcup I_1)\to \tilde H_0(I_0\sqcup I_1)\cong R\langle b_1 - b_0\rangle$, for $b_0,b_1$ the midpoints of $I_0,I_1$, is the boundary map in the long exact sequence for reduced homology,  and $\la$ is the intersection pairing on $H_1(S^+)$, identified with $H_1(S,I_0\sqcup I_1)$ by the above isomorphism. 
On morphisms, the functor $F$ takes a mapping class to its induced map on relative homology.

\begin{rem}[Curved forms] \label{remark curved symmetry}
The above construction of the surface $S^+$ may seem a little ad hoc. 
It is possible to instead work directly with $(H_1(S,I_0\cup I_1),\del)$, replacing the alternating form $\la$  with a {\em curved form} $\la'$, satisfying the condition $\lambda'(x,y)+\lambda'(y,x)=2\del(x)\del(y)$. 
This ``curved'' symmetry equation arises indeed geometrically when defining an intersection product on $H_1(S,I_0\cup I_1)$: 
one can define an asymmetric intersection product $\la'(a,b)$ on classes $a,b$ with potentially non-trivial boundary, by choosing a representative of $a$ with boundary in $(0,\frac{1}{2})$ in $I_0$ and $I_1$ and for $b$ with boundary in $(\frac{1}{2},1)$ instead. 
One can check that the resulting form is not alternating on classes that have non-trivial boundary, and instead satisfies the above equation. 

There is an equivalence of categories between 
the category of alternating forms with boundary $(M,\la,\del)$ and that of curved forms with boundary $(M,\la',\del)$, where the correspondence is given by setting $\la'(x,y)=\la(x,y)+\del(x)\del(y)$. 
This form $\la'$ will play a role a few places in the paper, for example to define what it means for a formed space with boundary to be non-degnerate, see Section~\ref{sec:genus}.  
\end{rem}

Note that adding a handle to form the surface $S^+$ from the tuple $(S,I_0,I_1)$, interpreted as the association $(H_1(S,I_0\sqcup I_1),\la,\del) \mapsto (H_1(S^+),\la)$, becomes the forgetful functor $\Fd\to \F$. 
Note also that the formed space $(H_1(S,I_0\sqcup I_1),\la,\del)$ contains $(H_1(S),\la|_{H_1(S)},0)$ as a subspace, where $H_1(S)$ identifies with $\ker(\del)$, and where $\la|_{H_1(S)}$ is the standard intersection pairing of $H_1(S)$. This type of subspace will play an important role below. (See also Remark~\ref{rem:HorX}.)

\begin{ex}\label{ex:F(D)}
Consider the disc $D=(D^2,I_0,I_1)$ in the category $\M_2$. We have $$F(D)=\big(H_1(D^2,I_0 \sqcup I_1),\la,\del\big)=(R,0,\id)=X.$$
Explicitly,  
$H_1(D^2,I_0,I_1)\cong R$ generated by an arc $\rho$ from  $b_0$ to $b_1$ in the disc. The surface $(D^2)^+$ is a cylinder  and the generator if $H_1((D^2)^+)$ is represented by the middle circle in the cylinder. 
\end{ex}

\smallskip

Recall from \cite[Sec 3.1]{HVW} that the monoidal structure of $\M_2$ is defined for $S_1=(S_1,I^1_0,I^1_1)$ and $S_2=(S_2,I_0^2,I_1^2)$ by 
$$S_1\hash S_2=(S_1\cup S_2/\sim\,,\,I_0\,,\,I_1)$$
where, for $j=0,1$, the equivalence relation identifies the second half of $I_j^1$ with the first half of $I_j^2$, and the interval $I_j$ is the union of the first half of $I_j^1$ and the second half of $I_j^2$.

A particulary interesting example for us is the surfaces obtained as iterated sums of discs in this category: by \cite[Lem 3.1]{HVW}, there is an isomorphism 
$$D^{\hash 2g+i}\cong (S_{g,i},I_0,I_1),\ \  i\in\{1,2\} $$
to a surface of genus $g$ with one or two boundary components, where the marked intervals lie in distinct boundaries in the latter case. 

Taking homology, we have 
$$H_1(S_1\hash S_2,I_0\sqcup I_1)\cong H_1(S_1,I^1_0\sqcup I^1_1)\oplus H_1(S_2,I_0^2\sqcup I_1^2), $$
with boundary map $\del=\del_1+\del_2$. 
The following proposition shows that the intersection pairing of 
$H_1(S_1\hash S_2,I_0\sqcup I_1)$ identifies with the sum of the intersection pairings for $S_1$ and $S_2$, with the sum defined above in the category $\Fd$: 

  \begin{prop}\label{prop:Hdmonoidal}
The functor $\Hd:\M_2\to \Fd$ is monoidal. In particular, for $i\in \{1,2\}$,  $$\Hd(D^{\hash 2g+i})=X^{\# 2g+i}=(R^{2g+i},\la_{X^{\hash 2g+i}},\del_{X^{\hash 2g+i}})$$
where $X=(R,0,\id)=\Hd(D)$ (see Example~\ref{ex:F(D)}).   
    \end{prop}

This result will be useful to have a geometric intuition for the formed space $X^{\# n}$, and allow us to import results from the geometric side to the algebraic side. 

    \begin{proof}[Proof of Proposition~\ref{prop:Hdmonoidal}]

Let $S= S_1 \hash S_2$ and $(M_i,\la_i,\del_i)= \Hd(S_i)$. We need to show that $\Hd(S)=(M_1\oplus M_2,\la,\del_1+\del_2)$ with $\la= \la_1 \# \la_2$. 

We first describe $S^+$ in terms of $S_1^+$ and $S_2^+$. 
Write $H=H_\ell \cup_C H_r$ for $H=I\x I$, $H_\ell=I\x [0,\frac{1}{2}]$, $H_r=I\x [\frac{1}{2},1]$ and $C=H_\ell\cap H_r$ is the core of the handle. 
We have  $$(S_1 \hash S_2)^+ = (S_1\cup {H_\ell})\cup_C (S_2\cup H_r) \simeq  S_1^+\cup_{H^1_r\sim H^2_\ell} S_2^+,$$ where we denoted by $H^1$ and $H^2$ the handles  attached to $S_1$ and $S_2$.

Using Mayer-Vietoris on the right-hand-side we find that the inclusions $S_i^+ \subset S^+$ induce an isomorphism of abelian groups $H_1(S_1^+) \oplus H_1(S_2^+) \cong H_1(S^+)$, i.e. $M=M_1 \oplus M_2$. 
Moreover, naturality of the long exact sequence of the pair and of Mayer-Vietoris gives that $\del= \del_1+ \del_2$. 

Pick bases $(a_1,\dots,a_{m-1},\rho)$ of $H_1(S_1,I_0^1\cup I_1^1)$ and  $(a'_1,\dots,a'_{n-1},\rho')$ of $H_1(S_2,I_1^2\cup I_1^2)$ with $\rho,\rho'$ arcs and the classes $a_i$ and $a'_j$ having trivial boundary.  
Consider the corresponding bases $(a_1,\dots,a_{m-1},\rho_m\cup C)$ of $H_1(S_1^+)$ and  $(a'_1,\dots,a'_{n-1},\rho'_n\cup C)$ of $H_1(S_2^+)$.  In these bases, we see that only the last generators of each basis have a potential intersection. In particular, we see that the intersection pairing of the sum is correct on all the other generators.

Thus, to finish the proof we need to check that $\la(\rho\cup C,\rho'\cup C)=1$. 
This intersection happens purely in the handle $H$, where the arcs cross, due to the definition of the glued surface $S_1\hash S_2$,  
see Figure~\ref{fig:Hcross}. 
\begin{figure}
    \centering
    \includegraphics[width=0.6\textwidth]{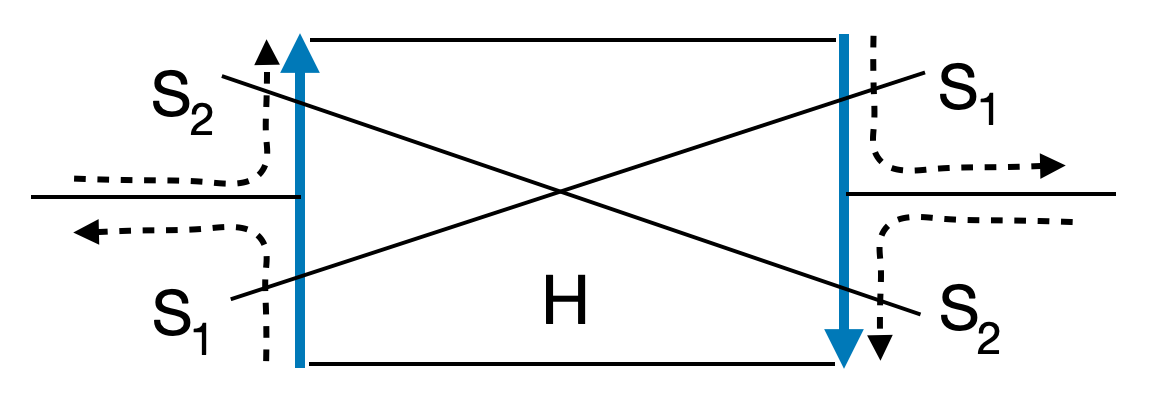}
    \caption{Transverse representatives of the extended arcs $\rho\cup C$ and $\rho'\cup C$ crossing in $H$}
    \label{fig:Hcross}
\end{figure}
      \end{proof}

\begin{rem}[Choices of bases for $X^{\# n}$ and their geometric description]
The underlying module of $X^{\# n}$ is $R^n$. 

As explained in Example~\ref{ex:F(D)}, $X=(R,0,\id)$ identifies with $F(D)$ and $1\in R$ corresponds to a generator of $H_1(D^2,I_0\sqcup I_1)$ which we pick as the unique isotopy class of arc in the disc going from $I_0$ to $I_1$. The standard basis of $X^{\# n}=(R^n,\la,\del)$ then corresponds to the collection of arcs in $D^{\hash n}$ from $I_0$ to $I_1$ going through each of the discs composing $D^{\hash n}$. These arcs are denoted $\rho_i$ in \cite{HVW}; see Section~4.1 of that paper for some of their properties. 

      An alternative geometric basis, more like that used in the proof above,  would be the curves $a_i=\rho_i*\rho_{i+1}^{-1}$ (or in homological notation $[\rho_i]-[\rho_{i+1}]$ for $i=1,\dots,n-1$), together with the arc $\rho_n$. The curves $a_i$ from a chain in the underlying surface of $D^{\hash n}$, that is $a_i$ and $a_{i+1}$ have intersection number 1 and curves further apart in the sequence do not intersect (see \cite[Lem 3.5]{HVW}). One then needs an additional arc, e.g.~$\rho_n$, to get a basis of $H_1(S,I_0\sqcup I_1)$. 

\end{rem}

\begin{rem}[Adding a handle $H$ versus summing with $X$]\label{rem:HorX}
    We can reinterpret the surface $S^+=S\cup H$ defined above as the underlying surface of $S\# D$ since adding the handle $H$ is essentially the same operation as summing with a disc in the category $\M_2$, with the only difference that we forget the marked intervals after gluing $H$. 
The  algebraic version of the identification $H_1(S,I_1 \sqcup I_2)\cong H_1(S^+)\cong H_1(S\# D)$, with their corresponding intersection forms, is the following: for any $(M,\la_M,\del_M)=(N,\la_N,\del_N) \# X$ one has $(\ker(\del_M),\la_M|_{\ker(\del_M)}) \cong (N,\la_N)$ as alternating forms. 
This fact will play a role in Section~\ref{sec: complexes}, in the proof of Lemma~\ref{lem:cut form properties}, when we study the formed space in the complement of a split $X$-summand. It will also be proved at that point in the more general algebraic context. 
\end{rem}

  \subsection{Braided monoidal structure} 

We will  formulate our main stability result as stability of the automorphism groups of objects of $\Fd$ with respect to the stabilization map summing with the object $X$.  To apply the framework of \cite{RWW}, or its generalization by Krannich in \cite{krannich}, one needs an action of a braided monoidal category on the category $\Fd$. We will define this action by giving a braided monoidal structure on the subcategory of $\Fd$ generated by $X$. This structure  will be constructed using the geometric model, namely the category $\M_2$. 

\medskip

  In \cite[Sec 3.2]{HVW}, it is shown that the monoidal category $\M_2$ is not braided, but that none-the-less the disc-object $D$ is a Yang-Baxter element in $\M_2$, that is, it comes equipped with a morphism
  $T:D\hash D\to D\hash D$ satisfying the braid relation: 
  $$(T\hash \id_D)(\id_D\hash T)(T\hash \id_D)=(\id_D\hash T)(T\hash \id_D)(\id_D\hash T).$$ The underlying surface of $D\hash D$ is a cylinder, and $T$ is the Dehn twist along the middle circle of the cylinder. It turns out that, just as in \cite{HVW}, the inverse Dehn twist $T^{-1}$ is actually most convenient for our purpose. Given that $\Hd:\M_2\to \Fd$ is a monoidal functor, it follows that the pair $(X,\beta)$, for $\beta=\Hd(T^{-1})$, is a Yang-Baxter element in $\Fd$.
  What we will show here is that, unlike $D$, the object $X$ satisfies the stronger condition that the full monoidal subcategory of $\Fd$ generated by it is braided, i.e.~$\beta$ is also natural in morphisms of $\Fd$. 
(We prove the naturality property of $\beta$ only for completness. It is not as such necessary for our main results.) 
 
   \begin{prop}\label{prop:beta}
      The twist $\beta:X^{\# 2}\to X^{\# 2}$ defines a braiding on the full monoidal subcategory $\FdX$ of $\Fd$ generated by $X$. 
Moreover, the resulting block braid  
  $$\beta_{n,m}: X^{\# n} \# X^{\#m} \xrightarrow{\cong} X^{\#m} \# X^{\#n}$$
is explicitly given in  matrix form,  with respect to the standard 
basis of $\Z^{n+m}$ by the matrix  
  $$\beta_{n,m}= \begin{pmatrix}
      A_{m,n} & I_m \\ (-1)^mI_n & 0
    \end{pmatrix}\ \ \textrm{with}\ \  (A_{m,n})_{ij}=(-1)^{i+1} 2 $$
 
  \end{prop}

  \begin{proof}
  The hexagon identity follows from the Yang-Baxter  equation, as a direct consequence of the Yang-Baxter equation for $T^{-1}$ \cite[Sec 3.2]{HVW}. The matrix description of $\beta_{n,m}$ can be read out from the effect of block braids on the arcs denoted $\rho_i$ in \cite{HVW}, whose homology classes correspond to the standard generators of $\Z^{n+m}$ under the functor $\Hd$: The computation of the braiding on these arcs is done in the proof of Proposition 5.7 in \cite{HVW}, and the matrix description for $\beta_{n,m}$ can be read directly from this computation.

Left is to check  naturality of the braiding,  i.e.~that for any $\phi \in \Aut(X^{\#n}) \subset GL_n(\Z)$ and any $\psi \in \Aut(X^{\#m}) \subset GL_m(\Z)$ we have $\beta_{n,m}\circ (\phi\oplus \psi)=(\psi\oplus \phi)\circ \beta_{n,m}$. 
This holds if $A_{m,n} \phi =A_{n,m}= \psi A_{m,n}$, which we check now. 

  Explicitly, if $\phi$ and $\psi$ are given in matrix form by matrices $B$ and $C$, then
  $$A_{n,m} \phi=A_{n,m} B= \begin{pmatrix}  2 \sum_j b_{j1} & \dots &2 \sum_j b_{jn} \\
    -2\sum_j b_{j1} & \dots & -2\sum_j b_{jn} \\
    \vdots& & \vdots\\
    (-1)^{m+1}2\sum_j b_{j1} & \dots& (-1)^{m+1}2\sum_j b_{jn}  \end{pmatrix} $$
  and
$$ \psi A_{m,n} = C A_{n,m}= \begin{pmatrix}2 \sum_i (-1)^{i+1} c_{1i} & 2 \sum_i (-1)^{i+1}c_{1i} & \dots &2 \sum_i (-1)^{i+1} c_{1i}\\
    2 \sum_i (-1)^{i+1} c_{2i} & 2 \sum_i (-1)^{i+1}c_{2i} & \dots &2 \sum_i (-1)^{i+1} c_{2i}\\
    \vdots & \vdots & & \vdots\\
  2 \sum_i (-1)^{i+1} c_{mi} & 2 \sum_i (-1)^{i+1}c_{mi} & \dots &2 \sum_i (-1)^{i+1} c_{mi}\\ \end{pmatrix}$$ 
The map $\del$ is represented by the matrix $\del= \begin{pmatrix}
    1 & 1 & \dots & 1
\end{pmatrix}$ 
with respect to the standard basis, hence the equation $\del B=\del$ gives that each column in $B$ must sum to 1, giving the identity $A_{m,n}B=A_{m,n}$. 

To check that $CA_{m,n}=A_{m,n}$ we will use both the invariance of $\del$ and $\la$: The fact that $\del C= \del$ gives that $C^T (\del^T \del) C = \del^T \del$, and 
since $C^T \la C = \la$, we have that $C^T \la' C= \la'$ where $\la'= \la + \del^T \del$ is the ``curved'' form of Remark \ref{remark curved symmetry}. In matrix form, we have 
$$\la'=\begin{pmatrix}
    1& 2& 2& \dots&2\\
    0& 1 & 2&\dots&2\\
    \vdots & \ddots & \ddots & \ddots 
    & \vdots\\
        0& \dots &0& 1 & 2 \\
    0&  & \dots & 0& 1
\end{pmatrix} \ \ \textrm{with}\ \ (\la')^{-1}=\begin{pmatrix}
    1& -2& \ 2& \dots& (-1)^{m-1} 2\\
    0& 1 & -2&\dots& \\
    \vdots & \ddots & \ddots & \ \ \ddots 
    & \vdots\\
          0& \dots &0& 1 & \ -2 \\
    0&  & \dots & 0& \ \ 1
\end{pmatrix}$$
from which we can compute that $$(\la')^{-1} \del^T= \begin{pmatrix}
(-1)^{m-1}\\
    \vdots \\
    -1\\
    1
\end{pmatrix}=(-1)^{m-1}\begin{pmatrix}
    1 \\ -1 \\ \vdots \\ (-1)^{m-1}
\end{pmatrix}.$$ 
Now putting together that $C^T= \la' C^{-1} (\la')^{-1}$ and $C^T\del^T=\del^T$, we get that $C ((\la')^{-1} \del^T)= (\la')^{-1} \del^T$, from which 
the identity  $C A_{m,n}=A_{m,n}$ follows. 
    \end{proof}

\begin{rem} \label{rem char vector}
The last part of the proof recovers the fact, also noted in \cite[Cor 3.5]{MPPRW}, that the vector $v_n=e_1-e_2+\dots +(-1)^{n-1} e_n$ in $\Z^n$ is preserved by the automorphism group of $X^{\# n}$. This special vector behaves differently depending on whether $n$ is odd or even: in the first case $\la(v_n,-)\equiv 0$ and it generates the radical of $\la$, while in the second case $\la(v_n,-)=\del$. The existence of this special vector of two different flavours will be visible in our description of $X^{\# n}$ in the odd and even case in Proposition~\ref{prop:X^n} below. 
Moreover, from the point of view of the curved form $\la'$, which is non-degenerate on $X^{\#n}$, the vector $v_n$ is characterized by $\la'(v_n,-)=\del$ on $X^{\#n}$ as one can check by evaluating on the standard basis.
\end{rem}

    \begin{rem}
Note that the failure of naturality of the twist $T$ in the category $\M_2$ is measured by conjugation with the hyperelliptic involution on one of the factors (see \cite[Prop 5.7]{HVW}). Now the hyperelliptic involution acts as $-1$ on each of the standard generators $[\rho_i]$ of $X^{\hash n}=H_1(D^{\# n},I_0\sqcup I_1)$ and hence the obstruction to naturality in the category $\M_2$ disappears when going to $\Fd$. 
\end{rem}

\subsection{The hyperbolic genus and the arc genus I}\label{sec:genus}

Recall from above the hyperbolic formed space $\HH=(R^2,\la_{\HH})$ in $\F$. 
Let $\HH^{\oplus g}=\HH\oplus\dots \oplus\HH$ its $g$-fold sum.

    The {\em hyperbolic genus} (called {\em Witt index} in \cite{Fri17}) of a formed space $M=(M,\la) \in \F$ is 
$$g_{\HH}(M,\la):=\max\{g\in\bbN \ |\ M\cong \HH^{\oplus g}\oplus M'\ \textrm{in}\ \F\}.$$

\begin{rem}
    We could likewise define the hyperbolic genus of a formed space with boundary $(M,\la,\del)\in\Fd$ using the same formula, replacing $\HH$ by the object $(\HH,0)$ of $\Fd$, setting the boundary map to be trivial. Note that $(\HH,0)^{\# g}=(\HH^{\oplus g},0)$ since the boundary map is trivial. In fact, the $(\HH,0)$--genus of $(M,\la,\del)$ in $\Fd$ is really the same thing as the hyperbolic genus of $(\ker \del,\la|_{\ker\del})$. It will appear in this latter form in our computations below. 
\end{rem}

For the rest of the paper the most relevant quantity will be the {\em arc genus} (or $X$--genus)  of a formed space $(M,\la,\del)$, defined as 
$$g_{X}(M,\la,\del):=\max\{g\in\bbN \ |\ M\cong M'\#X^{\# g} \ \textrm{in}\ \Fd\}.$$
This is the quantity that will rule the relevant stability bounds. 

The above definitions of hyperbolic  and arc genus are based on how many copies of $\HH$  or $X$ one can split off $(M,\la)$  or $(M,\la,\del)$. 
We will see below, in Lemma~\ref{lem:splitH}, that the splitting condition ``comes for free'' for such objects. The result will use that the formed spaces $X^{\# n}$ are non-degenerate, in a sense that we define now.   

Recall that a formed space $(M,\la)$ is {\em non-degenerate} if the form $\la:M\otimes M \to R$ is a non-degenerate pairing.  
The hyperbolic space $\HH^{\oplus g}$ is for example non-degenerate for every $g$. 
For formed spaces with boundary, we generalize this definition as follows: 
Recall that the curved form associated to a formed space $(M,\la,\del)$ is the form $\la'$ defined by $\la'(x,y)=\la(x,y)+\del(x)\del(y)$. 
\begin{Def}
A formed space with boundary $(M,\la,\del)$ is {\em non-degenerate} if the curved form $\la': M \otimes M \to R$ is a non-degenerate pairing, i.e.~if the map $M \to M^\vee$, $m \mapsto \la'(m,-)$ is an isomorphism. 
\end{Def}
The formed spaces with boundary $X^{\#n}$ are non-degenerate for any $n \geq 1$, as can be checked by inspecting their curved forms, whose matrix is given in the proof of Proposition~\ref{prop:beta} above. 

\begin{lem}\label{lem:splitH}
\begin{enumerate}[(i)]
    \item Any morphism $\HH^{\oplus g} \xrightarrow{f} (M,\la)$ in $\F$ gives an orthogonal splitting 
    $$(M,\la) \cong \HH^{\oplus g} \oplus (\im(f)^{\perp},\la|_{\im(f)^{\perp}}).$$ 
    \item Any morphism $X^{\#n} \xrightarrow{f} (M,\la,\del)$ in $\Fd$ gives a splitting 
    $$(M,\la,\del) \cong (M \setminus \im(f),\la|_{M \setminus \im(f)},\del|_{M \setminus \im(f)}) \# X^{\#n}$$
    where for $N \subset M$ we let $M \setminus N:=\{m \in M: \; \la'(n,m)=0 \; \forall n \in N\}$. 
\end{enumerate}
\end{lem}

\begin{proof}
Part (i) follows from the fact that $\HH^{\oplus g}$ is non-degenerate, see e.g.~\cite[Chap 3]{Fri17}. The same proof applies to Part (ii), as we detail now.  

Part (ii): Since $\la'$ is non-degenerate on $X^{\#n}$ then $f$ has to be injective and we get a splitting $M=(M \setminus \im(f)) \oplus \im(f)$ as $R$-modules. 
Observe that $\del=\del|_{M \setminus \im(f)} \oplus \del|_{\im(f)}$ splits with respect to the above decomposition. 
By definition, if $x \in M \setminus \im(f)$ and $y \in \im(f)$ then $\la'(y,x)=0$ and hence $\la(y,x)=-\del y \del x$ so that $
\la$ also agrees with the sum 
$\la|_{M \setminus \im(f)} \# \la_{\im(f)}$ on cross-terms. 
Thus, we get 
$$(M,\la,\del)=(M \setminus \im(f),\la|_{M \setminus \im(f)},\del|_{M \setminus \im(f)}) \# (\im(f),\la|_{\im(f)},\del|_{\im(f)}).$$ 
Since $f$ is injective and preserves the data then $(\im(f),\la|_{\im(f)},\del|_{\im(f)}) \cong X^{\#n}$, as required. 
\end{proof}

The following enhanced version of the above splitting result will be useful later, when identifying the destabilization complex:

\begin{lem} \label{lem: enhanced splitting}
    Let $f,f':N\# X^{\# n}\xrightarrow{\simeq} M$ be isomorphisms in $\Fd$ and suppose that $f|_{X^n}=f'|_{X^n}$. Then $f'=f\circ (g\# \id_{X^n})$ for $g\in \Aut_{\Fd}(N)$. 
    
\end{lem}

\begin{proof}
    By the definition of $\#$ and the previous lemma it follows from the assumption that $f(N)=M \setminus f(X^{\#n})$ and that $f'(N)=M \setminus f'(X^{\#n})$.
    Since $f|_{X^n}=f'|_{X^n}$, then $f(N)=f'(N)$, and the map $g=f' \circ f^{-1}|_N$ gives the result. 
\end{proof}

Thus, to get bounds on the genera it suffices to produce maps from $\HH^{\oplus g}$ or $X^{\#n}$ and the splitting is automatic. 
We will show later in Section \ref{sec:genus II} that the hyperbolic and arc genera are related. We start by relating now the formed space $X^{\# n}$ to the hyperbolic space $\HH^{\oplus g}$.

 \begin{prop} \label{prop:X^n}
 There are isomorphisms of formed spaces with boundary
 \begin{align*}
     &X^{\#2g}\cong (\HH^{\oplus g},\del=\la(e_g,-)) = \ (R^{2g}\,,\, \la_{\HH^{\oplus g}}\,,\,\del=\la(e_g,-))\\ 
     &X^{\#2g+1}\cong (\HH^{\oplus g},0)\ \# \ X \ = \  (R^{2g}\oplus R\,,\,\la_{\HH^{\oplus g}}\oplus 0\,,\,\del=pr_R).
 \end{align*}
     \end{prop}

     \begin{proof}
         We will show the result by induction on the number $n$ of $X$--summands, doing the cases $n=1,2,3$ by hand first. 
         
         For $n=1$, the result is trivial. 

         For $n=2$, our definition of $X^{\#2}$ is $\Big(R^2, \la= \begin{pmatrix} 0 & 1 \\ -1 & 0 \end{pmatrix},\del=\begin{pmatrix} 1 & 1 \end{pmatrix}\Big)$. 
         Let $x,y$ denote the standard basis vectors in this basis. 
         Then $e:=x-y, f:=y$ defines a new basis such that 
         $$X^{\#2}=\Big(R\langle e, f \rangle, \la= \begin{pmatrix} 0 & 1 \\ -1 & 0 \end{pmatrix},\del=\begin{pmatrix} 0 & 1 \end{pmatrix}= \la(e,-) \Big)$$
          as required. 

         For $n=3$, our definition of $X^{\#3}$ is 
         $$X^{\#3}=\Big(R^3, \la= \begin{pmatrix} 0 & 1 & 1 \\ -1 & 0 & 1 \\ -1 & -1 & 0 \end{pmatrix},\del=\begin{pmatrix} 1 & 1 & 1 \end{pmatrix}\Big).$$
         Let $x,y,z$ denote the standard basis vectors in this basis. 
         Then $e:=x-y, f:=y-z, u:= x-y+z$ defines a new basis such that 
         $$X^{\#3}=\Big(R\langle e, f, u \rangle, \la= \begin{pmatrix} 0 & 1 & 0 \\ -1 & 0 & 0 \\ 0 & 0 & 0 \end{pmatrix},\del=\begin{pmatrix} 0 & 0 & 1 \end{pmatrix} \Big) \cong  (\HH,0) \# X $$
          as required. 

          For an even number $n=2g\geq 4$,  we have that $X^{\# 2g}=X^{\# 2g-1} \# X$, which by induction is isomorphic to $(\HH^{g-1},0) \# X \# X= (\HH^{g-1},0) \# X^{\# 2}$. 
          The result then follows by applying the above description of $X^{\#2}$. 

          Finally for  and odd number $n=2g+1\geq 5$, we have that $X^{\# 2g+1}= X^{\# 2(g-1)+1} \# X^{\# 2}$, which by induction is isomorphic to $(\HH^{g-1},0)\# X \# X^{\#2}$, and the result follows from the above expression for $X^{\#3}$. 
     \end{proof}

      \begin{prop} \label{prop: automorphisms X^n}
$\Aut(X^{\# 2g})\cong \operatorname{Stab}_{\Sp_{2g}(R)}(e_g)$ and $\Aut(X^{\# 2g+1})\cong \Sp_{2g}(R)$. 
       \end{prop}

See \cite[Thm 3.4]{MPPRW} for an alternative proof of this result. 
       
       \begin{proof}
       By Proposition \ref{prop:X^n} we have $\Aut(X^{\#2g})\cong \Aut(\HH^{\oplus g},\del=\la(e_g,-))$, but using that $\HH^{\oplus g}$ is non-degenerate we can re-write the right hand side as $\Aut(\HH^{\oplus g},e_g)= Stab_{Sp_{2g}(R)}(e_g)$. 

       By Proposition \ref{prop:X^n} we have $\Aut(X^{\# 2g+1})\cong\Aut((\HH^{\oplus g},0) \# X)$. 
       Let $\la^{\vee}: M \rightarrow M^{\vee}$ be defined by $\la^{\vee}(m)= \la(m,-)$. Then $\del^{-1}(1) \cap \ker \la^{\vee}$ is preserved by any automorphism, and by the above description it consists of a single element $x$ which generates $X$ in $(\HH^{\oplus g},0) \# X$. 
       Moreover, $\langle x \rangle ^{\perp}= \HH^{\oplus g}= \ker \del$ must also be preserved by any automorphism. 
       Thus, $\Aut((\HH^{\oplus g},0) \# X) \subset \Aut(\HH^{\oplus g})=\Sp_{2g}(R)$, which implies the result as the other inclusion is trivial. 
       \end{proof}

       In order to get our main results, we need to compare the classical stabilization map on (even) symplectic groups induced by $-\oplus \HH$ in the category $\F$ with our double stabilization map induced by $-\# X^{\#2}$ in $\Fd$. The following result says that the two stabilization maps agree. 
       \begin{lem} \label{lem: stabilisations agree}

The following diagram commutes
\begin{align*}
    \xymatrix{\Aut(X^{\# 2g+1}) \ar[rr]^-{- \# \id_{X}} \ar[d]^-{\cong} & & \Aut(X^{\# 2g+2})\ar[rr]^-{- \# \id_{X}}  \ar[d]^-{\cong} & & \Aut(X^{\# 2g+3}) \ar[d]^-{\cong}\\ \Sp_{2g}(R) \ar[rr]^-{\text{incl}} && \rm{Stab}_{\Sp_{2g+2}(R)}(e_{g+1}) \ar[rr]^-{\text{incl}} & & \Sp_{2g+2}(R)}
\end{align*}
where the vertical maps are the isomorphisms obtained in Proposition~\ref{prop: automorphisms X^n} and the horizontal maps are the stabilization maps. Moreover, the bottom horizontal composition is $-\oplus \id_{\HH}$. 
 
       \end{lem}

    \begin{proof}
  The isomorphism $\Aut(X^{\# 2g+1})\cong \Sp_{2g}(R)$ comes from the identification $X^{\# 2g+1}=\HH^{\oplus g}\# X$ of Proposition \ref{prop:X^n}, together with the fact that the last $X$ summand is necessarily fixed by automorphisms, because it is generated by the only boundary 1 element $x=v_g$ that is orthogonal to $\ker\del_g$.
   The isomorphism $\Aut(X^{\#2g+2}) \cong \rm{Stab}_{\Sp_{2g+2}(R)}(e_g)$ is obtained from the previous one, using first that 
   \[X^{\#2g+2}=X^{\#2g+1}\# X= \HH^{\oplus g} \# X^{\#2},\]
   and then using the explicit identification given in the proof of Proposition \ref{prop:X^n} (with $x=v_g$ and $y=x_{g+2}$) to write $X^{\#2}=(\HH \langle e_{g+1},f_{g+1}'\rangle,\del=\la(e_{g+1},-))$, where $e_{g+1}=v_{g}-x_{2g+2}$ and $f_{g+1}'=x_{2g+2}$. 
   In particular, commutativity of the first square follows from the fact that if $\phi \in \Aut(X^{\#2g+1})$ then $\phi \# \id_X$ fixes $v_g-x_{2g+2}=e_{g+1}$ since it fixes both $v_g$ and $x_{2g+2}$. 
   
    From the proof of  Proposition~\ref{prop:X^n}, the isomorphism $\Aut(X^{\# 2g+3})\cong \Sp_{2g+2}(R)$ is obtained from the $2g+1$ case, using first that 
    $$X^{\#2g+3}=X^{\# 2g+1}\# X^{\# 2}=\HH^{\oplus g}\# X\#  X^{\#2} = \HH^{\oplus g}\oplus \HH \# X$$
    where for the last isomorphism we have $X^{\# 3}$ generated by $v_g,x_{2g+2},x_{2g+3}$.
    Now following the identifications given in the proof of that result, we see that the last $\HH$-summand has a new hyperbolic basis $(e_{g+1},f_{g+1})=(v_g-x_{2g+2},x_{2g+2}-x_{2g+3})$, while the last $X$-summand is generated by $v_g-x_{2g+2}+x_{2g+3}$.
    In particular notice that $e_{g+1}$ is the same as before while $f_{g+1}=f'_{g+1}-x_{2g+3}$, and that the last $X$-summand is generated by $e_{g+1}+x_{2g+3}:=v_{g+1}$. Hence, if $\varphi\in \Aut(X^{\# 2g+1}) \cong \rm{Stab}_{\Sp_{2g+2}(R)}(e_{g+1})$, its image $\varphi\# \id_{X}$ in $\Aut(X^{\# 2g+3})$ fixes $e_{g+1}$, because $\varphi$ fixes it, and fixes the last generator $x_{2g+3}$. Hence it fixes both the last $\HH$--summand and the last $X$--summand. Thus, the second square also commutes. 
    The fact that the bottom composition is just $-\oplus \id_{\HH}$ follows by construction. 
    \end{proof}

\begin{rem}[Other form parameters]\label{other form parameters}
   The orthogonal group $O_{n,n}(R)$ and the unitary group $U_{n,n}(R)$ can both be defined, just like the symplectic group $\Sp_{2n}(R)$, as the automorphism group of the hyperbolic space $\HH^n$, just changing the {\em form parameters}, working with quadratic or Hermitian forms instead of alternating forms, see e.g.~\cite{MirzaiivdK}. The category $\Fd$ can be defined in these other contexts, and the object $X$ exists just as well. Adapting accordingly the definition of the monoidal structure in $\Fd$, we can also ensure that $X\# X\cong (\HH,\del)$. However, it will no longer hold that $X^{\# n}$ has a large hyperbolic genus as $n$ increases, and hence the sequence of groups $\Aut(X^{\# n})$ will not be related to the classical groups $O_{n,n}(R)$ and $U_{n,n}(R)$ in those cases.
   For example, in the case of symmetric forms one can check that over $\Z$ the form associated to $X^{\#n}$ has signature $(1,n-1)$ and hence its hyperbolic genus is $1$ for any $n \ge 2$.
\end{rem}

\begin{rem}[Direct sum monoidal structure]
The category $\Fd$ has a monoidal structure $\#$ as discussed in Section \ref{sec: formedd}, but it also has a more obvious monoidal structure $\oplus$ given by direct sum $(M_1,\la_1,\del_1) \oplus (M_2,\la_2,\del_2)=(M_1\oplus M_2, \la_1 \oplus \la_2, \del_1 + \del_2)$. 
This second monoidal structure is symmetric but not relevant for our purposes since $\la_{X^{\oplus n}}=0$ for all $n$ so $\{X^{\oplus n}\}_{n\ge 0}$ is unrelated to symplectic groups. 
However, by Proposition \ref{prop:X^n} $X^{\#2g} \oplus X^{\#2h} \cong X^{\#2(g+h)}$ so the full subcategory $\mathcal{C}$ of $\Fd$ with objects formed spaces with boundary which are isomorphic to some $X^{\#2g}$ for $g \ge 0$, carries two natural monoidal structures: $\#$ and $\oplus$. 
It turns out that they are equivalent as monoidal structures but not as braided monoidal ones (hence in particular the face maps of the corresponding destabilization complexes are a priori different). More precisely, 
\begin{enumerate}[(i)]
     \item The identity functor $F=\text{Id}_\mathcal{C}: (\mathcal{C},\oplus) \to (\mathcal{C},\#)$ is strong monoidal. 
     In fact, we can give an explicit example of strong monoidality $L$ for the functor $F$ as follows: given $\underline{M_i}=(M_i,\la_i,\del_i) \in \Fd$ for $i \in \{1,2\}$ we let 
         \[\begin{tikzcd}
             L_{\underline{M_1}, \underline{M_2}}: F(\underline{M_1}) \# F(\underline{M_2}) \rar{\simeq} & F(\underline{M_1} \oplus \underline{M_2}) \\ (x,y) \rar & (x- (\del_2 y) \cdot m_1, y),
         \end{tikzcd}\]
         where $m_i \in M_i$ is uniquely characterized by $\del_i= \la_i(m_i,-)$. 
    \item There is no non-zero braided strong monoidal functor $(\mathcal{C},\oplus) \to (\mathcal{C},\#)$. 
\end{enumerate}
Part (i) can be checked explicitly using Proposition \ref{prop:beta} and Remark \ref{rem char vector} and using that $\del_i m_i=0$, and part (ii) follows from the fact that the braidings of Proposition \ref{prop:beta} never square to the identity on non-zero modules. 
\end{rem}

\subsection{Properties of the base ring} \label{subsec: properties of R}

In this section we will address the question of which (commutative) base rings $R$ satisfy the main results of this paper. These conditions will be some of the so called ``stability conditions''. 

Recall that a vector $m \in M$ is called \textit{unimodular} if it generates a split $R$-summand in $M$. 
Equivalently, $m$ is unimodular if there exits $\phi \in M^\vee$ such that $\phi(m)=1$. 
More generally, a sequence $(m_1,\dots,m_n)$ of vectors in $M$ is called \textit{unimodular} if the canonical map $R\langle m_1 \rangle \oplus \cdots \oplus R \langle m_n \rangle \to M$ is a split injection. 
Equivalently, $(m_1,\dots,m_n)$ is unimodular if there exist $\phi_1,\dots,\phi_n \in M^\vee$ such that $\phi_i(m_j)=\delta_{i,j}$. 

\begin{defn}\label{def:sr}\cite[Sec 5 and Def 6.3]{MirzaiivdK}
\begin{enumerate}
    \item  A ring $R$ satisfies the \textit{stable range condition} $(S_n)$ if for every unimodular vector $(r_0,\dots,r_n)$ in $R^{n+1}$ there are $t_0,\dots,t_{n-1} \in R$ such that the vector $(r_0+t_0 r_n, \dots,r_{n-1}+t_{n-1}r_n) \in R^n$ is unimodular. 
    If $n$ is the smallest such number we say that $R$ has \textit{stable rank} $sr(R)=n$, and it has $sr(R)=\infty$ if no such $n$ exists. 
    \item The {\em unitary stable rank $usr(R)$} is the smallest number $n\ge sr(R)$ and such that the subgroup $ESp_{2n+2}(R)$ of elementary matrices acts transitively on the set of unimodular vectors in $R^{2n+2}$. 
    \end{enumerate}
\end{defn}

The group $ESp_{2n+2}$ is defined in \cite[Sec 6]{MirzaiivdK}. We will not use the explicit description of this group, but rather the transitivity and cancellation properties derived from the unitary stable rank condition in \cite{MirzaiivdK,Fri17}. 

By definition, $sr(R)\le usr(R)$. The following are some examples of rings with bounds on the unitary stable ranks.

\begin{ex} \label{ex usr} \cite[Ex 6.5]{MirzaiivdK}
    \begin{enumerate}[(i)]
    \item A commutative Noetherian ring $R$ of finite Krull dimension $d$ satisfies $usr(R) \leq d+1$.  
    In particular,
if $R$ is a Dedekind domain then $usr(R) \leq 2$ and for a field $k$, the polynomial ring
$K = k[t_1,\dots,t_n]$ satisfies $usr(K) \leq n + 1$.
\item More generally, any $R$-algebra $A$ that is finitely generated as an $R$-module satisfies $usr(A) \leq d+ 1$
for $R$ again a commutative Noetherian ring of finite Krull dimension $d$.
\item Recall that a ring $R$ is called \textit{semi-local} if $R/J(R)$ is a left Artinian ring, for $J(R)$ the Jacobson
radical of $R$. A semi-local ring satisfies $usr(R) = 1$. 
\end{enumerate}
\end{ex}

The following result may be seen as somewhat surprising, as it gives a good behavior of hyperbolic summands under a genus condition ruled by the stable rank $sr(R)$, and not the unitary stable rank $usr(R)$. 
In the following section, we will combine this result with finiteness of the unitary stable rank to deduce some important properties of the hyperbolic genus that will be used for the rest of the paper.

\begin{prop} \label{prop: finiteness of usr*}
    Let $R$ be a (commutative) ring with $sr(R)<\infty$.  Let $(M,\la) \in \F$ with $g_{\HH}(M,\la) \ge \frac{sr(R)}{2}$ and let $l \in M^\vee$ be unimodular. 
    Then there is a hyperbolic summand $\HH \cong (H,\la|_H) \subset (M,\la)$ such that $l|_H$ is unimodular in $H^\vee$.
\end{prop}

\begin{proof}
    We will proceed in two steps, the first one is a special case and the second one is to deduce the general case from it. Write $g:=g_{\HH}(M,\la)$. 

\smallskip

    \noindent
    \textbf{Step 1.} Suppose first that $(M,\la)=\HH^{\oplus g} \oplus (R,0)$. 
    Forgetting the form, we can identify $M=R^{2g+1}$ as an $R$-module. Denote $x_1,\dots,x_{2g+1}$ the standard basis vectors of $R^{2g+1}$, where the first $2g$ vectors span $\HH^{\oplus g}$ and $x_{2g+1}$ spans the summand $(R,0)$. 

    Since $l$ is unimodular then $(l(x_1),\dots,l(x_{2g}),l(x_{2g+1})) \in R^{2g+1}$ is a unimodular vector.
    By assumption $2g \ge sr(R)$, so by definition of the stable rank there exist $t_1,\dots,t_{2g} \in R$ such that $(l(x_1)+t_1 l(x_{2g+1}),\dots,l(x_{2g})+t_{2g}l(x_{2g+1})) \in R^{2g}$ is unimodular. 
    Thus, there exist $a_1,\dots,a_{2g} \in R$ such that $\sum_{i=1}^{2g}{a_i(l(x_i)+t_i l(x_{2g+1}))}=1$.  Note that the vector $(a_1,\dots,a_{2g}) \in R^{2g}$ is necessarily itself also unimodular. 

    Now consider $u=(u_0,u_1)\in R^{2g}\oplus R=M$, where $u_0=(\sum_{i=1}^{2g}{a_i x_i})$ and $u_1=(\sum_{i=1}^{2g} a_i t_i) x_{2g+1}$.
    By construction $l(u)=1$, 
    and $u_0$ is unimodular.  Let $v \in \HH^{\oplus g}$ be such that $\la(u_0,v)=1$. 
    Now we also have $\la(u,v)=1$ since $x_{2g+1}$ is orthogonal to $\HH^{\oplus g}$. 
    Then $H=\langle u,v \rangle \subset M$ is a hyperbolic summand by Lemma \ref{lem:splitH}(i), and it has the desired properties. 

    \smallskip

\noindent
    \textbf{Step 2.} Now consider a general formed space $(M,\la)$. Up to isomorphism, 
$(M,\la)=\HH^{\oplus g} \oplus (N,\la_N)$.
    Since $l$ is unimodular, then there is some $x \in M$ such that $l(x)=1$.  Write $x=x_0+x_1$ with $x_0 \in \HH^{\oplus g}$ and $x_1 \in N$. 
    Since $N$ is orthogonal to $\HH^{\oplus g}$ and any vector is isotropic, we get a (not necessarily injective) morphism in $\F$ 
    $$\phi= \id \oplus x_1: \HH^{\oplus g} \oplus (R,0) \rar  \HH^{\oplus g}\oplus (N,\la|_N)= (M,\la).$$
   Note that the pull-back map $l':=l \circ \phi$ is unimodular since $l'((x_0,1))=l(x_0+x_1)=1$ by construction. 
    Thus, by Step 1 there is a hyperbolic $H \subset \HH^{\oplus g} \oplus (R,0)$ such that $l'|_{H}$ is unimodular in $H^\vee$. 
    Hence, by Lemma \ref{lem:splitH}(i), $\phi(H) \subset (M,\la)$ is a hyperbolic summand and by construction $l|_{\phi(H)}$ is unimodular in $\phi(H)^\vee$, as required.    
\end{proof}

\smallskip

Since PIDs have finite unitary stable rank, all of the above applies to them. However, we will in the following treat the case of PIDs separately, because PIDs have special properties that allow to get slightly better results than for general rings. 
We summarize these properties in the following lemma: 

\begin{lem} \label{lem: pid} 
Let $R$ be a PID and $M$ a finitely generated free module. Then
    \begin{enumerate}[(i)]
        \item If $m \in M$ is non-zero 
        then $m= r \cdot m'$ where $r \in R$ and $m' \in M$ is unimodular. Moreover, $GL(M)$ acts transitively on the set of unimodular vectors in $M$.
        \item If $\la \in \Lambda^2 M^\vee$ is an alternating form on $M$ then $\la$ is isomorphic to a canonical form 
        \[\bigoplus_{i=1}^{k} d_i\!\HH \oplus (R^l,0)\]
        such that $d_1 | d_2 | \cdots | d_k$ are non-zero and unique up to multiplying by units of $R$, and $d \HH$ means $\Big(R^2, \begin{pmatrix}
            0 & d \\ -d & 0
        \end{pmatrix}\Big)$. 
        \item If $(M,\la) \in \F$ then $g_{\HH}(M,\la)$ is half the number of units in the Smith normal form of a matrix representing $\la$. 
    \end{enumerate}
\end{lem}

\begin{proof}
To prove (i), we can assume $M=R^n$ for some $n$. Pick standard basis $e_1,\dots,e_n$ and let $m=r_1 e_1+ \cdots+r_n e_n$, $r_i \in R$. 
 Then let $r \in R$ be such that there is an equality of ideals $(r_1,\dots,r_n)=(r)$ in $R$. We must have $r \neq 0$ since $m\neq 0$ and we can write $r_i=r r_i'$ for unique $r_i' \in R$. 
 Then let $m'=r_1' e_1+\cdots+r_n' e_n$. 
 By construction $r \in (r_1,\dots,r_n)$ so there are $a_1,\dots,a_n \in R$ such that $r=a_1 r_1+\cdots a_n r_n$, so $1=a_1 r_1'+ \cdots a_n r_n'$. 
 Thus, $m'$ is unimodular and $m=rm'$, proving the statement. 

 The moreover part follows by Smith normal form since we can assume without loss of generality that $M=R^n$.

 The proof of (ii) can be found in \cite[Theorem IV.I]{Newman}. It is stated for skew-symmetric forms in characteristic not 2, but the same proof works also for alternating forms in characteristic 2. 

 Part (iii) follows immediately from (ii) as any two matrices representing $\la$ will have the same Smith normal form, and the Smith Normal Form of the canonical form defined above is given by  $diag(d_1,d_1,\dots,d_k,d_k,0,\dots,0)$.
\end{proof}

\subsection{The hyperbolic genus and the arc genus II}\label{sec:genus II}

Our goal now is to relate the hyperbolic and arc genera and then get bounds on how much the arc genus can decrease when taking kernels of linear functionals. We will later define the operation of ``cutting arcs" in a given module, see Section \ref{sec: cutting arcs}, and these results will imply bounds on how the arc genus decreases when cutting arcs. Given a formed space with boundary $(M,\la,\del)$, we will control its arc genus using the hyperbolic genus of $(M,\la)$ and that of the kernel of $\del$. 
To begin with, we summarize some properties of hyperbolic forms that will be used for the rest of the paper.

 \begin{prop}\label{prop:properties hyperbolics} Let $R$ be a ring 
 and $(M,\la)$ a formed space of genus $g=g_{\HH}(M,\la)$. 
 
 \noindent
 Suppose that $g \ge usr(R)+1$ or $R$ is a PID. 
 \begin{enumerate}[(i)]
         \item If $f: \HH \rightarrow (M,\la)$ is a morphism in $\F$,  
         then $g_{\HH}(\im(f)^{\perp},\la|_{\im(f)^{\perp}})=g-1$. 
         \item The group $\Aut(\HH^{\oplus g})\cong\Sp_{2g}(R)$ acts transitively on the set of unimodular vectors of $R^{2g}$. 
         \item If $l \in M^\vee$ is unimodular, 
         then there is a hyperbolic basis $(e_1,f_1,\dots,e_g,f_g)$ in $M$ such that $l(e_1)=1$, $l(e_i)=0$ for $2 \le i \le g$, and $l(f_i)=0$ for all $i$. 
         Moreover, if $R$ is a PID, then for any $l \in \HH^\vee$, there is a hyperbolic basis $e,f$ such that $l(f)=0$.
         \item If $l \in M^{\vee}$ is unimodular, 
         then $g_{\HH}(\ker l, \la|_{\ker l}) \geq g-1$. 
         Moreover, if $R$ is a PID then the unimodularity condition can be dropped.  
     \end{enumerate}
 \end{prop}

 \begin{proof}
 \textbf{Part (i)}: Let us first consider the case that $R$ is not a PID.  By Lemma \ref{lem:splitH}(i) we have $(M,\la) \cong \HH \,\oplus \,(\im(f)^\perp, \la|_{\im(f)^\perp})$. Hence $g_{\HH}(M,\la)=g \ge 1+ g_{\HH}(\im(f)^\perp, \la|_{\im(f)^\perp})$. 
 For the reversed inequality, we have that $(M,\la)\cong \HH^{\oplus g} \oplus (N,\la_N)=\HH \oplus (\HH^{g-1} \oplus (N,\la_N))$. Given that $g \ge usr(R)+1$,  
 cancellation (\cite[Cor 3.14]{Fri17})
 gives that $(\im(f)^\perp, \la|_{\im(f)^\perp}) \cong \HH^{g-1} \oplus (N,\la_N)$, proving the result. 

 When $R$ is a PID, use again Lemma \ref{lem:splitH}(i) to write $(M,\la) \cong \HH \oplus (\im(f)^\perp, \la|_{\im(f)^\perp})$. The number of units in the Smith normal form of a matrix representing $\la$ is thus the number of units in the Smith normal form of a matrix representing $\la|_{\im(f)^\perp}$ plus $2$. 
 This gives the result by Lemma \ref{lem: pid}(iii).

\smallskip

 \noindent
\textbf{Part (ii)}: This follows directly from the definition of $usr(R)$ under the assumption $g\ge usr(R)+1$. When $R$ is a PID, we need to check that the result also holds for small values of $g$. 
Let $x \in \HH^{\oplus g}$ be a unimodular vector. 
Since $\la$ is unimodular, there must exist $y\in \HH^{\oplus g}$ such that $\la(x,y)=1$. Thus, $\langle x,y\rangle$ is a hyperbolic summand in $\HH^{\oplus g}$, and it follows from Lemma~\ref{lem: pid}(ii) that the orthogonal complement $\langle x,y\rangle^\perp$ is isomorphic to $\HH^{\oplus g-1}$. This reduces the computation to genus 1, where we have that $\Sp_2(R)=\operatorname{SL}_2(R)$ (since $R$ is commutative), from which transitivity follows by transitivity of the $GL_2(R)$-action, which can be shown by writing any unitary vector in coordinates as a matrix and using that its
Smith normal form is $(1,0)$.

\smallskip

 \noindent
\textbf{Part (iii)}: 
Assume first that $g=g_{\HH}(M,\la) \ge {usr(R)+1}$. Then $g \ge sr(R)+1 \ge \frac{sr(R)}{2}$, so we can find a hyperbolic summand $\HH \subset (M,\la)$ such that $l|_{\HH}$ is unimodular by Proposition \ref{prop: finiteness of usr*}, and write $(M,\la)=\HH \oplus (\tilde{M},\tilde{\la})$ by Lemma \ref{lem:splitH}(i). 
Now use part(i) to deduce that $g_{\HH}(\tilde{M},\tilde{\la})=g-1$ and thus we can write $(M,\la)=\HH^{\oplus g} \oplus (N,\la_N)$ such that $l|_{\HH^{\oplus g}}$ is surjective. 

Since $\la|_{\HH^{\oplus g}}$ is non-degenerate then there is $v \in \HH^{\oplus g}$ unimodular such that $l|_{\HH^{\oplus g}}=\la(v,-)|_{\HH^{\oplus g}}$. Let $(e_1,f_1,\dots,e_g,f_g)$ be a hyperbolic basis of $\HH^{\oplus g}$. 
By part (ii) we can apply an automorphism of $\HH^{\oplus g}$ to assume without loss of generality that $v=-f_1$, which gives the result. 

\smallskip

Now suppose that $R$ is a PID and consider first the case where $l\in \HH^\vee$ is any linear map.   
Pick a hyperbolic basis $e,f$ of $\HH$ and let $d \in R$ be the generator of the ideal $(l(e),l(f))$, i.e.~$(d)=\im(l) \subset R$. 
   If $d=0$ then we are done.  Otherwise, let $A, C \in R$ be such that $d= A l(e)+ C l(f)$. 
   Now consider the ideal $(A,C) \subset R$. We claim that $(A,C)=R$: let $r \in R$ be a generator for $(A,C)$ then $A=r A'$, $C= r C'$ for some $A', C' \in R$, so $d= r d'$ where $d'= A' l(e) + C'l(f) \in (l(e),l(f))=(d)$, so $d'= u d$ for some $u \in R$. 
   Thus, $d= r d'= r u d$, so $r u =1$ so $r$ is a unit, as required. 
   Therefore, there are $B, D \in R$ such that $A D - B C=1$. Hence there is a matrix $\begin{pmatrix}
       A & B \\ C & D
   \end{pmatrix} \in SL_2(R)=\Sp_2(R)$,  sending the hyperbolic basis $e, f$ to a new hyperbolic basis $e',f'$ such that $l(e')= d$ and $d | l(f')$. 
   Applying the symplectic transformation $f' \mapsto f'- (l(f')/ d) e'$ gives the last part of the statement. 

   Assume finally that $l\in M^\vee$ is unimodular with $R$ a PID, and let $(e',f')$ be a basis of a hyperbolic summand in $M$ with $l(f')=0$ as just obtained. Consider $M'=<e',f'>^\perp \subset M$. Since $l$ is unimodular, there is some $x \in M'$ such that $l(te'+x)=1$ for some $t \in R$. 
Let $f''=te'+x+f'$. Then $l(f'')=1$ and $\la(e',f'')=1$ so $H=\langle e',f''\rangle$ is a new hyperbolic summand of $M$ in which $l$ is surjective. Now we can proceed as in the non-PID case, using the results already proved in part (i) and (ii).

   \smallskip

\noindent
\textbf{Part (iv)}: 
In the non-PID case, this folows from part (iii) by considering $\langle e_2,f_2,\dots,e_g,f_g\rangle$.
In the case that $R$ is a PID,  
 write $(M,\la)=\HH^{\oplus g} \oplus (N,\la_N)$. 
Consider the restriction of $l$ to $\HH^{\oplus g}$. By non-degeneracy it can be written as $\la_{\HH^g}(v,-)$ for some $v \in \HH^{\oplus g}$. If $v=0$, the result is trivial. Otherwise, 
use Lemma \ref{lem: pid}(ii) to write $v=r v'$ for some $r \in R \setminus \{0\}$ and $v' \in \HH^{\oplus g}$ unimodular. 
Since $\langle v' \rangle^\perp \subset \HH^{\oplus g}$ is contained in $\ker l$, it suffices to prove that $g(\langle v' \rangle^\perp) \geq g-1$. 
This follows from part (ii) since we can assume $v'=e_1$.
 \end{proof}

Recall the formed space $X=(R,0,\id)$. 
Note that $g_X(\HH,0)=0$ for the simple reason that $(\HH,0)=(R^2,\la_{\HH},0)$ has no ``arc" vector with boundary 1. More generally, if we want to have $g_X(M,\la,\del)>0$ then we need $\del$ to be surjective, which in our context is the same as unimodular, provided that $M \ne 0$. 
In general, we have the following result relating $g_X$ to $g_{\HH}$:

 \begin{prop} \label{prop:formula for gX}
 Let $R$ be a ring and $(M,\la,\del)$ a formed space with boundary over $R$. 
 \begin{enumerate}[(a)]
     \item 
 $g_X(M,\la,\del) \leq 1+g_{\HH}(M,\la)+g_{\HH}(\ker\del,\la|_{\ker\del})$.\\
 In particular, if $g_X(M,\la,\del)\ge 2n$ then $g_{\HH}(M,\la,\del)\ge n$.  
 \item If ${usr(R)}<\infty$ and one of the following three conditions holds: 
 \begin{enumerate}[(i)]
  \item $\del$ is unimodular and $R$ is a PID.
     \item $\del$ is unimodular and $g_{\HH}(M,\la) \ge {usr(R)+1}$. 
     \item $g_X(M,\la,\del) \ge 2 usr(R)+2$. 
 \end{enumerate}
 then equality holds:
   $g_X(M,\la,\del)= 1+g_{\HH}(M,\la)+g_{\HH}(\ker\del,\la|_{\ker\del}).$
   \end{enumerate}
   \end{prop}

   \begin{proof}
   We first show that $g_X(M,\la,\del) \leq 1+g_{\HH}(M,\la)+g_{\HH}(\ker\del,\la|_{\ker\del})$ always holds. 
   
   Pick a morphism $X^{\#n} \xrightarrow{f} (M,\la,\del)$ with $n=g_X(M,\la,\del)$. 
   If $n=2g$, Proposition~\ref{prop:X^n} gives that the underlying form on $X^{\# n}$ is $\HH^{\oplus g}$, and we can pick a hyperbolic basis in $\HH^{\oplus g}$ such that $\del=\la(e_1,-)$ so $\ker \del \supset 0 \oplus \HH^{g-1}$ in this basis. Using $f$, we then get that $g_{\HH}(M,\la) \geq g$ and $g_{\HH}(\ker \del, \la|_{\ker \del}) \geq g-1$. 
   If $n=2g+1$, Proposition~\ref{prop:X^n} gives that $f$ restricts to a morphism of formed spaces $(\HH^{\oplus g},0) \rightarrow (M,\la,\del)$. 
   Thus, $g_{\HH}(M,\la) \geq g_{\HH}(\ker\del,\la|_{\ker\del}) \geq g$ and the inequality also follows. Finally, from the inequality, it follows that $g_{\HH}(M,\la)\le n-1$ implies that $g_{X}(M,\la,\del)\le 2(n-1)+1=2n-1$, giving the second part of (a). 

   \smallskip
   
   Now we will show the reverse inequality $g_X(M,\la,\del) \geq 1+g_{\HH}(M,\la)+g_{\HH}(\ker\del,\la|_{\ker\del})$ whenever $\del$ is unimodular and $g_{\HH}(M,\la) \ge {usr(R)+1}$ or $R$ is a PID, that is (b) assumption (i) or (ii). (Note that when $\del$ is not unimodular we trivially have $g_{X}(M,\la,\del)=0$ and the inequality does not hold.) 
   
   By Proposition~\ref{prop:properties hyperbolics} (iv), $g_{\HH}(\ker\del,\la|_{\ker\del})\ge g_{\HH}(M,\la)-1$. On the other hand,  $g_{\HH}(\ker\del,\la|_{\ker\del})\le g_{\HH}(M,\la)$.   
   Hence there are two cases: when $M$ and $\ker\del$ have the same hyperbolic genus, and when they differ by one. 

   If $g_{\HH}(\ker\del,\la|_{\ker\del})=g_{\HH}(M,\la)=g$ then pick a morphism of forms $\HH^{\oplus g} \xrightarrow{\phi} (\ker\del,\la|_{\ker\del})$. 
  By Lemma~\ref{lem:splitH}(i), there is a splitting $M=\im(\phi) \oplus \im(\phi)^{\perp}$. 
   By construction $\im(\phi) \subset \ker\del$ and by assumption $\del$ is surjective, so $\del|_{\im(\phi)^{\perp}}$ must be surjective too. 
   Pick $x \in \im(\phi)^{\perp}$ with $\del(x)=1$, then $\phi \oplus x: (\HH^{\oplus g},0) \# X \rightarrow (M,\la,\del)$ is a morphism, and the source is isomorphic to $X^{\# 2g+1}$ by Proposition~\ref{prop:X^n}, giving the result by Lemma \ref{lem:splitH}(ii). 

   If $g_{\HH}(\ker\del,\la|_{\ker\del})=g_{\HH}(M,\la)-1$: let $g:= g_{\HH}(M,\la)$.  
   By Proposition \ref{prop:properties hyperbolics}(iii)  with $l=\del$ there is a morphism $\phi:\HH^{\oplus g} \to (M,\la)$ such that $\del(\phi(e_1))=1$ and $\del(\phi(f_1))=\del(\phi(e_2))=\dots=\del(\phi(f_g))=0$.
   In particular, $\phi$ gives a morphism $\phi: (\HH^{\oplus g},\del \circ \phi) \rightarrow (M,\la,\del)$ and by construction $\del \circ \phi$ is surjective, hence unimodular. 
   By Proposition~\ref{prop:X^n} and Proposition \ref{prop:properties hyperbolics}(ii) we have $(\HH^{\oplus g},\del \circ \phi) \cong X^{\# 2g}$, giving the required result by Lemma \ref{lem:splitH}(ii). This proves (b) under assumptions (i) or (ii). 

   Finally, if we assume that $g_X(M,\la,\del) \ge 2usr(R)+2$ then $\del$ is automatically surjective 
   and hence unimodular; and $g_{\HH}(M,\la) \ge {usr(R)+1}$ by (a), which finishes the proof by the previous part. 
   \end{proof}

The following corollary will give us in the next section an estimate of the arc-genus of forms obtained after ``cutting arcs". 

   \begin{cor} \label{cor R of a kernel}
      If $R$ has ${usr(R)+1}<\infty$, $M$ has $g_X(M,\la,\del) \ge 2{usr(R)+4}$, and  $l \in M^{\vee}$ satisfies that $\{l,\del\} \subset M^{\vee}$ is unimodular.  Then $$g_X(\ker l, \la|_{\ker l}, \del|_{\ker l}) \geq g_X(M,\la,\del)-2.$$ 
      Moreover, if $R$ is a PID then no lower bound on the arc genus $g_X(M,\la,\del)$ is needed. 
   \end{cor}

   \begin{proof}
  To prove the result, we will apply Proposition~\ref{prop:formula for gX} to both $M$ and its submodule $\ker l$. 
       
       Note first that the unimodularity condition implies that $\del$ and $\del|_{\ker l} \in (\ker l)^{\vee}$ are unimodular. If $R$ is a PID, assumption (b)(i) then holds for both $M$ and $\ker(l)$. 
       
       In the non-PID case, we check that both modules have large enough hyperbolic genus, checking assumption (b)(ii) instead.  
       Given that $g_{\HH}(M,\la)\ge g_{\HH}(\ker\del,\la|_{\ker \del})$, we have that 
        $$ 1+ 2 g_{\HH}(M,\la)\ge 1+ g_{\HH}(M,\la)+g_{\HH}(\ker \del, \la|_{\ker \del})\ge g_X(M,\la,\del)\ge 2{usr(R)}+4,$$
        where we used Proposition \ref{prop:formula for gX}(a) for $M$ for the second inequality (that holds with no assumption) and our assumption here for the last inequality. 
       This is only possible if $g_{\HH}(M,\la) \ge {usr(R)+2}$, as it has to be an integer. 
       Since $l \in M^\vee$ is unimodular, Proposition \ref{prop:properties hyperbolics}(iv) now gives that 
       \begin{equation}\label{eq:kerl}
    g_{\HH}(\ker l, \la|_{\ker l}) \ge g_{\HH}(M,\la) -1 \ge {usr(R)+1}.
       \end{equation}

       Therefore, we can apply Proposition \ref{prop:formula for gX}(b) to both $(M,\la,\del)$ and $(\ker l,\la|_{\ker l}, \del|_{\ker l})$, in both the PID and non-PID case, to get 
       \begin{align*}g_X(M, \la, \del) & = 1 + g_{\HH}(M, \la) + g_{\HH}(\ker \del, \la|_{\ker \del}) \\
       & \le 2+ g_{\HH}(\ker l, \la|_{\ker l}) + g_{\HH}(\ker \del, \la|_{\ker \del}) 
       \end{align*}
       using again \eqref{eq:kerl}, and
       \begin{align*}
         g_X(\ker l, \la|_{\ker l}, \del|_{\ker l}) & = 1 + g_{\HH}(\ker l, \la|_{\ker l}) + g_{\HH}(\ker l \cap \ker \del, \la|_{\ker l \cap \ker \del})\\
         & \ge g_{\HH}(\ker l, \la|_{\ker l}) + g_{\HH}(\ker \del, \la|_{\ker \del}).  
       \end{align*}
       where we applied Proposition \ref{prop:properties hyperbolics}(iv) to 
       $l|_{\ker \del}$ to get the inequality 
        $$g_{\HH}(\ker l \cap \ker \del, \la|_{\ker l \cap \ker \del}) \geq g_{\HH}(\ker \del, \la|_{\ker \del})-1.$$ 
       Combining the above two inequalities gives the result.
   \end{proof}

%% file: arccomplexes.tex
In this section, we define algebraic analogues of arc complexes and study their connectivity. Our goal is to define an algebraic version of the disordered arc complex of \cite{HVW} and show that it is highly connected. The proof of connectivity will  
be parallel to the proof in the geometric case, using a complex of ``non-separating arcs" along the way. We will see that algebraic arc complexes are closely related to posets of unimodular vectors, classically used to study linear groups. And in Section~\ref{sec:cancell}, we will show that the algebraic disordered arc complex identifies with the destabilization complex for stabilization by $X$, in the sense of \cite{RWW}. 

We will work with three types of simplicial objects:  simplicial complexes, semi-simplicial sets and nerves of posets. We start with a short section detailing how connectivity results can be passed between such types of objects in a specific situation. This will allow us to combine existing results proved in the context of posets, proof techniques developed in the context of simplicial complexes, while also using vertex ordering properties of semi-simplicial sets. 

\subsection{Simplicial complexes, semi-simplicial sets and posets}

Recall that a simplicial complex $\mathcal{S}$ has a set of vertices, and $p$-simplices are given as unordered collections of $p+1$ distinct vertices. On the other hand, in a semi-simplicial set $X=X_\bullet$, there is a set $X_p$ of $p$-simplices for every $p$, and face maps $d_i:X_p\to X_{p-1}$ for each $i=0,\dots,p$. Applying face maps repeatedly associates to any $p$-simplex $\s\in X_p$ an ordered collection of $p+1$ not necessarily distinct vertices in $X_0$. And this ordered collection of vertices need not determine $\s$. 

Any simplicial complex $\mathcal{S}$ can be given a semi-simplicial set structure by picking a total ordering of its set of vertices, defining the face map $d_i:\mathcal{S}_p\to \mathcal{S}_{p-1}$ to be the map forgetting the $i$th vertex. 
To $\mathcal{S}$, we can also associate the semi-simplicial set $\mathcal{S}^{ord}$, with the same vertices as $\mathcal{S}$ and with a $p$-simplex for each ordered sequence $(a_0,\dots,a_p)$ of vertices of  $p$-simplices $\{a_0,\dots,a_p\}$  of $\mathcal{S}$. So $\mathcal{S}^{ord}$ has $(p+1)!$ simplices for every $p$-simplex of $\mathcal{S}$. 

To $\cS$, we can likewise associate two posets: the poset $s\cS$ of simplices of $\cS$, oredered by inclusion, and 
 the poset $P(\mathcal{S})$ of  sequences $(a_0,\dots,a_p)$ of vertices of $\mathcal{S}$ with the property that $\{a_0,\dots,a_p\}$ is a $p$-simplex of $\mathcal{S}$, ordered by refinement. That is, for every $p$-simplex of $\cS$, there is one element in the poset $s\cS$ and $(p+1)!$ elements in the poset $P(\cS)$. 

\begin{lem}
    Let $\mathcal{S}$ be a simplicial complex, with its associated semi-simplicial set $S^{ord}$ and poset $P(\mathcal{S})$ as above. Then $P(\mathcal{S})\cong s \cS^{ord}$ is isomorphic to the poset of simplices of $S^{ord}$. Moreover, if $\s=\{a_0,\dots,a_p\}$ is a $p$-simplex of $\mathcal{S}$, and $P(\mathcal{S})_{a_0,\dots,a_p}$ denotes the subposet of sequences in $P(\mathcal{S})$ that end with $a_0,\dots,a_p$, then $$P(\mathcal{S})_{a_0,\dots,a_p}\cong s\Big(\big(\link_{\mathcal{S}}(\s)\big)^{ord}\Big)$$
    is isomorphic to the poset of simplices of the order complex of the link  of $\s=\{a_0,\dots,a_p\}$ in $\mathcal{S}$. 
\end{lem}

\begin{proof}
The statement follow directly from the definitions. 
\end{proof}

The geometric realization of a simplicial complex or semi-simplicial set has a copy of the standard $p$-simplex $\Delta^p$ for each $p$-simplex, glued along face inclusion, and the geometric realization of a poset has a $p$-simplex $\Delta^p$ for each length $p$ chain $a_0<\dots<a_p$ in the poset. When we talk about connectivity properties of such objects, we will always mean the connectivity properties of their geometric realization. 
Recall that the geometric realization of the poset of simplices $s\cS$ identifies with the barycentric subdivision of $\cS$. In particular, $\cS$ and $s\cS$ have the same connectivity.

\medskip

A simplicial complex $\mathcal{S}$ is called {\em weakly Cohen-Macaulay} (or {\em wCM} for short) {\em of dimension $n$} if it is $(n-1)$-connected, and the link of any $p$-simplex $\s$ in $\mathcal{S}$ is at least $(n-p-2)$-connected. 
Proposition 2.14 of \cite{RWW} states that when $\cS$ is wCM of dimension $n$, then $\cS^{ord}$ is necessarily at least $(n-1)$-connected. Building on this result, the following proposition shows that, under the wCM property, the connectivity properties of $\mathcal{S}$, $\mathcal{S}^{ord}$ and $P(\mathcal{S})$ are closely related: 

\begin{prop}\label{prop:SCSS}
    Let $\mathcal{S}$, $\mathcal{S}^{ord}$ and $P(\mathcal{S})$ be as above. Then the following are equivalent: 
    \begin{enumerate}[(i)]
        \item $\mathcal{S}$ is wCM of dimension $n$;
        \item $\mathcal{S}^{ord}$ is $(n-1)$-connected and $\big(\link_{\mathcal{S}}(\{a_0,\dots,a_p\})\big)^{ord}$ is $(n-p-2)$-connected for every $p$-simplex $\{a_0,\dots,a_p\}$  of $\mathcal{S}$; 
         \item $P(\mathcal{S})$ is $(n-1)$-connected and $P(\mathcal{S})_{a_0,\dots,a_p}$ is $(n-p-2)$-connected for every $p$-simplex $(a_0,\dots,a_p)$  of $\mathcal{S}^{ord}$.
    \end{enumerate}
\end{prop}

\begin{proof}
   First note that (ii) and (iii) are equivalent  since we have  $P(\mathcal{S})\cong s\mathcal{S}^{ord}$ and $P(\mathcal{S})_{a_0,\dots,a_p}\cong s\link_{\mathcal{S}}(\{a_0,\dots,a_p\})^{ord}$ by the lemma, and since the realization of poset of simplices of a simplicial object identifies with the barycentric subdivision of the realization, and hence are homeomorphic. 

   We will show that (i) and (ii) are equivalent. Assuming (i), \cite[Prop 2.14]{RWW} gives that $
   \mathcal{S}^{ord}$ is $(n-1)$-connected. Now note that, if $\s,\tau$ are simplices in $\cS$ with $\tau\in \link_{\mathcal{S}}(\s)$, then $\link_{\link(\s)}(\tau)\cong \link_{\mathcal{S}}(\tau *\s)$. It follows that if $\mathcal{S}$ is wCM of dimension $n$, then the link of any $p$-simplex $\s$ in $\mathcal{S}$ is wCM of dimension $n-p-1$. Indeed, if $\tau$ is a $q$-simplex, then $\tau*\s$ is a $(p+q+1)$-simplex, and the assumption on $\cS$ implies that $\link_{\link(\s)}(\tau)$ is $(n-(p+q+1)-2)=((n-p-1)-q-2)$-connected, as needed. So we can also apply \cite[Prop 2.14]{RWW} to $\link_{\mathcal{S}}(\{a_0,\dots,a_p\})$ to deduce that $\link_{\mathcal{S}}(\{a_0,\dots,a_p\})^{ord}$ is $(n-p-2)$-connected, showing that (i) implies (ii).  
   
   Assume now that (ii) holds. Forgetting the orderings induces a map on realizations $|\mathcal{S}^{ord}|\to |\mathcal{S}|$, and picking a total ordering of the vertices of $\mathcal{S}$ defines a section of that map. It follows that $\mathcal{S}$ is at least as highly connected as $\mathcal{S}^{ord}$. Likewise,  $\link_{\mathcal{S}}(\s)$ is at least as highly connected as $\link_{\mathcal{S}}(\s)^{ord}$. Hence (ii) implies that $\mathcal{S}$ is wCM of dimension $n$. 
\end{proof}


Finally, we give an adaptation of the simplicial approximation result for simplicial complexes, to semi-simplicial sets of the form $\cS^{ord}$:

\begin{prop}\label{prop:approx}
    Let $S$ be a simplicial complex and $Y\subset \cS^{ord}$ a sub-semi-simplicial set of the order complex of $S$. 
    Let $K$, $L$ be finite simplicial complexes,
with $L$ a subcomplex of $K$. Let $f : |K|\to |\cS^{ord}|$ be a continuous map such that the restriction $f|_L$
is a semi-simplicial map from $L$ to $Y$ for a semi-simplicial set structure on $L$. Then there exists a relative subdivision $(K_r,L)$ of $(K,L)$, a semi-simplicial set structure on $K_r$ extending that of $L$ and
a semi-simplicial map $g : K_r\to \cS^{ord}$ such that $g|_L = f|_L$ and $g$ is homotopic to $f$ keeping $L$ fixed.
\end{prop}

\begin{proof}
Consider $(\hat f, f):(|K|,|L|)\to (|\cS|^{ord},|Y|)\cong (|s \cS^{ord}|,|s Y|)$, with $s\cS^{ord}$  and $sY$ denoting the poset of simplices, or equivalently barycentric subdivisions, of $\cS^{ord}$ and $Y$. The barycentric subdivision of any simplicial object has the property that any $p$-simplices is determined by its $p+1$ (distinct) vertices. Hence, forgetting the ordering of the vertices, it can be considered as a simplicial complex, and we can use PL approximation of \cite{zeeman} for simplicial complexes. This gives a subdivision $K_r$ of $K$  (relative to $L$) and a simplicial map 
$$\hat f':K_r\rar s \cS^{ord}$$
restricting to $f$ on $L$ and such that $(\hat f',f)\simeq (\hat f,f)$. 

Consider the forgetful map $u:s \cS^{ord}\to \cS^{ord}$ that takes a vertex of $s \cS^{ord}$, that is a simplex $(a_0,\dots,a_p)$ of $\cS^{ord}$, to the last vertex $a_p$. This defines a simplicial map since a simplex of $s \cS^{ord}$ is a chain of simplices of $\cS^{ord}$, and the last vertices of such a chain necessarily define together a simplex of $\cS$ (of possibly smaller dimension). Moreover, using linear interpolation in the simplices, one can check that, on realizations and after identifying $|s\cS^{ord}|\cong |\cS^{ord}|$, this map is homotopic to the identity.  

Finally we can define a semi-simplicial set structure on $K_r$ extending that of $L$ as follows: Let $\s=\{v_0,\dots,v_q\}$ be a $q$-simplex of $K_r$. Because any $p$-simplex of $\cS^{ord}$ has $p+1$ distinct vertices, the image of $\s$ under the composed map 
$$g: K_r\xrightarrow{\hat f'} s \cS^{ord}\xrightarrow{\ u\ } \cS^{ord}$$ is determined by the collection of images $\{g(v_{0}),\dots,g(v_{q})\}$, that forms a $p$-simplex of $\cS$ for some $p\le q$, together with the map $\theta_\s:\{v_0,\dots,v_q\}\to \{0,\dots,p\}$ that records where the vertices are mapped. We need to pick, for each such simplex $\s$, an ordering of its vertices that makes $\theta_\s$ order preserving. We start by picking, arbitrarily extending that of $L$, for each vertex $v$ of $\cS^{ord}$, an ordering of the vertices of the subcomplex of $K_r$ of simplices mapped to $v$. Now for a general simplex $\s$ with $g(\s)=(v_0,\dots,v_p)$, we can write $\s=\s_{v_0}*\dots*\s_{v_p}$ where $\s_{v_i}$ is the face of $\s$ mapped to $v_i$. We order the vertices of $\s$ using on each $\s_{v_i}$ the ordering already chosen in the previous step, extending it to the total ordering induced by the ordering of the vertices $v_i$.  
\end{proof}

\subsection{Complexes of non-separating arcs}

Let $(M,\la,\del)$ be a formed space with boundary.
Recall that a collection $v_1,\dots,v_n\in M$ is unimodular if it generates a split $R^n$--summand. 
In what follows, we will write $\la(a,b)=a\bullet b$ for $\la$ the form in the formed space $a$ and $b$ are considered in.

\begin{Def}\label{def:arc}
    Let $M=(M,\la,\del)$ be a formed space with boundary. 
    \begin{enumerate}[(i)]
    \item  An {\em arc} in $M$ is an element $a\in M$ such that $\del(a)=1$.
        \item An arc $a$ is {\em non-separating} if $\{a\bullet -, \del\}$ is unimodular in $M^\vee=\operatorname{Hom}(M,R)$. 
    \end{enumerate}
\end{Def}

Note that an arc is automatically unimodular in $M$ since we can split $M=\ker \del \oplus R\langle a \rangle$ as $R$-modules. 
On the other hand, $a\bullet -$ is not necessarily unimodular in $M^\vee$; in fact it can be the zero map, see e.g.~the arc $x-y+z$ in $X^{\# 3}$ in the proof of Proposition~\ref{prop:X^n}. 
If $\del$ is assumed surjective, it is by definition unimodular in $M^\vee$, so the non-separating condition is equivalent to requiring that $a\bullet -$ is unimodular in $(\ker\del)^\vee$. 

If $a$ is a geometric arc in a bidecorated  surface $(S,I_0,I_1)$ going from $I_0$ to $I_1$, then its homology class $[a]$ defines an arc in the above sense in the associated formed space with boundary $(M,\la,\del)$ described in Section \ref{sec:surfaces}. And if $a$ is a non-separating arc, then that precisely means that there exists a curve $c$ in $S$ ``connecting both sides of $a$'', i.e. an element $c$ with $\del c=0$ and $a\bullet c=1$. So a geometric non-separating arc in a bidecorated  surface will likewise define an algebraic non-separating arc. 

\begin{Def}
     \begin{enumerate}[(i)]
\item The {\em algebraic arc complex} $\Aa^{alg}(M,\del)$ is the simplicial complex whose vertices are arcs in $M$, and where a collection of arcs $\{a_0,\dots,a_p\}$ defines a $p$--simplex if the vectors together form a unimodular sequence.
    \item The {\em algebraic non-separating arc complex} $\B(M,\la,\del)$ has vertices non-separating arcs, and 
a collection of arcs $\{a_0,\dots,a_p\}$ forms a $p$-simplex if  $\{a_0\bullet -,\dots,a_p\bullet -,\del\}$ is unimodular in $M^\vee$, i.e.~the arcs are jointly non-separating.
\end{enumerate}
\end{Def} 

As indicated above, the last condition has a geometric meaning when the arcs come from (pairwise disjoint) geometric arcs in a surface. It actually precisely corresponds to being non-separating by a Mayer-Vietoris argument. 
It can equivalently be stated as: $\{a_0\bullet -,\dots,a_p\bullet -\}$ is unimodular in $(\ker \del)^\vee$, provided $\del$ is itself surjective. 
Observe also that the condition in (ii) implies that of (i) so there is an inclusion $\B(M,\la,\del) \subset \Aa^{alg}(M,\del)$. 

Our first goal in this section is to prove a connectivity result for $\B(M,\la,\del)$. It will folllow from a  result about certain posets of unimodular sequences, as studied in particular by van der Kallen, Maazen, Charney, see e.g.~\cite{vdKL,Maazen,Cha84}. We will not actually need the connectivity of $\Aa^{alg}(M,\la)$, but it will be a direct consequence of these other results. We will work with the following posets:

\begin{Def}
    Let $\Uu(M)$ denote the poset of ordered non-empty unimodular sequences in $M$. For $N\le M$ a submodule and $m\in M$ an element, we define $\Uu(M,m+N)$ to be the subposet of $\Uu(M)$ of unimodular sequences in $M$ of elements of the form $m+n$ with $n\in N$.
\end{Def}

Our first connectivity results will be formulated in terms of the stable rank $sr(R)$ of Definition~\ref{def:sr} and following relative rank  $$r(M,N):= \max\{\rk(L): L \subset N \text{ is a direct summand of } M\}.$$ 

\begin{thm}\label{connectivity of U}
Let $M$ be a finitely generated free module.
The poset $\Uu(M,m+N)$ is $(r(M,N)-sr(R)-1)$--connected, and for $v_1,\dots,v_k\in \Uu(M,m+N)$, the subposet $\Uu(M,m+N)_{v_1,\dots,v_k}$ of sequences ending in $(v_1,\dots,v_k)$ is $(r(M,N)-sr(R)-k-1)$--connected. 
  \end{thm}

 This result that can be found in \cite[Thm 6.5]{S22a} in the case $R=\Z$, and is based on earlier results of \cite[2.6]{vdKL},\cite[Thm 2.4]{Fri17}. A complete proof is given in Appendix~\ref{appA}. 

 \medskip

Let $\X(M)$ and $\X(M,m+N)$ denote the corresponding simplicial complexes of unimodular vectors: vertices of $\X(M)$ are unimodular vectors, and a collection of unimodular vectors forms a $p$-simplex in $\X(M)$ if they together form a unimodular sequence. And $\X(M,m+N)$ is the full subcomplex of $\X(M)$ on unimodular vectors of the form $m+n$ for $n\in N$. 




\smallskip

Theorem~\ref{connectivity of U} has the following consequence:

  \begin{cor}
Both $\X(M,m+N)^{ord}$ and $\X(M,m+N)$ are $(r(M,N)-sr(R)-1)$--connected. Moreover $\X(M,m+N)$ is wCM of dimension $r(M,N)-sr(R)$. 
    \end{cor}

\begin{proof}
  The result thus follows from combining Theorem~\ref{connectivity of U} with Proposition~\ref{prop:SCSS} since the poset $\Uu(M,m+N)$ identifies with the poset of sequences $P(\X(M,m+N))$. 
%
  \end{proof}

\begin{prop}\label{prop:Aforfun}
Let $M$ be a finitely generated free module of rank $\rk(M)$ and $\del:M\to R$ a linear map. If $\del$ is surjective, then 
the algebraic arc complex $\Aa^{alg}(M,\del)$ is $(\rk(M)-sr(R)-2)$--connected. 
\end{prop}
  
  \begin{proof}
  Since $\del$ is surjective, we have that $\Aa^{alg}(M,\del)$ is non-empty, proving the result for $\rk(M)-sr(R)-2\le -1$. 
 Hence we can assume that $\rk(M) \ge sr(R)+2$. 
  
    The  complex of $\Aa^{alg}(M,\del)$ identifies with $\X(M,a+\ker\del)$ for $a\in M$ an arc. 
    Since $\del$ is surjective then $\ker\del \subset M$ is a summand, so we have that $\tM(M,\ker\del)=\rk(\ker\del)$.
    Moreover, $R^{\rk(M)} \cong M \cong R \oplus \ker \del$ as $R$-modules, since $M$ is finitely generated free. Then cancellation, \cite[Proposition 2.7]{Fri17}, gives that $\ker \del \cong R^{\rk(M)-1}$ since $\rk(M) \ge sr(R)+1$. 
    The result then follows from Theorem~\ref{connectivity of U}.
    \end{proof}

 The  connectivity of the complex of geometric non-separating arcs is usually deduced from that of the complex of all arcs. In the algebraic case, while there is a map $\B(M,\la,\del)\to \Aa^{alg}(M,\la)$,
    it is more convenient to deduce the connectivity of $\B$ from that of an arc complex closer to it and whose connectivity is also determined by Theorem~\ref{connectivity of U}. We describe this complex now. 

    \smallskip

    Let $\la^\vee(\ker\del)$ be the submodule of $(\ker\del)^\vee$ of maps that can be written as $c\bullet -$ for $c\in \ker \del$. Fix some arc $m$. 
There is a map
$$\B(M,\la,\del)\to \X((\ker\del)^\vee,m \bullet - +\la^\vee(\ker\del))$$
associating to $a$ the map $a\bullet -: \ker\del \to R$. 

For $M=(M,\la,\del)$, we write 
$$\tM(M,\la,\del):=r\big((\ker\del)^\vee,\la^\vee(\ker\del)\big).$$

\begin{prop}\label{prop:connectivity B} Let $(M,\la,\del)\in \Fd$. Then 
$\B(M,\la,\del)$ and $\B(M,\la,\del)^{ord}$ are $(\tM(M,\la,\del)-sr(R)-1)$--connected. 
\end{prop}

Recall from \cite{HatWah} that a simplicial complex $\mathcal{X}$ is a {\em complete join complex} over another simplical complex $\mathcal{S}$ if there is a surjective map $\pi:\mathcal{X}\to \
\mathcal{S}$ and vertices $x_0,\dots,x_p$ of $\mathcal{X}$ form a $p$-simplex if and only if the projections $\pi(x_0),\dots,\pi(x_p)$ form a $p$-simplex in $\mathcal{S}$. The proof will use that the simplicial complex $\B$ can be seen as a join complex over the simplicial complex of unimodular sequences just described. 

\begin{proof}
The forgetful map $\B(M,\la,\del)\to \X((\ker\del)^\vee,m \bullet-+\la^\vee(\ker\del))$ taking an arc $a$ to the map $a\bullet -: \ker\del\to R$ exhibits  $\B(M,\la,\del)$ as a complete join complex over $\X((\ker\del)^\vee,m\bullet- +\la^\vee(\ker\del))$. 
Indeed, a vertex of  $\B(M,\la,\del)$ is a vertex $\phi$ of $\X((\ker\del)^\vee,m\bullet- +\la^\vee(\ker\del))$ together with the choice of an arc $a\in M$ such that $\phi=a\bullet -$, and we can write $a=m+c$ for some $c\in \ker\del$. Moreover,  the condition for arcs $a_0,\dots,a_p$ to form a $p$-simplex of $\B(M,\la,\del)$  is determined by whether their images form a $p$--simplex in $\X((\ker\del)^\vee,m \bullet-+\la^\vee(\ker\del))$. \\
Now  $\X((\ker\del)^\vee,m \bullet-+\la^\vee(\ker\del))$ is wCM of dimension $r((\ker\del)^\vee,\la^\vee(\ker\del))-sr(R)$. It then follows from \cite[Prop 3.5]{HatWah} that $\B(M,\la,\del)$ is also wCM of dimension $r((\ker\del)^\vee,\la^\vee(\ker\del))-sr(R)$, and in particular $(r((\ker\del)^\vee,\la^\vee(\ker\del))-sr(R)-1)$--connected. Finally by \cite[Prop 2.14]{RWW}, the ordered complex $\B(M,\la,\del)^{ord}$ has the same connectivity. 
\end{proof}

The following result will allow us to rewrite the above connectivity bound in terms of the arc genus. 

\begin{prop} \label{prop: t}
Let $(M,\la,\del)$ be a formed space with boundary such that $g_X(M,\la,\del) \ge 2usr(R)+2$, or $R$ is a PID. Then $\tM(M,\la,\del)\ge g_X(M,\la,\del)-2$. 
\end{prop}

\begin{proof}
The result is trivial if $g_X(M,\la,\del)=0$, so we can assume without loss of generality that $g_X(M,\la,\del)\ge 1$, i.e.~that $\del$ is unimodular. 
By Proposition~\ref{prop:formula for gX}(b)((i),(iii)) we have 
$$g_X(M,\la,\del) = g_{\HH}(M,\la)+g_{\HH}(\ker\del,\la|_{\ker\del})+1.$$
By Proposition \ref{prop:formula for gX}(a), the inequality $g_X(M,\la,\del) \ge 2usr(R)+2$ implies $g_{\HH}(M,\la) \ge usr(R)+1$, and hence by Proposition \ref{prop:properties hyperbolics}(iv) we have 
$$2 g_{\HH}(\ker\del,\la|_{\ker\del})\ge g_{\HH}(M,\la)+g_{\HH}(\ker\del,\la|_{\ker\del})-1.$$ 

Pick a hyperbolic subspace $L \subset \ker\del$ such that $(L,\la|_L) \cong \HH^{g_{\HH}(\ker\del,\la)}$. 
Then, $\la^\vee(L)$ is a direct summand of $(\ker\del)^\vee$ because of the orthogonal decomposition $\ker\del=L \oplus L^{\perp}$ (see Lemma \ref{lem:splitH}(i)). 
Thus $\tM(M,\la,\del)=r\big((\ker\del)^\vee,\la^\vee(\ker\del)\big) \geq \rk(\lambda^\vee(L))=\rk(L)=2g_{\HH}(\ker\del,\la)$, which proves the result since the above computations give that $2g_{\HH}(\ker\del,\la)\geq g_X(M,\la,\del)-2$.
\end{proof}

The above results assemble to the following result: 

\begin{cor} \label{cor: connectivity B}
 Let $(M,\la,\del)$ be a formed space with boundary with $g_X(M,\la,\del) \ge 2usr(R)+2$, or such that $R$ is a PID, 
 then $\B(M,\la,\del)$ and $\B(M,\la,\del)^{ord}$ are $(g_X(M,\la,\del)-sr(R)-3)$--connected.  
\end{cor}

\subsection{Cutting arcs} \label{sec: cutting arcs}

Following the geometric argument in \cite{HVW}, we need to understand the effect of ``cutting a simplex of arcs" in $\B(M,\la,\del)$. 

Recall from Lemma~\ref{lem:splitH} that for $N \subset M$, we denote 
$M\minus N=\{m\in M\ |\ \la'(n,m)=0 \ \forall n\in N\}$, where $\la'(n,m)=\la(n,m)+\del n\del m$. 
Given a $p$-simplex $\sigma=\{a_0,\cdots,a_p\}$, we will write $M\minus \s=M\minus N$ for $N=\langle a_0,\dots,a_p\rangle$ the subspace generated by the arcs. Given that $\del(a_i)=1$ for each $i$, we can rewrite this subspace as 
$$M \minus \sigma:= \bigcap_{i=0}^{p}{\ker(\del+a_i \bullet-)} \subset M.$$
We call $(M\minus \sigma, \la|_{M\minus \sigma},\del|_{M\minus \sigma})$ the \textit{cut formed space}.

When $a_0,\dots,a_p$ are homology classes of geometric arcs in a bidecorated surface, 
the cut formed space $M\minus \s$ identifies with the relative first homology group associated to the  surface obtained by cutting the corresponding collection of (geometric) arcs, for an appropriate choice of marked intervals $I_0,I_1$ in the cut surface.

The main result we will need about cut formed spaces is the effect cutting arcs has on the rank and arc-genus. 

\begin{lem} \label{lem:cut form properties}
Let $\sigma=\{a_0,\dots,a_p\}$ be a $p$-simplex in $\B(M,\la,\del)$. 
\begin{enumerate}[(i)]
    \item If either $\rk(M) \ge sr(R)+p+1$ or $R$ is a PID then 
$$\rk(M \minus \sigma)= \rk(M)-(p+1).$$
\item If either $g_X(M,\la,\del) \ge 2 usr(R)+2p+4$ or $R$ is a PID then $$g_X(M \minus \sigma, \la|_{M \minus \sigma},\del|_{M \minus \sigma}) \geq g_X(M,\la,\del)-(2p+1).$$
\end{enumerate}

\end{lem}

\begin{proof}
  By definition $\{a_0 \bullet -, \cdots, a_p \bullet -, \del\}$ is unimodular in $M^\vee$ so $\{\del+a_0 \bullet -, \cdots, \del+ a_p \bullet -\}$ is also unimodular, which implies that $\rk(M \minus \sigma)= \rk(M)-(p+1)$ which applies under the current assumption by \cite[Prop 2.7]{Fri17}
  , proving (i) when $R$ is not a PID.
  For $R$ a PID we use the same argument noting that $M \setminus \sigma$ is a summand of the free finitely generated $R$-module $M$ with quotient isomorphic to $R^{p+1}$, so it has to be free of rank $\rk(M)-p-1$ by the classification theorem of finitely generated modules over PIDs.

\smallskip

We will show that $g_X(M \minus \sigma, \la|_{M \minus \sigma},\del|_{M \minus \sigma}) \geq g_X(M,\la,\del)-(2p+1)$ by induction on $p$. 
Let us begin by assuming the case $p=0$ and finish the induction, and then we will study the case $p=0$ in detail. 

Let $\tau=\{a_0,\dots, a_{p-1}\}$, then $\tau$ is a $(p-1)$-simplex in $\B(M,\la,\del)$ so by induction we have $g_X(M \minus \tau, \la|_{M \minus \tau},\del|_{M \minus \tau}) \geq g_X(M,\la,\del)-(2p-1)$. We want to apply Corollary \ref{cor R of a kernel} to $M\minus \tau$ with $\del=\del|_{M\minus \tau}$ and $l=(\del+a_p \bullet -)|_{M \minus \tau}$. 
Since $\{a_0 \bullet -, \cdots, a_p \bullet -, \del\}$ is unimodular in $M^\vee$ then so is $\{\del+a_0 \bullet-, \cdots, \del+a_{p-1} \bullet-, \del+ a_p \bullet -, \del\}$. 
Therefore $\{(\del+a_p \bullet -)|_{M \minus \tau}, \del|_{M \minus \tau}\}$ is unimodular in $(M \minus \tau)^\vee$. Moreover, either $R$ is a PID, or we have 
$$g_X(M\minus \tau,\la|_{M\minus\tau},\del|_{M\minus\tau})\ge g_{X}(M,\la,\del)-(2(p-1)+1)\ge 2usr(R)+2p+4-2p+1=2usr(R)+5,$$
using induction for the first inequality and the assumption for the second one. 
Hence we can apply Corollary \ref{cor R of a kernel}, and it gives the required result to finish the induction step.

\smallskip

Now let us consider the case $p=0$, so that $\sigma=\{a\}$. 
Here we cannot apply the above induction step since that would only show that $g_X(M \minus\sigma, \la|_{M \minus \sigma},\del|_{M \minus \sigma}) \geq g_X(M,\la,\del)-2$ instead of the required improved bound $g_X(M \minus \sigma, \la|_{M \minus \sigma},\del|_{M \minus \sigma}) \geq g_X(M,\la,\del)-1$.
To simplify the notation let us denote $\tilde{M}=M \minus\sigma=\ker(\del + a\bullet -)$, $\tilde{\la}=\la|_{M \minus \sigma}$ and $\tilde{\del}=\del|_{M \minus \sigma}$ for the rest of this proof. 

By definition, $a$ defines a map $X \to (M,\la,\del)$ so by Lemma \ref{lem:splitH}(ii) we have an isomorphism $(M,\la,\del) \cong (\tilde{M},\tilde{\la},\tilde{\del}) \# X$.

Since $\{a \bullet -, \del\}$ is unimodular in $M^\vee$, then $\tilde{\del}$ is unimodular in ${\tilde{M}}^\vee$. 
Without loss of generality $\rk(M) \geq 2$  (otherwise the result we want to show is trivial), and hence $\rk(\tilde{M}) \geq 1$ by the first part of this lemma, since the assumption of (ii) is stronger than that of (i) because $2usr(R)+2\ge sr(R)+1$.

By Proposition~\ref{prop:formula for gX}(a) and our assumption on $g_X(M,\la,\del)$, we have that $g_{\HH}(M,\la,\del)\ge usr(R)+2$, and (b)(ii) of the proposition gives that 
$$g_X(M,\la,\del) = 1+g_{\HH}(M,\la)+g_{\HH}(\ker\del,\la|_{\ker\del}).$$
Applying now Proposition~\ref{prop:properties hyperbolics}(iv) with $l=\del + a\bullet -$, we also get that $g_{\HH}(\tilde{M},\tilde{\la})\ge usr(R)+1$, so we can also apply Proposition~\ref{prop:formula for gX}(b)(ii) to $\tilde{M}$, which gives 
$$g_X(\tilde{M},\tilde{\la},\tilde{\del})=1+g_{\HH}(\tilde{M},\tilde{\la})+g_{\HH}(\ker \tilde{\del},\tilde{\la}|_{\ker \tilde{\del}}).$$

We are left to compare the terms in the formulas for $g_X(M,\la,\del)$ and $g_X(\tilde{M},\tilde{\la},\tilde{\del})$.
We first show that there is an isomorphism of formed spaces 
\begin{equation}\label{eq:kerdel}
    f: (\ker \del, \la|_{\ker \del})\  \xrightarrow{\ \cong\ } \ (\tilde{M},\tilde{\la})
\end{equation}
so that they both have the same $\HH$-genus. (This is an algebraic version of the geometric observation in Remark \ref{rem:HorX}.)  
Indeed, the map $f$ can be defined 
by $f(u)=u -(a \bullet u) a$, with inverse $g: \tilde{M} \rightarrow \ker\del$ defined by $g(v)=v-\del(v) a$, where one checks that both maps preserve the forms.

Secondly, the unimodularity of $\{a \bullet -, \del\}$ in $M^\vee$ implies that there is $x \in M$ such that $\del x=0$ and $a \bullet x = 1$. 
We then claim that the forms $(\ker \tilde{\del},\tilde{\la}|_{\ker \tilde{\del}})$ and $(\langle a, x \rangle^\perp, \la|_{\langle a, x \rangle^\perp})$ are isomorphic, so they have the same $\HH$-genus. 
Indeed, the maps $f': \ker \tilde{\del} \rightarrow \langle a, x \rangle^\perp$, $f'(u)=u+(x \bullet u) a$, and $g':\langle a, x \rangle^\perp \rightarrow \ker \tilde{\del}$, $g'(v)=v-(\del v) a$, give a pair of inverse isomorphisms preserving the forms. 

Since $\langle a, x \rangle \subset M$ is isomorphic to $\HH$ then Proposition \ref{prop:properties hyperbolics}(i) implies that $g_{\HH}(\langle a, x \rangle^\perp,\la|_{\langle a, x \rangle^\perp})=g_{\HH}(M,\la)-1$, which finishes the proof. 
\end{proof}

\subsection{Disordered arc complexes}

Geometric arcs in a surface with the same endpoints $b_0,b_1$ in the boundary and with disjoint interiors, are naturally ordered at $b_0$ and $b_1$. This leads to two natural subcomplexes of the complex of all (non-separating) collections of such arcs: the one where the orderings of the arcs at $b_0$ and $b_1$ agree, and the one where the ordering is reversed. The latter is the first one that was considered in the literature, and it goes under the name the {\em ordered arc complex}. The first one is the one whose algebraic version will be relevant to us, and, to distinguish it from its ``opposite" version, it goes under the name {\em disordered arc complex}. (See \cite[Sec 2]{HVW}). We define now its algebraic analogue:

\begin{Def}\label{def:disordered}
  The {\em algebraic disordered arc complex} $\D(M,\la,\del)$ is the simplicial subcomplex of $\B(M,\la,\del)$,  with the same vertices and where $\{a_0,\dots,a_p\}$ forms a $p$-simplex in $\D(M,\la,\del)$ if it is a $p$-simplex of $\B(M,\la,\del)$ and its vertices can be ordered in such a way that $a_i\bullet a_j=1$ for all $i<j$. 

  More generally, for $\s=\langle a_0,\dots,a_p\rangle$ a simplex of $\B(M,\la,\del)$, we define $\D_{<\s}(M,\la,\del)$ to be the simplicial subcomplex of $\link_{\B(M,\la,\del)}(\sigma)$ on the vertices $b$ such that $b\bullet a_i=1$ for all $i=0,\dots, p$, with $b_0,\dots,b_q$ forming a $p$-simplex if they can be ordered in such a way that $b_i\bullet b_j=1$ for all $i<j$. 
 \end{Def} 

Note that the vertices of a simplex of  $\D(M,\la,\del)$ are canonically ordered if $char(R) \neq 2$, so the complex is actually naturally a subcomplex of $\B(M,\la,\del)^{ord}$ in that case. Also, a vertex $b$ in the link of a simplex $\s$ in $\D(M,\la,\del)$ will have a position with respect to $\s$ in that case; as apparent in the above definition, we will be particularly interested in vertices ``below $\s$", i.e.~such that $b\bullet a_i=1$ for all vertices $a_i$ of $\s$. These different positions of vertices in the link do not exist in characteristic 2, making the characteristic 2 case actually simpler, as we will see in the proof of the theorem below.

\begin{lem}\label{lem:Dconnect}
Let $\s$ be a simplex of $\B(M,\la,\del)$.  
Then $\D_{<\s}(M,\la,\del) \cong \D(M \minus \s, \la|_{M \minus \s},\del|_{M \minus \s})$.  
\end{lem}

\begin{proof}
    We will check both inclusions. Let us write $\sigma=(a_0,\dots,a_p)$. 
    
    Let $\tau=(b_0,\dots,b_q)$ be a $q$-simplex in $\D_{<\sigma}(M,\la,\del)$, then $(b_0,\dots,b_q,a_0,\dots,a_p)$ is a $(p+q+1)$-simplex in $\B(M,\la,\del)$. 
    By definition, $b_i \bullet a_j=1=\del(b_i)$ for all $i,j$ so $b_i \in M \minus \sigma$ is an arc in the cut formed space for all $i$. 
    Unimodularity of $\{\del,b_0\bullet-,\dots,b_q\bullet-,a_0\bullet-,\dots,a_p\bullet-\}$ in $M^\vee$ implies unimodularity of $\{\del|_{M \minus \sigma},b_0\bullet-,\dots,b_q\bullet-\}$ in $(M \minus \sigma)^\vee$. 
    Thus, $\{b_0,\dots,b_q\}$ is a $q$-simplex in $\B(M \minus \sigma,\la|_{M \minus \sigma},\del|_{M \minus \sigma})$. 
    Finally, the condition $b_i \bullet b_j=1$ for $i<j$ shows that $\tau$ is a $q$-simplex in $\D(M \minus \sigma,\la|_{M \minus \sigma},\del|_{M \minus \sigma})$.

    Conversely, if $\tau=(b_0,\dots,b_q)$ is a $q$-simplex in $\D(M \minus \sigma,\la|_{M \minus \sigma},\del|_{M \minus \sigma})$ then $\del(b_i)=1$ for all $i$, so each $b_i \in M$ is an arc too, and $\{\del|_{M \minus \sigma}, b_0\bullet-,\cdots, b_q\bullet-\}$ is unimodular in $(M\minus \sigma)^\vee$, which implies unimodularity of $\{\del,b_0\bullet-,\cdots,b_q\bullet-,a_0\bullet-,\cdots,a_p\bullet-\}$ in $M^\vee$ since $\{\del,a_0\bullet-,\cdots,a_p\bullet-\}$ is unimodular in $M^\vee$ by definition. 
    Finally, condition $b_i \in M \minus \sigma$ says $b_i \bullet a_j=\del(b_i)=1$ for all $i,j$, Thus, $\tau$ is $q$-simplex in $\D_{<\sigma}(M,\la,\del)$, as required.
\end{proof}

\begin{thm} \label{thm: connectivity D}
  Let $(M,\la,\del) \in \Fd$.  The complex $\D(M,\la,\del)$ is 
  $\frac{g_X(M,\la,\del)-2usr(R)-6}{3}$-connected. If $R$ is a PID, the connectivity bound can be improved to $\frac{g_X(M,\la,\del)-5}{3}$-connected. 
\end{thm}

The idea of the proof is to deduce the connectivity of the disordered arc complex $\D(M,\la,\del)$ from that of the simplicial complex $\B(M,\la,\del)$. When the characteristic of $R$ is not 2, we will actually work with the associated semi-simplicial set $\B(M,\la,\del)^{ord}$. We will use a ``bad simplex argument", appropriately adapted to order complexes of simplicial complexes in the characteristic not 2 case; see Remark~\ref{rem:bad} for general comments about the generalization from simplicial complexes to such semi-simplicial sets,  in the language of the general framework  given in \cite[Sec 2.1]{HatVogTheter}. 

\begin{proof}
Let $A(R)=5$ if $R$ is a PID and $A(R)=2usr(R)+6$ otherwise. So we want to show that $\D(M,\la,\del)$ is 
  $\frac{g_X(M,\la,\del)-A(R)}{3}$-connected. 
Note that $A(R) \ge sr(R)+3$,  since $sr(R)=2$ when $R$ is a PID, and  since $2usr(R)+6 >  sr(R)+3$ in the general case. 

\smallskip

When $g_X(M,\la,\del) \leq A(R)-4$ the result is vacuously true. 
When $g_X(M,\la,\del) \in \{A(R)-3,A(R)-2,A(R)-1\}$ we just need to show that $\D(M,\la,\del) \neq \emptyset$, which is the case since $A(R) \ge 5$ (in both cases) so we can pick a split $X^{\# 2}$ summand in $M$ by Lemma \ref{lem:splitH}(ii) and the generator of the first $X$--summand will be a non-separating arc in $M$, and hence also a vertex of $\D(M,\la,\del)$. 
 
Thus, we can assume without lost of generality that 
\begin{equation}\label{eqCR}
    g_X(M,\la,\del) \geq A(R)\geq sr(R)+3.
\end{equation}
We will prove the result by induction on $\rk(M)$, noting that $\rk(M) \geq g_X(M,\la,\del)$. 

\smallskip

Let $f:S^k \rightarrow \D(M,\la,\del)$, $k \leq \frac{g_X(M,\la,\del)-A(R)}{3}$.  We want to show that $f$ is null-homotopic. We will start with the case of $\operatorname{char}(R)\neq 2$. 
The inequalities \eqref{eqCR} give that 
$$k \leq \frac{g_X(M,\la,\del)-A(R)}{3}\leq g_X(M,\la,\del)-A(R)  \le g_X(M,\la,\del)-sr(R)-3.$$ 
Thus, $\B(M,\la,\del)^{ord}$ is $k$-connected by Corollary~\ref{cor: connectivity B}. 
Thus we have a commutative diagram
\begin{align}
\xymatrix{S^k \ar[r]^-{f}\ar[d]&\D(M,\la,\del) \ar[d]\\
  D^{k+1}\ar[r]^-{\hat{f}} \ar@{-->}[ur] & \B(M,\la,\del)^{ord}}
\end{align}
where the vertical maps are the obvious inclusions. 

By Proposition~\ref{prop:approx}, we can approximate the pair $(f,\hat f)$ by a simplicial map, with respect to some PL approximation of $D^{k+1}$.
%
%
We will inductively replace $\hat{f}$ until its image lies inside $\D(M,\la,\del)$, hence producing the required null-homotopy.

We say that a $p$-simplex $\sigma$ in $D^{k+1}$ is \textit{bad} if $\hat{f}(\sigma)=(a_0,\cdots,a_{p'})$ and there exists some $j>0$ with $a_0 \bullet a_j \neq 1$.

Observe that by definition vertices are never bad. In fact, any bad $p$-simplex satisfies $p\ge p' \geq 1$. 
Note also that the condition of being bad is ``dense'', in the sense that if $\hat{f}(\sigma)$ does not lie in $\D(M,\la,\del)$ then $\sigma$ contains a bad simplex as a face. 
Therefore it suffices to give a procedure to modify $\hat{f}$ so that it has fewer bad simplices, and then apply this procedure until there are no bad simplices left. 

Let $\sigma$ be bad of maximal dimension $p$. Consider $\link(\sigma) \subset D^{k+1}$. We claim that 
$$\hat{f}(\link(\sigma)) \subset \D_{<\tilde{f}(\s)}(M,\la,\del) \subset \B(M,\la,\del)^{ord}$$ 
  
Indeed, let $\tau \in \link(\sigma)$, and let $\hat{f}(\tau)=(a_0',\cdots,a_q') \in \B(M,\la,\del)^{ord}$. If $a'_j\bullet a_i\neq 1$ for some $0\le i\le p'$ and $0\le j\le q$, with $v$ a vertex of $\tau$ mapped to $a'_j$, then $v*\s$ is a bad simplex of strictly larger dimension than $\s$, contradicting the maximality assumption. 
Likewise, we must have $a_i'\bullet a_j'=1$ for all $0\le i<j\le q$ as otherwise, if $v,w$ are vertices of $\tau$ mapping to $a_i'$ and $a_j'$ respectively, then $v*w*\s$ would be a bad simplex of larger dimension than $\s$, contradicting again the maximality. 
Given that $\tau*\s$ maps to a simplex of $\B(M,\la,\del)^{ord}$, we must have that $(a'_0,\dots,a'_q)$ forms a simplex of  $\D_{<\tilde{f}(\s)}(M,\la,\del)$. 

Let $M'=M \minus \hat{f}(\sigma)\le M$, $\la'=\la|_{M'}$ and $\del'=\del|_{M'}$, then $\D_{<\tilde f(\s)}(M,\la,\del) \cong \D(M',\la',\del')$ in a canonical way as in Lemma~\ref{lem:Dconnect}.
To apply induction, we need to check that $\rk(M')<\rk(M)$. 
Using \eqref{eqCR}, we have 
$$k\le \frac{g_X(M,\la,\del)-A(R)}{3} \le g_X(M,\la,\del)-A(R)\le g_X(M,\la,\del)-sr(R)-2.$$
Since $\sigma$ is a simplex in $D^{k+1}$, it follows that 
$$p'+1\le p+1\le k+2\le g_X(M,\la,\del)-sr(R)\le rk(M)- sr(R).$$
Hence Lemma \ref{lem:cut form properties}(i) gives that $\rk(M')=\rk(M)-(p'+1)<\rk(M)$, and so we can apply induction and conclude that $\D(M',\la',\del')$ is $\frac{g_X(M',\la',\del')-A(R)}{3}$-connected. 

We want to use the same lemma to estimate the genus of $M'$. 
If $R$ is not a PID then $A(R)=2usr(R)+6$ so
\begin{align*}
2p'+2\ \le \ 2k+4\  &\le \ \frac{2}{3}\Big(g_X(M,\la,\del)-A(R)\Big)+4 \\
&\le \ g_X(M,\la,\del)-A(R)+4\ =\ g_X(M,\la,\del)-2usr(R)-2    
\end{align*}
which gives that $g_X(M,\la,\del)\ge 2usr(R)+2p'+4$.  

Thus Lemma \ref{lem:cut form properties}(ii) also applies giving  
$$g_X(M',\la',\del') \geq g_X(M,\la,\del)-(2p'+1) \geq g_X(M,\la,\del)-3p',$$ since $p' \geq 1$ because $\sigma$ is bad. 
When $R$ is a PID then Lemma \ref{lem:cut form properties}(ii) applies unconditionally without an apriori bound on the arc genus getting the same end result.

    Together, this gives that $\D(M',\la',\del')$ is $(k-p') \geq (k-p)$-connected. Now $\link(\sigma) \cong S^{k-p}$ since  $\sigma$ is a simplex in the interior of $D^{k+1}$ because it is bad.   Hence the map 
    $\hat{f}|_{\link(\sigma)}$ extends to a map $F: D^{k-p+1} \rightarrow \D(M',\la',\del')$, which by Lemma~\ref{lem:Dconnect} we can view as a map $$F: D^{k-p+1} \rightarrow \D_{<\hat{f(}\s)}(M,\la,\del).$$
By PL approximation (that can be directly applied since $\D$ is a simplicial complex), we can assume that $F$ is simplicial with respect to a triangulation of $D^{k-p+1}$ extending that of $\link(\s)$.    

Now, we can modify the PL structure of $D^{k+1}$ inside $\St(\s)=\link(\s)*\s\cong D^{k-p+1}*\del \s$, and 
 we can replace $\hat{f}|_{St(\sigma)}$ with the map 
$$F*\hat{f}: D^{k-p+1}*\del \sigma \cong \St(\sigma) \rightarrow \B(M,\la,\del)^{ord}$$
where we order the vertices in the target by putting those coming from $D^{k-p+1}$ before the ones from $\del \s$.
Now we claim that we have reduced the number of bad simplices of maximal dimension $p$. 
Indeed, if $\tau=\tau_0*\tau_1 \in D^{k-p+1}*\del \sigma$ has image $(a_0',\dots,a_q',a_{i_1},\dots,a_{i_r})$, then by construction $a_i' \in M'$ for all $i$ and $\del(a_i')=1$ so $a_i' \bullet a_j=1$ for all $i,j$. 
Also, $a_0'\bullet a'_i=1$ for all $i>0$, so the only way this can be bad is if $\tau_0=\emptyset$. 
Thus, $\tau \subset \del \sigma$ has dimension less than $p$, as required. This finishes the proof in the case of $\operatorname{char}(R)\neq 2$. 

\medskip

If $\operatorname{char}(R)=2$, we start likewise with a map $f:S^k\to \D(M,\la,\del)$, but we now extend it to a map $\hat f: D^{k+1}\to \B(M,\la,\del)$, instead of to $\B^{ord}(M,\la,\del)$. Just as above, we approximate $\hat f$ by a simplicial map, from a PL triangulation of the disc, and we want to modify $\hat f$ inductively so that it eventually has image in $\D(M,\la,\del)$. 

We say that a $p$-simplex $\sigma$ in $D^{k+1}$ is \textit{bad} in the $\operatorname{char}(R)=2$ case if $\hat{f}(\sigma)=\langle a_0,\cdots,a_{p'}\rangle$ and for every $i=0,\dots,p'$, there exists $j\neq i$ with $a_i \bullet a_j \neq 1$. Let $\sigma$ be a bad simplex of maximal dimension $p$.  We claim that we again have that 
$$\hat{f}(\link(\sigma)) \subset \D_{<\tilde{f}(\s)}(M,\la,\del) \subset \B(M,\la,\del).$$ Indeed, if $v$ is a vertex of the link, then we must have $v\bullet a_i=1$ for all $i=0,\dots,p'$, by maximality of $\s$, and likewise of $\langle v,w\rangle$ is a 1-simplex in the link, then we must have $v\bullet w=1$ by maximality of $\s$. The rest of the proof proceeds exactly as in the $\operatorname{char}(R)\neq 2$ case.  
\end{proof}

\begin{rem}[Bad simplex argument for semi-simplicial sets]\label{rem:bad}
We describe here how the proof to Theorem~\ref{thm: connectivity D}, in the characteristic not 2 case, can be interpreted as a quite direct adaptation to semi-simplicial sets of the bad simplex argument, as formalized in \cite[Sec 2.1]{HatVogTheter}. For brevity, we will write $\B(M)=\B(M,\la,\del)$ and $\D(M)=\D(M,\la,\del)$. 

The heart of the proof of the theorem is to modify a map $D^{k+1}\to \B(M)^{ord}$, from a PL triangulation of the disc to the ordered complex of $\B(M)$, so that its image in the end lands in the subcomplex $\D(M)$ of disordered arcs. To achieve this, we defined the notion of a ``bad" simplex: a simplex $(a_0,\dots,a_p)$ is bad if there is a $j>0$ such that $a_0\bullet a_j\neq 1$. 
Our definition of badness uses the ordering of the vertices, and hence does not make sense in $\B(M)$. This is the reason we work with $B(M)^{ord}$ rather than $\B(M)$, even though $\D(M)$ can also be considered as a subcomplex of the simplicial complex $\B(M)$. 

Our badness condition has the property that 
\begin{enumerate}
    \item Any simplex of $\B(M)^{ord}\setminus \D(M)$ has a bad face;
    \item If two faces of a simplex are bad, their join is also bad. 
\end{enumerate}
In \cite{HatVogTheter}, for a simplex $\s$ of the larger complex, $\B^{ord}(M)$ in our case, they call a simplex $\tau$ in its link {\em good for $\s$} if any face of $\tau*\s$ that is bad is a face of $\s$. In our case, we see that the simplices that are good for $\s$ will necessarily be part of the ``down-link" $\link_{<\s}$, the simplices $\tau$ in the link such that $\tau*\s$ is a simplex, with the vertices of $\tau$ placed before those of $\s$ in the ordering. In fact the subcomplex $G_\s$ of simplices that are good for $\s$ identifies with the complex $\D_{<\s}(M)$, that we showed in Lemma~\ref{lem:Dconnect} to be isomorphic to $\D(M\setminus\s)$. 

Following \cite[Prop 2.1]{HatVogTheter}, keeping in addition track of the ordering of the vertices, we can apply the bad simplex argument to show that $\D(M)$ is $n$-connected, deforming any disc of dimension $k+1\le n+1$, if we can show that $G_\s$ is $(n-p)$--connected for each bad $p$-simplex $\s$. Now this is done by induction, treating the first case by hand, and using our genus estimate for $M\setminus\s$. In our case, $n=\frac{g(M)-c}{3}$, where $c$ is a constant, and Lemma~\ref{lem:cut form properties} says that $g(M\setminus\s)\ge g(M)-(2p+1)$. The needed connectivity of $G_\s$ follows from the fact that $\frac{g(M)-(2p+1)-c}{3}\ge \frac{g(M)-3p-c}{3}=n-p$ for any $p\ge 1$, which is enough since there are no bad 0-simplices. 
\end{rem}

\section{Cancellation properties and destabilization complex}\label{sec:cancell}

In order to get the final homological stability results we need to compare the disordered arc complex to the canonically defined destabilization complex in \cite{RWW}. This comparison will use a cancellation property in the category $\Fd$, that we start by describing. 

\subsection{Cancellation}

In this section we prove a general cancellation result based on the dual unitary stable rank, with a proof using the connectivity of the disordered arc complex $\D$. In Appendix~\ref{appB}, we prove a stronger cancellation result in the case when the base ring $R$ is a PID, using instead a classification result for formed spaces with boundary that is only valid for PIDs, see Theorems~\ref{thm: cancellation} and~\ref{thm: classification}.

\begin{thm}[General cancellation] \label{thm: general cancellation}
    Let $(M_i,\la_i,\del_i) \in \Fd$ with $\del_i$ surjective 
    for $i=1,2$ and such that $$(M_1,\la_1,\del_1)\#X \cong (M_2,\la_2,\del_2)\# X.$$
    If $g_X(M_1,\la_1,\del_1) \ge 2usr(R)+5$ then $(M_1,\la_1,\del_1) \cong (M_2,\la_2,\del_2).$
\end{thm}

\begin{proof}
Let $(M,\la,\del)=(M_1,\la_1,\del_1)\#X$ and let $a_1, a_2$ denote the canonical generators of the $X$-summands under the identification 
$$(M,\la,\del)=(M_1,\la_1,\del_1)\#X\cong (M_2,\la_2,\del_2)\#X,$$
where we have fixed an isomorphism as in the statement. For $i=1,2$ we have that $(M_i,\la_i,\del_i)=(M \setminus a_i, \la|_{M \setminus a_i},\del|_{M \setminus a_i})$ by Lemma \ref{lem:splitH}(ii), and our assumption on $\del_i$ 
is equivalent to saying that the arcs $a_1,a_2$ are non-separating. 
The above cancellation property is thus equivalent to the fact that $\Aut_{\Fd}(M,\la,\del)$ acts transitively on the set of non-separating arcs in $(M,\la,\del)$.

We will now show that transitivity of the action follows from the connectivity of the complex $\D(M,\la,\del)$, which holds by Theorem~\ref{thm: connectivity D} since $g_X(M,\la,\del) \ge 2(usr(R)+1)+4$ by our assumption. 

Observe that the set of vertices of $\D(M,\la,\del)$ is precisely the set of non-separating arcs. Thus it suffices to show that if $a_1,a_2$ are connected by a sequences of edges in $\D(M,\la,\del)$, then there is an automorphism in $\Aut_{\Fd}(M,\la,\del)$ taking $a_1$ to $a_2$. 
Without loss of generality assume that $(a_1,a_2)$ is a $1$-simplex in $\D(M,\la,\del)$. Then it defines a morphism $(a_1,a_2): X^{\#2} \to (M,\la,\del)$, and by Lemma \ref{lem:splitH}(i)
 we have a splitting $(M,\la,\del)=(M',\la',\del')\# X^{\#2}$ where $M'=M \setminus (a_1,a_2)$ denotes the cut form. 
Now the braiding $\beta_{1,1}^{-1}$ on the last two coordinates defines such a morphism, finishing the proof. 
\end{proof}

In the language of \cite{RWW}, the above theorem says that the category $\Fd$ satisfies {\em local cancellation}, see Definition~1.9 in that paper. 

\begin{rem} \label{rem: csr}
One could define the  \textit{cancellation stable rank} $\Dr(A,R)$ of an $R$-module $A$ to be the smallest integer $k$ with the property that for any $N \in \Fd$ $$N\# X^{\# p+1}\cong A\# X^{\#n} \ \ \textrm{with}\ \ g_X(N)>0 \ \ \textrm{and}\ \ 0 \le p \le n-k \ \ \Longrightarrow \ \ N\cong A\# X^{\#n-p-1}.$$

In this language, Theorem \ref{thm: general cancellation} says that $\Dr(A,R) \le 2usr(R)+6$ for any $R$-module $A$. 
And Theorem \ref{thm: cancellation}, proved in the appendix, improves the above bound to $\Dr(A,R)=2$ when $R$ is a PID.

    Note that the connectivity of the destabilization complex is always related to a local cancellation property, see \cite[Lem 2.3]{RWW}. But here we do not know yet that the complex of disordered arc is (essentially) the destabilization complex in the sense of \cite{RWW}, and in fact we will use the cancellation result to show the equivalence between the two complexes. 
\end{rem}

\subsection{The destabilization complex}\label{sec:destabilization}

The monoidal category $\Fd$ of formed spaces with boundary introduced in Section~\ref{sec: formedd} has a subcategory $\FdX$, generated by the object $X=(D^2,0,\id)$, that is braided monoidal by Proposition~\ref{prop:beta}. The monoidal structure of the full category $\Fd$ then makes the classifying space $B\Fd$ into an  
 $E_1$-module over the $E_2$-algebra $B\FdX$ in the language of \cite{krannich}, see in particular Section 7 of that paper that gives this particular generalization of the set-up of \cite{RWW} needed here. 
 
 Given a choice of an object $A$ of $\Fd$, there is a {\em destabilization complex} $W_n(A,X)$ associated to stabilization by $X$, whose connectivity, under mild conditions, rules homological stability for adding copies of $X$. We describe the space (semi-simplicial set) $W_n(A,X)$ below, and relate it to the disordered arc complex $\D(A\# X^{\# n})$. Before doing so, we will modify the category $\Fd$ slightly, as in \cite[Rem 1.1]{RWW}, to force a stronger cancellation property than the one given by Theorems~\ref{thm: cancellation} and \ref{thm: general cancellation}, as the stronger cancellation property is needed to be able to deduce homological stability from the connectivity of $W_n(A,X)$.  This will have the effect of modifying $W_n(A,X)$ so that it is the correct complex of destabilization.

\smallskip

Define $\Fd^c(A,X)$ to be the category with one object $N_k:=A\# X^{\# k}$ for each $k\ge 0$,  and with the automorphisms groups $\Aut_{\Fd}(N_k)$ as only morphisms. Now $B\Fd^c(A,X)$ is again an $E_1$-module over the $E_2$--algebra $B\FdX$ via the sum $\#$, and it satisfies local cancellation by construction: For all $0\le p\le n-1$, 
$$N_k\# X^{\# p+1}\cong A\# X^{\#n} \Longrightarrow \ \ N_k\cong A\# X^{\#n-p-1}, \ \textrm{and in particular } k=n-p-1.$$

\medskip

 We are now ready to define the destabilization complex. 

\begin{Def}\cite[Def 7.5]{krannich}
    Let $A$ be a formed space with boundary, and let $n\ge 1$. The destabilization complex $W_n(A,X)$ is a semi-simplicial set  with set of $p$--simplices 
    $$W_n(A,X)_p=\{(N_{k},f) \ |\  N_k=A\# X^{\#k} \textrm{ and } f:N_{k}\# X^{\# p+1}\xrightarrow{\cong} A\# X^{\# n} \ \textrm{in}\ \Fd^c(A,X)\}/_\sim$$ 
    where $k=n-p-1$ is necessary for the isomorphism to exist as we just saw, and $(N_k,f)\sim (N_{k},f')$ if  $f'=f\circ (g\# \id)$ for some $g\in\Aut_{\Fd}(N_k)$. (The data of the {\em complement} $N_k$ is actually superfluous her, since $N_k=N_{n-p-1}$ is the only possibility.)  

The $i$th face map is obtained by adding a copy of $X$ to the complement $N_k$ and precomposing $f$ with an appropriate braiding to get a new pair $d_i(N_k,f)=(N_k\# X,d_if)=(N_{k+1},d_if)$ with 
$$d_if: N_k\# X\# X^{\# p}\xrightarrow{\id\# \beta_{X^{\# i},X}^{-1}\# \id} N_k\# X^{\# i}\# X \# X^{\# p-i}\xrightarrow{\ f\ } A\# X^{\# n}.$$
There is a simplicial action of $\Aut(A\# X^{\# n})$ on $W_n(A,X)$ by postcomposition. 
\end{Def}

Note that the action of $\Aut(A\# X^{\# n})$ on $W_n(A,X)$ is transitive on the set of $p$-simplices for all $p$ by construction. (This is in fact equivalent to the local cancellation property enforced on the category, see \cite[Thm 1.10(a)]{RWW}). Hence we have an identification 

$$W_n(A,X)_p\cong \Aut(A\# X^{\# n})/\Aut(A\# X^{\# n-p-1}).$$

\subsection{Disordered arcs as a destabilization complex}

Suppose $(M,\la,\del)\in \Fd$ has genus  $g_X(M,\la,
\del)=n$, so that we can write $M=A\# X^{\# n}$ for some $A\in \Fd$.  
We will now relate the disordered algebraic arc complex $\D(M,\la,\del)$ and the destabilization complex $W_n(A,X)$ just defined.

\medskip

Recall that the vertices of simplices of $\D(M,\la,\del)$ are canonically ordered when $\operatorname{char}(R)\neq 2$. Hence $\D(M,\la,\del)$  can be considered as a semi-simplicial set in that case, the $i$th face a simplex $\s=(a_0,\dots,a_p)$ being the simplex $d_i\s$ obtained from $\s$ by removing its $(i+1)$st vertex $a_i$. In the characteristic 2 case, $\D(M,\la,\del)$ cannot be identified with a semi-simplicial set, and we will instead compare $\D(M,\la,\del)$ to the simplicial complex $S_n(A,X)$ associated to $W_n(A,X)$. The complex $S_n(A,X)$ has the same vertices as $W_n(A,X)$, and vertices $v_0,\dots,v_p$ form a simplex in $S_n(A,X)$ if there exists a $p$-simplex of $W_n(A,X)$ having $v_0,\dots,v_p$ as its vertices. By \cite[Thm 2.10]{RWW}, under good conditions, the simplicial complex $S_n(A,X)$ is highly connected if and only if the semi-simplicial set $W_n(A,X)$ is highly connected, which implies that it is most often equivalent to work with the one or the other for the purpose of homological stability. (In this case, we will indeed be in the situation where $W_n(A,X)$ has $(p+1)!$ simplices for each $p$-simplex of $S_n(A,X)$, having one for each ordering of the vertices, and we could equivalently compare $\D^{ord}(M,\la,\del)$ with $W_n(A,X)$.) 

\begin{prop} \label{prop: D vs W}
Let $(M,\la,\del)\!\cong\!A\# X^{\# n}$ be a formed space with boundary with $g_X(M,\la,\del)\!\ge\!n$. 
\begin{enumerate}[(i)]
    \item If $\s \in \D(M,\la,\del)$ is a $p$-simplex then $M\cong M\minus \s \,\# \, X^{\# p+1}$. 
    \item $p$--simplices of  $\D(M,\la,\del)$ are in one-to-one correspondence with maps $f: X^{\# p+1}\to (M,\la,\del)$ such that $\del$ is still surjective on the cut form $M \minus \im(f)$.
    \item If $\operatorname{char}(R)\neq 2$, then 
    $$sk_{\leq n-2usr(R)-6} \D(M,\la,\del)\cong 
    sk_{\leq n-2usr(R)-6}(W_n(A,X)).$$.
    \item If $\operatorname{char}(R)=2$, then 
    $$sk_{\leq n-2usr(R)-6} \D(M,\la,\del)\cong 
    sk_{\leq n-2usr(R)-6}(S_n(A,X)).$$ 
\end{enumerate}
When $R$ is a PID, the isomorphisms in (iii) and (iv) hold for the $(n-2)$-skeleta. 
  \end{prop}

  We could have changed the definition of $\D$ to make the isomorphism hold for the full semi-simplicial set/simplicial complex instead of for the stated skeleta, but this would have no effect on our main results. 
  
\begin{proof}
Without loss of generality, we can assume that $M=A\# X^{\# n}$. 

\noindent
{\bf Part(i):} A $p$-simplex $\sigma=(a_0,\dots,a_p)$ defines a morphism $f_\s:X^{\#p+1} \rightarrow (M,\la,\del)$ in $\Fd$ sending the $i$-th standard basis element of the underlying module $R^{p+1}$ to $a_i$, since the conditions $a_i\bullet a_j=1$ for $i<j$ and $\del(a_i)=1$ for being a $p$-simplex in $\D(M,\la,\del)$ precisely say that $\la$ and $\del$ restrict to $\la_{X^{p+1}}$ and $\del_{X^{p+1}}$ on the submodule generated by these arcs.
The result now follows from Lemma \ref{lem:splitH}(ii) since $M \minus \s=M\minus \im(f_\s)$.

\smallskip

\noindent
{\bf Part(ii):} We have already seen in the proof of Part (i) how to associate to a simplex $\s$ a map $f_\s:X^{\# p+1}\to M$. Now the non-separating condition implies that there is an element $a$ of $M\minus \s$ such that $\del(a)=1$, giving the required surjectivity. Conversely, if $f$ is a map as in the statement, then the evaluation of $f$ at the standard generators of $X^{\# p+1}$ gives a collection of arcs $a_i=f(e_{i+1})$ that satisfy $a_i\bullet a_j=1$ for all $i<j$, because $f$ respects the form. 
Moreover, by the surjectivity assumption we can find $a \in M$ such that $\del a=1$ and $a \in M \setminus \im(f)$, so $\{a_0\bullet-,\dots,a_p\bullet -,\del\}$ is unimodular by evaluating this collection of maps on the elements $a_0,\dots,a_p,a$ themselves and observing the corresponding matrix has rank $p+2$. 
Thus, $\sigma=(a_0,\dots,a_p)$ defines the required simplex with $f_\s=f$. 

\smallskip

\noindent
{\bf Part(iii):} There is a forgetful map $sk_{\le n-2}W_n(A,X)\to \D(M,\la,\del)$ taking a $p$--simplex $(N,f)=(N_{n-p-1},f)$ of $W_n(A,X)$, with $f:N\# X^{\# p+1}\to M$,  to the collection of arcs $(f(e_0),\dots,f(e_{p}))$ for $e_0,\dots,e_p$ the standard basis vectors of the $X^{\# p+1}$--summand in the source.  Here the non-separating condition for the arcs is guaranteed by the fact that $(N,f)$ is not a maximal simplex, so that $N=A\# X^{\# n-p-1}$ has at least one $X$--summand, where we observe that any simplex is the face of an $n-1$-simplex by using the transitivity of the action of $\Aut(M,\la,\del)$ on the simplices. 
This map is injective Lemma \ref{lem: enhanced splitting}. 
Surjectivity of the map when restricting both complexes to their $(n-2usr(R)-6)$-skeleta (resp.~$(n-2)$-skeleta in the PID case) follows from  cancellation: 
Let $\sigma= (a_0,\dots,a_p)$ be a $p$-simplex in $\D(M,\la,\del)$ with $p \le n-2usr(R)-6$ (resp.~$p\le n-2$). By Part(i) we have $M=A \# X^{\#n}\cong M \setminus \sigma \# X^{\#p+1}$, so by the cancellation Theorem~\ref{thm: general cancellation} (resp.~Theorem~\ref{thm: cancellation}), we have that $M \setminus \sigma \cong N_{n-p-1}$. 
Thus, we get $f: N_{n-p-1} \# X^{\#p+1} \xrightarrow{\cong} M$ taking values $a_0,\dots,a_p$ in the standard basis of $X^{\#p+1}$.

We are left to check that the map is simplicial. Because both semi-simplicial sets admit a simplicial action by $\Aut(A\# X^{\# n})$ that it transitive on $p$-simplices for every $p$, and the map is equivariant, it is enough to check that the map respects the face maps of our favorite $p$--simplex. Consider the simplex $(A\# X^{\# n-p-1},\id)$. Its image in $\D(M,\la,\del)$ is the collection of arcs defined by the last $p+1$ standard generators inside the submodule $X^{\# n}$, and the face maps in $\D$ are the forgetful maps, with $d_i$ forgetting the $i+1$st arc in that collection. We need to check that this corresponds to the face map in the source, defined using the braided structure. This comes down to the same computation as in the geometric case, since the braid action is geometric: by Proposition~\ref{prop:Hdmonoidal}, $X^{\# n}$ identifies with $\Hd(D^{\hash n})$ and the braiding is defined as the image of the geometric braiding in $D^{\hash n}$. The last part of the proof of \cite[Prop 4.4]{HVW} computes the effect of the block braid $\beta^{-1}_{X^{\# i},X}$ on the standard generators $\rho_{n-p-1},\dots,\rho_n$ whose homology class are, by definition the standard generators $e_{n-p-1},\dots,e_n$ of $X^{\# n}$. Just as in the geometric case, the conclusion of the computation is that the face map $d_i$ in the destabilization complex has the effect of forgetting $e_{n-p-1+i}$. 

\smallskip

\noindent
{\bf Part (iv):} We now assume that $\operatorname{char}(R)=2$. 
The forgetful map $W_n(A,X)\to \D(M,\la,\del)$ is still well-defined an surjective, but it now factors through $S_n(A,X)$, since, if different orderings of the same vertices define a simplex in $W_n(A,X)$, they will have the same image in $\D(M,\la,\del)$. The map $S_n(A,X)\to \D(M,\la,\del)$ is still surjective, and now it is also injective since if $(N,f)$ and $(N',f')$ have the same image in $\D(M,\la,\del)$, that means that we have an equality of sets $\{f(e_0),\dots,f(e_p)\}= \{f'(e_0),\dots,f'(e_p)\}$, which in turn means that $(N,f)$ and $(N',f')$ represent the same simplex in $S_n(A,X)$ since vertices of $S_n(A,X)$ are determined by their value on the corresponding generator $e_i$ in $X^{\# p+1}$. Finally the map is an equivariant simplicial map for the same reasons as in the characteristic not 2 case.
\end{proof}

\begin{cor} \label{cor: connectivity W}
Let $A\in \Fd$. The destabilization complex $W_n(X,A)$ is $\frac{n-2usr(R)-7}{3}$-connected. If $R$ is a PID, $W_n(X,A)$ is $\frac{n-5}{3}$-connected. 
\end{cor}

\begin{proof}
Let $k=\lfloor\frac{n-2usr(R)-7}{3}\rfloor$ (resp.~$k=\lfloor\frac{n-5}{3}\rfloor$ if $R$ is a PID). Note first that the result holds trivially when $k\le -2$. For $k=-1$, we need to check that $W_n(A,X)$ is non-empty, which follows in both cases from the fact that $n\ge 2$ under this assumption. So we can assume $k\ge 0$. 

By Theorem \ref{thm: connectivity D} we know that $\D(A \# X^{\#n})$ is 
$k'$--connected for $k'=\lfloor\frac{n-2usr(R)-6}{3}\rfloor$ (or $k'=\lfloor\frac{n-5}{3}\rfloor$ if $R$ is a PID), and by Proposition~\ref{prop: D vs W}, the complexes $\D(A \# X^{\#n})$ and $W_n(A,X)$ (respectively $S_n(A,X)$ in characteristic 2) share the same $m$-skeleton for $m=n-2usr(R)-6$ (or $m=n-2$ if $R$ is a PID). To prove the result, it is enough to check that  both $k\le k'$, which is immediate, and $k\le m-1$, where in the characteristic 2 case, we also use \cite[Theorem 2.10]{RWW} to get the connectivity of $W_n(A,X)$ from that of $S_n(A,X)$.

In the non-PID case, $m=n-2usr(R)-6$ and $k=\lfloor\frac{m-1}{3}\rfloor$. So we need to check that $\lfloor\frac{m-1}{3}\rfloor\le m-1$, which holds as long as $m\ge 0$. This covers the case $\lfloor\frac{m-1}{3}\rfloor\ge 0$. 

If $R$ is a PID, we have $m=n-2$ and $k=\lfloor\frac{m-3}{3}\rfloor$, and the inequality $\lfloor\frac{m-3}{3}\rfloor\le m-1$ holds as long as $m\ge -1$, which covers the case $\lfloor\frac{m-3}{3}\rfloor\ge 0$ as needed. 
\end{proof}

\begin{rem}[The destabilization complex in terms of matrices]
When $A=0$, the underlying module of $A\# X^{\# n}$ is $R^n$, and hence the group $\Aut(X^{\# n})$ can be considered as a subgroup of $\GL_n(R)$. Given that the action of $\Aut(X^{\# n})$ on the set of $p$-simplices $W_n(0,X)_p$ is transitive for every $p$, it is possible to use matrices to describe the semi-simplicial set $W_n(0,X)$. We explain here that simplices in $W_n(0,X)$ correspond to collections of column-vectors in matrices in $\Aut(X^{\# n})$. This is the point of view taken in the article \cite{MPPRW}, see Lemma 3.14 in that paper. 

Indeed, as in Part (iii) in the proof of Proposition~\ref{prop: D vs W}, we can consider the forgetful map $W_n(0,X)_p\to (R^n)^{p+1}$, taking a $p$-simplex $(N,f)$, with $f:N\# X^{\# p+1}\to X^{\# n}$, to the collection $(f(e_{0}),\dots,f(e_{p}))$ of images of standard basis vectors in $X^{\# p+1}$. This map is injective by Lemma~\ref{lem: enhanced splitting} and, if we represent morphisms $f$ by matrices, has image the collections $(c_{n-p},\dots,c _n)$ of last $p+1$ column vectors of matrices in $\Aut(X^{\# n})$. Finally, the same braid group computation used in the proof of the proposition shows that the $i$th face map in $W_n(0,X)$ corresponds, under this map, to the map forgetting the $i$th vector in the collection. From this it follows that in fact any collection of $(p+1)$ column vectors in a  matrix representing an element of $\Aut(X^{\# n})$ defines a $p$-simplex of $W(0,X)$, and any simplex can be represented that way. 
\end{rem}

\section{The homological stability theorem}\label{sec:stability}

Our main stability theorem is now a direct application of the ``stability machine'' of \cite{RWW} and its generalization~\cite{krannich}. We explain here how this works. 

\smallskip

Recall from Section~\ref{sec:destabilization} the category $\Fd^c(A,X)$ with objects the formed spaces with boundary $N_k=A\# X^{\#k}$ for all $k\ge 0$ and their automorphisms as morphisms. This category is a module over the braided monoidal category $\FdX$, with the action induced from the monoidal structure of $\Fd$. This action in particular encodes the stabilization maps 
$$\s_n: \Aut_{\Fd}(A\# X^{\#n}) \xrightarrow{-\#\id_{X}} \Aut_{\Fd}(A\# X^{\# n+1})$$
that we will be studying. Note that these maps are injective by definition (for example by considering their matrix forms). 
We are most particularly interested in the case $A=0$ is the trivial module, since 
$$\Aut_{\Fd}(X^{\# 2n+1})\cong \Sp_{2n}(R)$$
is the symplectic group, with $\Aut_{\Fd}(X^{\# 2n})$ a group one could define as an ``odd symplectic group'' (see Proposition~\ref{prop: automorphisms X^n}). 

\medskip

 As we will verify now, the main result of \cite{RWW,krannich} directly applies to the above stabilization maps. It says that if the complex of destabilizations (of Section~\ref{sec:destabilization}) is highly connected, as we have just shown, then stability holds, also with certain types of twisted coefficients. We start by stating and proving the result with constant coefficients. 
 
\begin{thm}\label{thm:stabconst}
    Let $R$ be a ring with finite dual unitary stable rank.  
    The map 
    $$H_i\big(\Aut_{\Fd}(A\# X^{\#n});\Z\big) \xrightarrow{-\#\id_{X}} H_i\big(\Aut_{\Fd}(A\# X^{\# n+1}),\Z\big)$$
    is an epimorphism for $i\le \frac{n-2usr(R)-3}{3}$ and an isomorphism for $i\le \frac{n-2usr(R)-6}{3}$. When $R$ is a PID the map is an epimorphism for $i\le \frac{n-1}{3}$ and and isomorphism for $i\le \frac{n-4}{3}$. 

    Moreover, if $A=0$ and $n=2g+1$ is odd then the map 
    $$H_i\big(\Aut_{\Fd}( X^{\#2g+1});\Z\big) \xrightarrow{-\#\id_{X}} H_i\big(\Aut_{\Fd}( X^{\# 2g+2}),\Z\big)$$
    is a monomorphism for all $i$.
\end{thm}

\begin{proof}
Let us begin by proving the first part.
    By Proposition~\ref{prop:beta}, the category $\FdX$ is braided. As $\Fd^c(A,X)$ is a module over $\FdX$, \cite[Lem7.2]{krannich} shows that the classifying space $B\Fd^c(A,X)$ is an $E_1$-module over the $E_2$-algebra $B\FdX$. 

     $W_n(A,X)$  is the semi-simplicial set denoted $W^{\text{RW}}(A\# X^{\# n})_\bullet$ in \cite[Def 7.5]{krannich}. By Section 7.3 in that paper,  it is homotopy equivalent to the semi-simplicial space  $W(A\# X^{\# n})_\bullet$ of \cite{krannich} since $\Fd^c(X,A)$ is a groupoid satisfying cancellation, and the stabilization maps $\s_n$ are injective. By Remark~2.7 of that paper, this determines the connectivity assumption of Theorem A in that paper: the canonical resolution of the assumption of the theorem is $m$-connected, if and only if the space $W(A\# X^{\#n}_\bullet)$ is $(m-1)$-connected.

Define the grading $g_\F:\Fd^c(A,X)\to \mathbb{N}$ by 
$$g_\F(A\# X^{\# n})=\left\{\begin{array}{ll} n-3& R \ \textrm{is a PID}\\
 n-2usr(R)-5&\textrm{otherwise.}\end{array}\right.$$
Given that  $W(A,X)$ is $\left(\frac{g_\F-2}3\right)$--connected by Corollary \ref{cor: connectivity W}, we have that the canonical resolution
 is $\left(\frac{g_\F-2+3} 3\right)$--connected. 
Hence we can apply
 \cite[Thm A and Rem 2.24(i)]{krannich}, which give that the stabilization map $\s_n$ induces an 
         isomorphism in homology in degrees $i\leq\frac{g_\F-1} 3$ and
        an epimorphism for $i\leq \frac{g_\F+2}3$, which gives the result.
        
To prove the last part of the statement, 
suppose $(M,\la,\del)=(N,\la_N,\del_N)\# X \in \Fd$. Then there is a diagram
    \begin{align*}
\xymatrix{\Aut_{\Fd}(N,\la_N,\del_N) \ar[rr]^-{c} \ar[d]^-{s} & & \Aut_{\F}(N,\la_N) \ar[r]^-{d} & \Aut_F(\frac{(N,\la_N)}{Rad(N,\la_N)}) \\ 
\Aut_{\Fd}(M,\la,\del) 
\ar[rr]^-e & &\Aut_{\F}(\ker \del, \la|_{\ker \del}) \ar[u]^\cong_{c_f}  &}
\end{align*}
where $c,d,e$ are the canonical maps, $Rad(M,\la)=\{m \in M: \la(m,-)=0\}$ denotes the radical, and $s$ is the stabilization map, and where $c_f$ is conjugation by the isomorphism $f$ of \eqref{eq:kerdel} (with $N=\tilde{M}$, see also Remark \ref{rem:HorX}). 
We claim that the diagram commutes.

Indeed, the map $f$ is given explicitly by $f(u)=u-\la(a,u) a$ with inverse $f^{-1}(v)=v-\del(v) a$, where $a$ generates the right $X$-summand. 
Let $\alpha \in \Aut_{\Fd}(N,\la,\del)$ and $n \in N$.  Then $c(\alpha)(n)=\alpha(n)$. Going the other way around the square, we have 
\begin{align*}
f\big((e \circ s)(\alpha)(f^{-1}(n))\big)&= f((\al \# \id_X)(n-\del(n)a)) \\
& = f(\alpha(n)-\del(n) a)=\alpha(n)-\del(n)a-\la(a,\alpha(n))a+\del(n)\la(a,a)a  \\
& = \alpha(n)-\del(n)a+\del(\alpha(n))a+0=\al(n)
\end{align*}
 where the last equality holds since $\del(\alpha(m))=\del(m)$. Thus, commutativity follows.

Finally, when $(N,\la_N,\del_N)=X^{\#2g+1}$, the top composition $d\circ c$ identifies with the isomorphism 
$\Aut_{\Fd}(X^{\#2g+1}) \cong \Aut_{\F}(\HH^g)=Sp_{2g}(R)$ of Proposition \ref{prop: automorphisms X^n}, and the commutativity of the diagram implies that $s$ is injective, also in homology, as required.        
\end{proof}

\begin{proof}[Proof of Theorem~\ref{thm:A}]
   The statement of Theorem~\ref{thm:A} follows  by taking $A=0$ in the previous result, applying Proposition~\ref{prop: automorphisms X^n} to identify $\Aut(X^{\#n+1})=Sp_n(R)$.  
    The fact that the map $H_i(Sp_n(R),\Z) \to H_i(Sp_{n+1}(R),\Z)$ is always injective for $n$ even then follows from the last part of the previous result. 
    Finally, Lemma \ref{lem: stabilisations agree} gives that the stability map under consideration agrees with the classical one on (even) symplectic groups.
\end{proof}

\subsection{Non-trivial coefficients}\label{sec:twisted}

Under the same assumptions as for constant coefficients, homological stability also automatically holds for certain systems of twisted coefficients. These come in two flavours: {\em abelian coefficients} and {\em finite degree coefficient systems}. We describe them now, and state the corresponding stability results. 

\medskip

Consider the abelianization $\Aut(A\# X^{\# n})\rar H_1(\Aut(A\# X^{\# n}),\Z)$. Theorem~\ref{thm:stabconst} gives that the maps 
$$H_1(\Aut(A\# X^{\# n}),\Z) \rar H_1(\Aut(A\# X^{\# n+1}),\Z) \rar \dots$$
are eventually isomophisms. This way we can consider any module $M$ over the stable group $H_1(\Aut(A\# X^{\# \infty}),\Z)$ as a $\Aut(A\# X^{\# n})$--module by restriction. Such a coefficient system is called {\em abelian}.

By \cite[Lemma A1.(i)]{Krannich2020} it is known that $H_1(Sp_{2g}(\Z),\Z)=0$ for $g \ge 3$, hence abelian coefficient systems are not interesting in the case $R=\Z$, which is the main motivation of this paper. However other other rings $R$ these abelianizations do not need to vanish. 
For such rings, Theorem A of \cite{krannich} applies just as for constant coefficients, and give the following result:

\begin{thm}\label{thm:stabab}
    Let $R$ be a ring with finite dual unitary stable rank, and $M$ an $H_1(\Aut(A\# X^{\# \infty}),\Z)$--module.  
    The map 
    $$H_i\big(\Aut_{\Fd}(A\# X^{\#n});\Z\big) \xrightarrow{-\#\id_{X}} H_i\big(\Aut_{\Fd}(A\# X^{\# n+1}),\Z\big)$$
    is an epimorphism for $i\le \frac{n-2usr(R)-4}{3}$ and an isomorphism for $i\le \frac{n-2usr(R)-7}{3}$. When $R$ is a PID the map is an epimorphism for $i\le \frac{n-2}{3}$ and and isomorphism for $i\le \frac{n-5}{3}$. 
\end{thm}

The other type of non-trivial coefficients we will consider are those of finite degrees, that we will define now. 

\begin{Def}\label{def:coeff}\cite[Def 4.1]{krannich}
    A {\em coefficient system} for the groups $G_n=\Aut_{\Fd}(A\# X^{\# n})$ with respect to stabilization by $X$  is a $G_n$-module $M_n$ for each $n\ge 1$ with $G_n$-equivariant maps $s_n:M_n\to M_{n+1}$ satifying that the braiding $\id_{A\# X^{\# n}}\#\beta_{1,1}$, switching the last two $X$-summands in $A\# X^{\# n+2}$, acts trivially on the image of $M_n$ in $M_{n+2}$ by the double stabilization $s_{n+1}\circ s_n$.  
\end{Def}

When $A=0$, a coefficient system is thus a sequence of compatible symplectic group representations for all $n$'s, odd and even, satisfying the braid condition just described. 
Given a sequence of classical symplectic representations, one can attempt to produce a coefficient system in the above sense by setting $M_{2n}=M_{2n+1}$, since $G_{2n}=\Sp_{2n-1}(R)$ is a subgroup of $G_{2n+1}=\Sp_{2n}(R)$ (by Proposition~\ref{prop: automorphisms X^n}), and check whether the braid condition is satisfied. Such a construction will however not be relevant for us here, since such coefficient systems are never of finite degree (in the sense of Definition~\ref{def:findeg} below) unless they are constant.

 \begin{ex}\label{ex:coeff1} 
    \begin{enumerate}[(i)]
        \item 
Suppose $A=(M,\la,\del)$. The simplest example of a coefficient system is to set $M_n=M\oplus R^n$,  the underlying module of $A\# X^{\# n}$. Then $\id \x \Aut(X^{\# 2})$ acts trivially on the image of $M\oplus R^n$ in $M\oplus R^{n+2}=M\oplus R^n\oplus R^2$, and so the braiding in particular acts trivially. 

\item The previous coefficient system is not irreducible, and admits the following sub-coefficient system: let $M_n=\ker(\del: M\oplus R^n\to R)$. When $A=0$, we have that $M_{2n+1}\cong R^{2n}$, and one checks that on odd entries this coefficient system identifies with the defining representation of the symplectic groups, see \cite[Rem 3.7]{MPPRW}.

\item Additional examples over the rationals are discussed in \cite[sec 4.2]{MPPRW}.
\end{enumerate}
\end{ex}

\medskip

Given a coefficient system $\M=\{(M_n)_{n\ge 1},(s_n)_{n\ge 1}\}$, one can define its suspension $\Sigma \M$ by setting $\Sigma M_n=M_{n+1}$ and 
$$\Sigma s_n: \Sigma M_n=M_{n+1} \xrightarrow{s_{n+1}} M_{n+2}\xrightarrow{\id \# \beta_{1,1}} M_{n+2} = \Sigma M_{n+1}. $$
One checks that this is again a coefficient system, and that the maps $s_n$ assemble to define a morphism of coefficient systems $\M\to \Sigma \M$, see \cite[Def 4.4]{krannich}. Denoting $\ker \M$ and $\operatorname{coker}\M$ the coefficient system obtained by taking the cokernel of this map, we can now define the degree of a coefficient system: 

 \begin{Def}\cite[Def 4.10]{RWW}\label{def:findeg} A coefficient system $\M=\{(M_n)_{n\ge 1},(s_n)_{n\ge 1}\}$ is of 
 \begin{enumerate}[(i)]
     \item (split) degree $-1$ at $n_0$ is $M_n=0$ for all $n\ge n_0$. 
     \item {\em degree $r$ at $n_0$} if $\ker \M$ is of degree -1 at $n_0$ and $\operatorname{coker}\M$ is of degree $r-1$ at $n_0-1$. 
     \item {\em split of degree $r$ at $n_0$} if $\ker \M$ is of degree -1 at $n_0$ and $\operatorname{coker}\M$ is split of degree $r-1$ at $n_0-1$, where a coefficient system is split if all its structure maps $s_n$ are split injective in the category of coefficient systems. 
 \end{enumerate}
  \end{Def}

    \begin{ex} 
    \begin{enumerate}[(i)]
    \item Constant coefficient systems are of degree 0. 
        \item 
  The coefficient system $M_n=M\oplus R^n$, with $s_n:R^n\to R^{n+1}$ the standard inclusion, as in Example~\ref{ex:coeff1}(i), is a split coefficient system of degree 1, with cokernel the constant (and hence degree 0) coefficient system $R$. 
  \item Taking $A=0$ as in Example~\ref{ex:coeff1}(ii), the coefficient system $M_n=ker(\del: R^n\to R)$ is a coefficient system of degree 1, with cokernel the constant (and hence degree 0) coefficient system $R$. 
         \end{enumerate}
    \end{ex}

\begin{rem}\label{rem:twisted}
While standard examples of coefficient systems will tend to fit in our set-up, stabilizing with $X$ in the category $\Fd$ for formed spaces with boundary, just as well as in the more classical stabilization with $\HH$ in the category $\F$ of formed spaces, there is no direct translation between finite degree coefficient systems in the sense of the above definitions and finite degree coefficient systems coming from the same definitions but in the context of the category $\F$ instead, as in \cite[Sec 5.4]{RWW} or in \cite{Fri17}. 
\end{rem}

\begin{thm}\label{thm:stabtwist}
    Let $R$ be a ring with finite dual unitary stable rank, and let $M=\{M_n,s_n\}_{n\ge 1}$ be a coefficient system of degree $r$ for the groups $\{\Aut_{\Fd}(A\# X^{\#n})\}_{n\ge 1}$. The map 
    $$H_i\big(\Aut_{\Fd}(A\# X^{\#n});M_n\big) \xrightarrow{-\#\id_{X}} H_i\big(\Aut_{\Fd}(A\# X^{\# n+1}),M_{n+1}\big)$$
      is an epimorphism for $i\le \frac{n-2usr(R)-3r-2}{3}$ and an isomorphism for $i\le \frac{n-2usr(R)-3r-5}{3}$. When $R$ is a PID the map is an epimorphism for $i\le \frac{n-3r}{3}$ and and isomorphism for $i\le \frac{n-3r-3}{3}$. 
\end{thm}

\begin{proof}
This follows just like from Theorem~\ref{thm:stabconst}, from applying 
        \cite[Thm C]{krannich}. 
\end{proof}

\begin{proof}[Proof of Theorem~\ref{thm:B}]
   Theorem~\ref{thm:B} follows from the above result by setting $A=0$, applying Proposition~\ref{prop: automorphisms X^n} to identify the automorphism groups with symplectic groups.  
\end{proof}

%% file: unimodular.tex
In this section we will prove Theorem \ref{connectivity of U}, which gives a bound on the connectivity of the poset of unimodular sequences constrained to an affine subspace. 
The main ingredient of the proof is a generalisation of a result of van der Kallen in \cite{vdKL}, combined with further generalisations by Friedrich in \cite{Fri17} and the first author in \cite{S22a}.

\medskip

As in \cite{vdKL,Fri17},  
for a set $X$ we let $\mathcal{O}(X)$ denote the poset of finite non-empty sequences of elements of $X$, ordered by refinement. 

We say that $F \subset \mathcal{O}(X)$ satisfies the \textit{chain condition} if it is closed under taking subsequences. 
For $(v_1,\cdots,v_n) \in F$ we let $F_{(v_1,\dots,v_n)}$ denote the poset of sequences $(w_1,\cdots,w_m) \in \mathcal{O}(X)$ such that $(w_1,\cdots,w_m,v_1,\cdots,v_n) \in F$.

If $M$ is an $R$-module, the poset $\Uu(M)$ identifies with the subposet of sequences in $\mathcal{O}(M)$ which are unimodular.
And for $N\le M$ a submodule and $l\in M$ an element,  we have  
$$\Uu(M,l+N)=\mathcal{O}(l+N)\cap \Uu(M)$$ is the subposet of $\Uu(M)$ of unimodular sequences in $M$ of elements of the form $l+n$ with $n\in N$.

The proof of Theorem~\ref{connectivity of U} uses the following lemma from \cite{vdKL}, which we state here for convenience. 
\begin{lemma}\cite[Lem 2.13]{vdKL} \label{lem van der kallen}
Suppose $F \subset \Uu(M)$ satisfies the chain condition. 
Let $X \subset M$ be a subset. 
\begin{enumerate}[(1)]
    \item If $O(X)\cap F$ is $d$-connected, and $\mathcal{O}(X) \cap F_{(v_1,\cdots,v_m)}$ is $(d-m)$-connected for all sequences $(v_1,\cdots,v_m) \in F \setminus \mathcal{O}(X)$,  
    then $F$ is $d$-connected. 
    \item If $\mathcal{O}(X) \cap F_{(v_1,\cdots,v_m)}$ is $(d+1-m)$-connected for all sequences $(v_1,\cdots,v_m) \in F \setminus \mathcal{O}(X)$, and there is a length 1 sequence $(y_0)\in F$ with $\mathcal{O}(X) \cap F \subset F_{(y_0)}$, 
    then $F$ is $(d+1)$-connected. 
\end{enumerate}    
\end{lemma}

It will be convenient for the sake of the proof of the theorem to work in  $M^\infty:=M\oplus R^\infty$, so we always have additional unimodular vectors at hand, and we will consider posets of the form $\mathcal O(X)\cap \mathcal U(M^\infty)$, with $X$ an affine subspace of $M$ or the union of two such. Note that for $N\le M$, the posets $$\Uu(M,l+N)\ \cong\  \Uu(M^\infty,l+N)=\mathcal{O}(l+N)\cap \Uu(M^\infty)$$ are isomorphic. 
The proof follows closely \cite{Fri17}, and consists in checking that the argument there (and heavily inspired by \cite{vdKL}) carries over. 

\smallskip

Recall that  $$\tM(M,N):= \max\{\rk(L): L \subset N \text{ is a direct summand of } M\}.$$ 

\begin{proof}[Proof of Theorem~\ref{connectivity of U}]
Let $e_1$ denote the first standard generator in $R^\infty\le M^\infty=M\oplus R^\infty$. We will prove the result by induction on $r=\tM(M,N)=\tM(N,M^\infty)$ using the following two additional statements: 
\begin{enumerate}[(a)]
     \item $\mathcal{O}(l+N \cup l+ N+e_1) \cap \Uu(M^{\infty})$ is $(\tM(M,N)-sr(R))$-connected. 
    \item $\mathcal{O}(l+N \cup l + N+e_1) \cap \Uu(M^{\infty})_{(v_1,\cdots,v_k)}$ is $(\tM(M,N)-sr(R)-k)$-connected for any $k \geq 1$ and $(v_1,\cdots,v_k) \in \Uu(M)$.
\end{enumerate} 
We start by noting that (1),(2),(a) and (b) hold for $r \leq sr(R)-1$. Indeed,  parts (1),(2) and (b) are vacuously true in that case since any poset is $(-2)$-connected, and 
statement (a) is also true since $\mathcal{O}(l+ N \cup l+ N+e_1) \cap \Uu(M^{\infty})$ contains the one element sequence $(l+e_1)$ and hence is non-empty, so $(-1)$-connected.

Thus we can assume that $r=\tM(M,N) \geq sr(R)$. We fix $R^r\cong N_0\le N$ a  direct summand exhibiting that $\tM(M,N)=r$. In particular, $M=N_0\oplus M'$. We fix also a basis $x_1,\dots,x_r$ for $N_0$. 
Without loss of generality we can assume that $l\in M'$, as setting its $N_0$ component to 0 does not affect the subspace $l+N$.

\medskip

\noindent \textit{Proof that $(b)_{r-1}\Rightarrow (b)_r$.}

Let $Y=l+N \cup l+N+e_1$ and $(v_1,\cdots,v_k)\in Y$ a unimodular sequence. We need to show that  $\mathcal{O}(Y) \cap \Uu(L^{\infty})_{(v_1,\cdots,v_k)}$ is $d$-connected for $d=r-sr(R)-k$.

We start by finding $f \in \GL(M^{\infty})$ such that $f(Y)=Y$ and the projection $f(v_1)|_{N_0}$ is unimodular.  
Indeed, since $v_1$ is unimodular in $M^\infty$, there is $\phi: M^\infty \rightarrow R$ linear such that $\phi(v_1)=1$. 
Write $v_1=\sum_{i=1}^{r}{\lambda_i x_{i}}+a+b\in N_0\oplus M'\oplus R^\infty$.  Then $\phi(v_1)=\sum_{i=1}^{r}{\lambda_i \phi(x_{i})}+\phi(a+b)=1$ so $(\la_1,\dots,\la_t,\phi(a+b))$ is a unimodular vector in $R^{r+1}$. 
 
Since $r \geq sr(R)$, 
there are $\mu_1,\cdots,\mu_r \in R$ such that 
$(\lambda_1+\mu_1\phi(a+b),\cdots,\lambda_r+\mu_r \phi(a+b))$ is unimodular in $R^r$. 
Now define $f: M^\infty=N_0\oplus M'\oplus R^\infty \rightarrow N_0\oplus M'\oplus R^\infty$ via 
$$f(\nu_1,\dots,\nu_r,y,z)=(\nu_1+\mu_1 \phi(y+z),\dots,\nu_r+\mu_r\phi(y+z),y,z)$$
where we have identified $N_0$ with $R^r$ using the chosen basis. 
The map $f$ is invertible, with inverse of the same form. 
We have that both $f, f^{-1}$ satisfy that $f(N) \subset N$, $f^{-1}(N) \subset N$, so $f(N)=N$ as a set. 
Also, for any $x \in M^\infty$, $f(x) \in x+N$ and $f^{-1}(x) \in x+N$, hence $f(x+N)=x+N$ as sets. 
Thus, in particular $f(Y)=Y$ as sets. 
By construction, $f(v_1)|_{N_0}$ is unimodular.

Now we need to treat the case $r=\tM(M,N)= sr(R)$ separately: here we only need to consider the case $k=1$ and show that the poset is non-empty.
Observe that by definition $(f(v_1),l+e_1)$ is a unimodular sequence in $M^\infty$, and hence so is $(f^{-1}(l+e_1),v_1)$. 
Since $l+e_1 \in Y$ then $f^{-1}(l+e_1) \in Y$ too, so we are done. 

Now let us treat the general case $r>sr(R)$:
In that case $GL_r(R)$ acts transitively on the set of unimodular vectors of $N_0 \cong R^r$, so we can now modify $f$ so that $f(v_1)|_{N_0}=x_1$.  
Thus without loss of generality we can assume that we are in the situation that the $x_{1}$-coordinate of $v_1$ is $1$. 

Now set 
$$X:= l+N' \cup l+N'+e_1 \subset Y,$$
where $N'\le N$ is the complement of the submodule generated by $x_1$ in $N$.  Then
$$\mathcal{O}(X)\cap F=\mathcal{O}(X)\cap \Uu(M^\infty)_{(v_1,\dots,v_k)}=\mathcal{O}(X)\cap \Uu(M^\infty)_{(v_2,\dots,v_k)}$$
 and $(b)_{r-1}$ implies that $\mathcal{O}(X)\cap F$ is $(r-1-sr(R)-(k-1))$--connected, i.e.~$d$--connected. 
Likewise, given a sequence $(u_1,\cdots,u_m) \in F \setminus \mathcal{O}(X)$, the poset $F \cap \mathcal{O}(X)_{(u_1,\cdots,u_m)}=\mathcal{O}(X)\cap \Uu(M^\infty)_{(v_2,\dots,v_k,u_1,\dots,u_m)}$ is $(d-m)$-connected by $(b)_{r-1}$, which implies that $F$ itself is $d$--connected, as needed,  by Lemma \ref{lem van der kallen}(1).

\medskip

\noindent \textit{Proof that $(b)_r\Rightarrow (2)_r$.}
We now set  
$$X=  l+N' \cup l+N'+x_{1} \ \subset\ l+N, $$
for $N'$ as above, and let
$$F= \Uu(M,l+N)_{(v_1,\cdots,v_k)} =\mathcal{O}(l+N) \cap \Uu(M^{\infty})_{(v_1,\cdots,v_k)},$$
where $(v_1,\dots,v_k)\in \Uu(M)$. We need to show that $F$ is $d$--connected for $d=r-sr(R)-1-k$. 
For $m \geq 0$ and a (possibly empty) sequence $(u_1,\cdots,u_m) \in F \setminus \mathcal{O}(X)$ we have 
$$F \cap \mathcal{O}(X)_{(u_1,\cdots,u_m)}= \mathcal{O}(l+N' \cup l+N'+x_{1}) \cap \Uu(M^{\infty})_{(u_1,\cdots,u_m,v_1,\cdots,v_k)}$$
which is $(r-1-sr(R)-(m+k))$-connected by part $(b)_r$ (which applies as $k \geq 1$). 
Hence Lemma \ref{lem van der kallen}(1) implies that $F$ is $d$--connected for $d=r-1-sr(R)-k$, as needed. 

\medskip

\noindent \textit{Proof that $(a)_{r-1},(2)_{r-1}\Rightarrow (1)_r$} 
Let $X$ and $F$ be as above, just taking now $k=0$, i.e.~replacing $(v_1,\dots,v_k)$ with the empty sequence. Then $F\cap \mathcal{O}(X)$ is $(r-1-sr(R))$--connected by $(a)_{r-1}$ and $F \cap \mathcal{O}(X)_{(u_1,\cdots,u_m)}$ is $(r-1-sr(R)-k)$--connected by $(2)_{r-1}$. 
Lemma \ref{lem van der kallen}(1) then implies that $F$ is $(r-1-sr(R))$--connected, as needed.

\medskip

\noindent \textit{Proof that $(1)_r,(2)_r\Rightarrow (a)_r$.} 
Finally for (a) we will use Lemma \ref{lem van der kallen}(2). We take $$X=l+N\ \ \textrm{and}\ \ F=\mathcal{O}(l+N \cup l+N+e_1) \cap \Uu(M^{\infty}),$$ 
with
$$y_0=l+e_1.$$
Let $(u_1,\cdots,u_m) \in F\setminus \mathcal{O}(X)$. Without loss of generality we can assume that $u_1 \in l+N+e_1$, and we let $u_i'=u_i-s_i u_1$ for $i \geq 2$, where the $s_i \in R$ are chosen so that the $e_1$-component of $u_i'$ vanishes for $2 \leq i \leq m$. 
Then, 
$$\mathcal{O}(X) \cap F_{(u_1,\cdots,u_m)}= \mathcal{O}(l+N) \cap \Uu(M^\infty)_{(u_2',\cdots,u_m')}$$
which is $(r-sr(R)-m)$-connected by $(1)_{r}$ if $m=1$ and likewise $(r-sr(R)-1-(m-1))$-connected by $(2)_{r}$ if $m \geq 2$. 
Note also that $\mathcal{O}(X) \cap F \subset F_{(y_0)}$ is. Hence we can apply Lemma \ref{lem van der kallen}(2) with $d=r-sr(R)-1$, which gives the desired connectivity of $F$.  
\end{proof}

%% file: cancellation.tex
The goal of this appendix is to improve the bounds of the cancellation Theorem \ref{thm: general cancellation} when the ring $R$ is a PID.
The proof uses a classification of the objects of $\Fd$ that holds when $R$ is a PID. We start by stating the cancellation theorem, and then give the classification result. 

\begin{thm}[Cancellation for PIDs]\label{thm: cancellation}
    Let $R$ be a PID and $(M_i,\la_i,\del_i) \in \Fd$ with 
    $\del_i$ surjective for $i=1,2$ and such that  
 $$(M_1,\la_1,\del_1)\# X \cong (M_2,\la_2,\del_2) \# X.$$
 Then $(M_1,\la_1,\del_1) \cong (M_2,\la_2,\del_2)$. 
\end{thm}

Note that the assumption that both $\del_i$ are surjective, and hence both $(M_i,\la_i,\del_i)$ have positive arc genus, is needed. Indeed, 
Proposition \ref{prop:X^n} gives a decomposition $X^{\# 2g+1} \cong (\HH^g,0) \# X$ but $X^{\#2g} \ncong (\HH^g,0)$. 

\medskip
 
The classification result used in the proof of the above theorem will be given in terms of the following invariant:   

\begin{Def}
    Let $R$ be a PID and let $(M,\la,\del)$ be a formed space with boundary with $M\neq 0$. Its {\em form data} is the tuple 
    $$\DD(M,\la,\del)=(n,l,d_1,\dots,d_k,\delta_1,\dots,\delta_{k+1})$$ 
    where
    \begin{enumerate}[(i)]
        \item $n=\rk(M)$ is the rank of $M$; 
        \item $(l,d_1,\dots,d_k)$ is the Smith normal form of $\la$, representing the form $$diag(d_1,d_1,\dots,d_k,d_k,0,\dots,0)$$ with $l=n-2k \geq 0$ zeroes; 
        \item for $1\le i\le k$, 
        $(\delta_i)=\del(M_i)$, where $M_i=\{m \in M: \forall m' \in M \; d_{i} | \la(m,m')\}$.
        \item 
       $(\delta_{k+1})=\del(\operatorname{Rad}(M,\la))$, where $\operatorname{Rad}(M,\la)=\{m \in M: \forall m' \in M \; \la(m,m')=0\}$; 
    \end{enumerate}
\end{Def}

As already used in Lemma~\ref{lem: pid}, the fact that the Smith normal form of $\la$ has the above form follows from \cite[Theorem IV.I]{Newman}, where $d_1|d_2|\dots|d_k$ are uniquely determined up to scaling them by units in $R$.  
The above form data adds to the Smith normal form the data of the ideals of $R$ given by applying $\del$ to the subspaces $M_i$ determined by $\la$.





The $\delta_i$'s are likewise well-defined up to scaling by a unit, and 
we have $\delta_{i-1}|\delta_i| \frac{d_{i}}{d_{{i-1}}} \delta_{i-1}$ for $i \geq 2$ because $\frac{d_{i}}{d_{{i-1}}} M_{i-1} \subset M_i \subset M_{i-1}$. Setting $M_{k+1}=\Rad(M,\la)$, we see that the relation still hold for $i=k+1$.  


We say that two tuples $(n,l,d_1,\dots,d_k,\delta_1,\dots,\delta_{k+1})$ and $(n',l',d_1',\dots,d_k',\delta_1',\dots,\delta_{k+1}')$ are {\em equivalent} if $n=n'$, $l=l'$ and there are units $u_i,v_i$ such that $d_i=u_id_i'$ and $\delta_i=v_i\delta_i'$. 
The following result gives a classification of all the objects in $\Fd$: 

\begin{thm}[Classification of formed spaces with boundary over a PID] \label{thm: classification}
If $R$ is a PID then the isomorphism class of a formed space $(M,\la,\del) \in \Fd$ with $M\neq 0$ is determined by its equivalence class of form data $\DD(M,\la,\del)$. 

Moreover, any form data $(n,l,d_1,\dots,d_k,\delta_1,\dots,\delta_{k+1})$ satisfying the relations 
$d_1|d_2|\dots|d_k$ and $\delta_{i-1}|\delta_i| \frac{d_{i}}{d_{{i-1}}} \delta_{i-1}$ for $2\leq i \leq k+1$
is realized. 
\end{thm}

\begin{proof}
Suppose that $(n,l,d_1,\dots,d_k,\delta_1,\dots,\delta_{k+1})$ is a tuple satisfying the above relations. We start by constructing a \textit{standard formed space with boundary} realizing it: 
Take $M=R^n$. Writing the standard basis of $R^n$ as $e_1,f_1,\dots,e_k,f_k,g_1,\dots,g_l$,  
define $\la$ via $\la(e_i,f_i)=d_i$ for $1 \leq i \leq k$ and $\la$ vanishes otherwise between standard basis elements, and define 
$\del$ via $\del e_{i}=\delta_i$, $\del f_i=0$ for $1 \leq i \leq k$ and $\del g_j=\delta_{k+1}$ for $1 \leq j \leq l$. 

One verifies that this form realises the given sequence of invariants using that $M_i$ has basis 
\[\Big\{\frac{d_{i}}{d_1}e_1,\frac{d_{i}}{d_1}f_1,\frac{d_{i}}{d_2}e_2,\frac{d_{i}}{d_2}f_2,\dots,e_{i},f_{i},e_{i+1},f_{i+1},\dots,e_k,f_k,g_1,\dots,g_l\Big\},\]
$\operatorname{Rad}(M,\la)$ has basis
$\{g_1,\dots,g_l\},$
and that
$\delta_i| \frac{d_{i}}{d_{j}} \delta_j$ for $1 \leq j \leq i$ and $\delta_i|\delta_t$ for $t \geq i$.

\medskip

Given $(M,\la,\del) \in \Fd$ with $M\neq 0$,  
we will now construct an isomorphism to the standard formed space associated to its form data. 

\smallskip

By Lemma \ref{lem: pid}(ii), the formed space $(M,\la)$ is isomorphic to the canonical formed space 
\[\bigoplus_{i=1}^{k} d_i\!\HH \oplus (R^l,0)\]
determined by the first part of its invariants. 
Picking such an isomorphism allows us to assume that $M=R^n$ with standard basis $e_1,f_1,\dots,e_k,f_k,g_1,\dots,g_l$ such that $\la$ agrees with the above canonical form in this basis.
It remains to show that we can apply an automorphism of $(M,\la)$ to ensure that also $\del$ is canonical. 
We will proceed in several steps, each of which preserves $\la$ but modifies $\del$ to make it closer to the standard one.  

\smallskip

\noindent
\textbf{Step 1}: This step will make $\del$ standard on $\Rad(M,\la)=0 \oplus R^l$. \\
If $l=0$ there is nothing to do, so suppose $l \geq 1$. We need to find an element in $\GL(0 \oplus R^l) \subset \Aut(M,\la)$ so that $\del(g_j)=\delta_{k+1}$ for $1 \leq j \leq l$. 
This is immediate if $\delta_{k+1}=0$ since in that case $\del$ is the zero map on the radical. If not, $\delta_{k+1}$ generates the ideal $\del(\Rad(M,\la))$ and $\Rad(M,\la)=0 \oplus R^l$, so there is a non-zero vector $v \in 0 \oplus R^l$ such that $\del v=\delta_{k+1}$. 
Using Lemma~\ref{lem: pid}(ii), we can write $v= r v'$ with $r \in R$ and $v'$ unimodular. We must have $v'\in \Rad(M,\la)$ since $r\neq 0$. 
Hence $\del v' \in (\delta_{k+1})$, giving that $r$ is a unit so $v$ is unimodular. 
Now $\GL(0 \oplus R^l)$ acts transitively on the set of unimodular vectors by Lemma \ref{lem: pid}(i), so we can apply an automorphism to ensure that $\del g_1=\delta_{k+1}$. 
Finally, apply the automorphism given by $g_i \mapsto g_i +(1-\frac{\del g_i}{\delta_{k+1}}) g_1$ for each $1<i\le l$ to ensure that $\del$ is standard on $0 \oplus R^l$. 

\smallskip

In the remaining steps we will not modify $g_1,\dots,g_l$; we will only modify the remaining basis elements.

\smallskip

\noindent
\textbf{Step 2}: This step makes $\del f_i=0$ using an automorphism of $d_i\HH$ for each $1 \le i \le k$. \\
By Proposition \ref{prop:properties hyperbolics}(iii) with $l=\del$, there is an element $\phi\in \Aut(\langle e_i,f_i,\rangle,\la|_{\langle e_i,f_i,\rangle}) = \Aut(d_i\!\HH) \cong \Aut(\HH)$ that takes $(e_i,f_i)$ to a new basis $(\phi(e_i),\phi(f_i))$ with $\del \phi(f_i)=0$. 

\smallskip

\noindent
\textbf{Step 3}: This step makes $\del e_i| \delta_{k+1}$ for $1 \le i \le k$ using an automorphism of $d_i\!\HH\oplus (R^l,0)$ that preserves the condition  $\del f_i=0$ of the previous step and fixes $R^l$. \\
If $l=0$ there is nothing to do, so suppose $l \geq 1$. We will proceed in two steps. Firstly apply the automorphism $\psi\in \Aut(d_i\!\HH\oplus (R^l,0))$ defined by $\psi(e_i)=e_i$ and $\psi(f_i)=f_i+g_1$, so that the new basis has the property that $\del \psi(e_i)=\del e_i$ and $\del \psi(f_i)= \delta_{k+1}$. 
Then we repeat step 2 to get new bases $\Tilde{e_i}, \Tilde{f_i}$ such that $\del \Tilde{e_i}$ generates the ideal $(\del e_i, \delta_{k+1})$, and in particular $\del \Tilde{e_i} | \delta_{k+1}$, and $\del \Tilde{f_i}=0$. 

\smallskip
To finish the proof we need to further change the hyperbolic bases $(e_i, f_i)$ so that $\del e_i=\delta_i$ (keeping $\del f_i=0$). We will do so by decreasing induction on $i$ in Step 5 repeatedly using the following automorphism: 


\smallskip

\noindent
\textbf{Step 4}: Construction of an automorphism of $d_i\!\HH\oplus d_j\!\HH$ for $i<j$ ensuring that $\del e_i|\del e_j|\frac{d_j}{d_i}\del e_i$, without affecting the above properties. \\
Fix $i<j$ and suppose that $\del f_i=0=\del f_j$. 
We first apply the automorphism $\zeta\in \Aut(d_i\!\HH\oplus d_j\!\HH)$ defined by 
\begin{align*}
    \zeta(e_i)=e_i, \ \
    \zeta(f_i)=f_i+e_j, \ \
    \zeta(e_j)=e_j, \ \
  \zeta(f_j)=f_j+\frac{d_j}{d_i}e_i,
\end{align*}
which gives a new basis with the property that now  
$\del \zeta(f_i)=\del e_j$  
and $\del \zeta(f_j)=\frac{d_j}{d_i} \del e_i$.
Applying step 2 to the $i$th and $j$th summand then gives a new basis $(e_i',f_i',e_j',f_j')$ such that $\del e_i'$ generates the ideal $(\del e_i,\del e_j)$ and $\del e_j'$ generates the ideal $(\del e_j, \frac{d_j}{d_i} \del e_i)$, while $\del f_i'=0$, $\del f_j'=0$.  This gives the desired congruences since $\del e_i'$ then divides both $\del e_j$ and $\frac{d_j}{d_i}\del e_i$, so that $\del e_i'|\del e_j'$ and $\del e_j'$ divides $\frac{d_j}{d_i}\del e_i$ and $\frac{d_j}{d_i}\del e_j$, and hence also $\frac{d_j}{d_i}\del e_i'$.  

Observe that this special transformation changes the spans of both $e_i, f_i$ and $e_j, f_j$ but, as it is an automorphism of $(M,\la)$, the submodules $M_1, \dots, M_k, \Rad(M,\la)$ are unchanged. 

\smallskip

\noindent
\textbf{Step 5}: This step inductively uses step 4 to ensure that $\del e_m=\delta_m$ for each $m=1,\dots,k$. \\
Assume that for a fixed $1 \le m \le k$ we already have that $\del e_j= \delta_j$ for $m+1 \le j \le k$. Then we will modify $e_m, f_m$ so that $\del e_m= \delta_m$ (and $\del f_m=0$).
 Here the start of the induction, the case $m=k$, where the assumption holds trivially. 

We make the induction step in two parts. 

\begin{enumerate}[(i)]
    \item This step applies an automorphism of $d_m\!\HH\oplus d_{m+1}\!\HH$ to ensure that $\del e_m|\delta_{m+1}|\dots|\delta_{k+1}$. \\
   If $m=k$, this holds by step 3 that already ensured that $\del e_k | \delta_{k+1}$, so there is nothing to do. 
If $m<k$, apply the step 4 transformation with $i=m<m+1=j$. 
    This changes $(e_m,f_m,e_{m+1},f_{m+1}) \mapsto (e_m',f_m',e_{m+1}',f_{m+1}')$. 
    We firstly claim that we sill have $\del e_{m+1}'= \delta_{m+1}$, up to a unit: this is because $\del e_{m+1}'$ generates the ideal $(\del e_{m+1}, \frac{d_{m+1}}{d_m} \del e_m)$ and $\del e_{m+1}=\delta_{m+1}$ by induction hypothesis and $\frac{d_{m+1}}{d_m} e_m \in M_{m+1}$ by definition so $\delta_{m+1} | \frac{d_{m+1}}{d_m} \del e_m$.
    Secondly, $\del e_m'$ generates the ideal $(\del e_m, \delta_{m+1})$ so in particular $\del e_m'| \delta_{m+1}$. Since $\delta_{m+1}| \dots | \delta_{k+1}$ by construction then we are done. 

    \item This step applies an automorphism of $d_1\!\HH\oplus\dots\oplus d_m\!\HH$ to ensure that 
    $\del e_m|\frac{d_m}{d_i} \del e_i$ for each $1 \le i <m$. \\
Apply Step 4 inductively for each pair $(i,j)$ with $1\le i<m=j$.  This replaces $(e_i,f_i,e_m,f_m)$ with $(e_i',f_i',e_m',f_m')$ satisfying in particular that $\del e_m'|\frac{d_m}{d_i} \del e_i'$. Note that the modified $\del e_m'$ generates the ideal $(\del e_m,\frac{d_m}{d_i}\del e_i)$. In particular it divides $\del e_m$ and so the congruences obtained in part (i) still hold, and likewise for the congruences $\del e_m|\frac{d_m}{d_i'} \del e_{i'}$ for $i'<i$, for the same reason.   
   
\end{enumerate}
We are left to check that the above procedure does indeed force the equality $\del e_m= \delta_m$, finishing the induction step.  
By definition $\delta_m$ generates the ideal $\del(M_m)$ and $M_m$ has a basis given by 
\[\Big\{\frac{d_{m}}{d_1}e_1,\frac{d_{m}}{d_1}f_1,\frac{d_{m}}{d_2}e_2,\frac{d_{m}}{d_2}f_2,\dots,e_{m},f_{m},e_{m+1},f_{m+1},\dots,e_k,f_k,g_1,\dots,g_l\Big\}.\]
By the above steps we have that $\del e_m | \frac{d_m}{d_i} \del e_i$ for $1 \le i \le m-1$ and $\del e_m | \del e_{m+1} | \cdots | \del g_1 = \dots = \del g_l$, giving (up to a unit) the equality $\del e_m= \delta_m$. This finishes the proof.

\end{proof}

\begin{proof}[Proof of Theorem~\ref{thm: cancellation}]
By the surjectivity assumption on $\del_i$,  
we have $(M_i,\la_i,\del_i) \cong (M'_i,\la'_i,\del'_i) \# X$. 
Adding $\# X^{\#2}$ to each side and using that $X^{\#3} \cong X \# \HH$ (Proposition \ref{prop:X^n}), we get isomorphisms
$$(M_1,\la_1,\del_1) \# \HH\cong (M_1,\la_1,\del_1) \# X^{\#2} \cong (M_2,\la_2,\del_2) \# X^{\#2} \cong (M_2,\la_2,\del_2) \# \HH.$$ 
 
Thus, it suffices to study of cancellation by $\HH$ instead. 
Now, the effect of $-\# \HH$ on the form data invariants is as follows: 
let 
$$\DD(M,\la,\del)=(n,l,d_1,\dots,d_k,\delta_1,\delta_2,\dots,\delta_{k+1})$$
be the original form data. Then, $$\DD((M,\la,\del)\#\HH)=(n+2,l,d'_1, d_2',\dots,d_{k+1}',\delta_1',\delta_2',\dots,\delta_{k+2}')$$
where $(d'_1,\delta'_1)=(1,\delta_1)$ and  $(d_i',\delta'_i)=(d_{i-1},\delta_{i-1})$ for $2 \leq i \leq k+1$, because $\del$ vanishes on $\HH$. 
In particular, the form data invariants of $(M,\la,\del)$ are determined by those of $(M,\la,\del) \# \HH$, which implies cancellation.   
\end{proof}